\documentclass{article} 
\usepackage[dvips]{color,graphicx}% Required for inserting images
\usepackage{amsmath, amsthm, amsfonts, amssymb}
\usepackage[hyperindex=true,breaklinks=true,colorlinks=true,linkcolor=black]{hyperref}
\usepackage[all]{xy}
\usepackage[T1]{fontenc}
\usepackage{calligra}
\usepackage{comment}
\usepackage{enumerate}
\usepackage{booktabs}
\usepackage{float}
\usepackage{cite}
\usepackage{caption}
\usepackage{subfigure}
\usepackage{relsize}
\usepackage{overpic}
\usepackage{array}
\usepackage[margin=2.5cm]{geometry}
\usepackage{mathrsfs}
\usepackage{xcolor}
\usepackage{color}
\usepackage{dirtytalk}
\usepackage{textcomp}
\usepackage{fancyhdr}
\vbadness=\maxdimen
\newcommand\bR{\mathbf{R}}
\newcommand\R{\mathbb{R}}
\newcommand\vdiv{\mathop{\mathrm{div}}\nolimits}

\newcommand\bpi{\boldsymbol{\Pi}}

\newcommand\bx{\boldsymbol{x}}
\newcommand\bv{\boldsymbol{v}}

\newcommand{\bc}{\mathfrak{c}}

\newcommand{\bu}{\boldsymbol{u}}
\newcommand{\bn}{\boldsymbol{n}}

\newcommand{\bp}{\boldsymbol{p}}
\newcommand{\bt}{\boldsymbol{t}}
\newcommand{\bw}{\mathfrak{w}}
\newcommand{\vv}{\mathfrak{v}}
\newcommand{\zz}{\mathfrak{z}}

\newcommand{\bg}{\boldsymbol{g}}
\renewcommand{\P}{\mathbb{P}}

\newcommand{\corr}[1]{{#1}}

\DeclareMathOperator{\rot}{rot}
\renewcommand{\div}{\text{div}}

\newtheorem{assu}{Assumption}[section]
\newtheorem{thm}{Theorem}[section]
\newtheorem{rem}{Remark}[section]
\newtheorem{lemm}[thm]{Lemma}

\newtheorem{prop}{Proposition}[section]

\title{Discontinuous in time Virtual Element method for Darcy equations coupled with Multi Species Transport with First Order Reaction Network }
\author{Ruben Caraballo and Franco Dassi\thanks{Department of Mathematics and Applications, University of Milano-Bicocca, Via Cozzi 55, 20125 Milan, Italy (\href{mailto:rubenantonio.caraballodiaz@unimib.it}{rubenantonio.caraballodiaz@unimib.it}, \href{mailto:franco.dassi@unimib.it}{franco.dassi@unimib.it}}}
\date{}

\begin{document}
\maketitle
\begin{abstract}
We study transport phenomena involving chemically reactive species, 
modeled by \corr{advection--diffusion--reaction} systems coupled with flow fields governed by Darcy’s law. 
Both the velocity field and the species concentrations are discretized using the Virtual Element Method,
while time integration is performed through a discontinuous Galerkin scheme. 
This work represents a preliminary study, 
in which we introduce some simplifications of the full model. 
In particular, we assume a concentration-independent viscosity in the Darcy problem, 
constant diffusion tensors in the \corr{advection–-diffusion-–reaction} systems, 
and first-order reaction networks with liquid-phase degradation. 
We derive an abstract error estimate by means of a technique that combines \corr{Gauss--Radau} interpolation with numerical integration.
The theoretical results are supported by numerical experiments that exhibit arbitrary-order accuracy in both space and time.
\end{abstract}

\section{Introduction}

The simulation of transport phenomena involving chemically reactive species 
plays a key role in many environmental and engineering applications, 
such as groundwater contamination, and the modeling of water and air quality.
These processes are governed by advection–diffusion–reaction systems, 
where the advective term is associated with a flow model satisfying a continuity equation, i.e., the conservation of mass.

A wide variety of numerical methods have been proposed to solve \corr{these kinds of problems;
see, e.g.,~\cite{Miloslav2012,Gong2025Positivity} and~\cite{Dawson2004}} for a comprehensive review.
In general, these approaches can be divided into two main families. 
The first ones are the \emph{Streamline Diffusion} (SD) methods,
which rely on a continuous Galerkin formulation~\cite{Brooks:1982:SUP}. 
The second ones include \emph{Discontinuous Galerkin} (DG) methods~\cite{Cockburn:1999:SEO,Houston:2002:DHP}.

When the \emph{exact} velocity field is used, 
both SD and DG methods are known to be stable, accurate, and globally conservative. 
However, when an \emph{approximate} velocity field is \corr{employed,} these schemes may lose two essential properties: 
zeroth-order accuracy and/or global conservation.

The first property refers to the ability to preserve constant solutions. 
In an \corr{advection--diffusion--reaction} problem 
with constant initial, boundary, and source data, 
the exact solution remains constant in time.
A method which is zeroth-order accurate is able to numerically reproduce this behavior.
If it is not, spurious sources or sinks may appear.  
The second property, global conservation, ensures 
that the total mass varies only due to fluxes across the domain boundary.

Interestingly, SD and DG schemes exhibit complementary behavior in this regard. 
SD schemes preserve zeroth-order accuracy, 
but ensure global conservation only under restrictive conditions on the velocity field~\cite{Dawson2004}.
Conversely, DG schemes guarantee global conservation regardless of the velocity approximation, 
but may fail to preserve constant solutions unless additional compatibility conditions are imposed.

In this work, 
we consider the coupling between an \corr{advection–-diffusion--reaction} system and a velocity field obtained from a Darcy problem. 
The Darcy equation is discretized by the \emph{Mixed Virtual Element Method}~\cite{Brezzi:2014:BPM,beirao:2016:MVE}, 
while the \corr{time--dependent} \corr{advection--diffusion--reaction} equation is solved using 
the \emph{Virtual Element Method} (VEM) for spatial discretization~\cite{Brezzi:2014:BPM,Beirao:2013:BPV} 
combined with a \emph{Discontinuous Galerkin} time-stepping scheme~\cite{Ern:2016:DGM}, 
following the approach proposed in~\cite{Beirao:2025:SST}. 
For the time discretization, we adopt a Gauss–Radau quadrature formula with respect to the temporal variable~\cite{Ern:2016:DGM}. 
This choice leads to high--order, fully discrete time--integration schemes well suited for adaptive temporal refinement. 
As a consequence, while the proposed continuous VEM formulation is exact for zeroth-order polynomial solutions, 
unlike DG schemes it does not inherently guarantee 
the global mass conservation of the scalar concentrations.
The main contributions of this paper are:
\begin{itemize}
  \item the proof of existence and uniqueness of the proposed semi-discrete and fully discrete VE–DG schemes \corr{within this kind of problem;}
  \item the derivation of error estimates in suitable discrete norms;
  \item the extension of the analysis to the case of multiple species coupled through a first-order reaction network.
\end{itemize}

The remainder of the paper is organized as follows. 
In Section~\ref{sec:model}, we introduce the governing equations for both the Darcy and the transport problems,
along with the underlying model assumptions and the well-posedness of the continuous formulation. 
In Section~\ref{sec:vem}, we briefly describe the Virtual Element spaces employed for the discretization,
focusing on the projection operators and the properties essential for the subsequent analysis. 
Section~\ref{sec:disc} represents the core of the work;
here, we define the semi-discrete and \corr{fully discrete} formulations, 
establishing their well-posedness and deriving the corresponding error estimates. 
Finally, in Section~\ref{sec:numExe}, we present several numerical experiments to validate the proposed approach, 
including one applicative test case involving the interaction of two species. 

\paragraph{Notation.}
Throughout this article we adopt the standard notation for function spaces 
and operators~\cite{McLean:2000,adams2003sobolev,evans2010partial,showalter1997monotone,lions1969quelques}.
Specifically, we denote by~\(\partial_t\) the derivative with respect to time, 
and by~\(\nabla\) and~\(\Delta\) the spatial gradient and Laplace operator, respectively. 
Given an open and bounded domain~\(\mathcal{D} \subset \mathbb{R}^2\), $p\in [1,\infty]$ and $s\ge 0$, 
we denote by~\(L^p(\mathcal{D})\) the standard Lebesgue space with norm
\[
\|u\|_{0,p,\mathcal{D}} =
\begin{cases}
\left( \displaystyle\int_{\mathcal{D}} |u(\bx)|^p\,d\bx \right)^{1/p}, & 1 \le p < \infty, \\[0.6em]
\operatorname*{ess\,sup}_{\bx \in \mathcal{D}} |u(\bx)|, & p = \infty,
\end{cases}
\]
and by~\(W^{m,p}(\mathcal{D})\) the standard Sobolev space
\[
W^{m,p}(\mathcal{D}) = \{ u \in L^p(\mathcal{D}) : D^\alpha u \in L^p(\mathcal{D}),\ \text{ for all }|\alpha| \le m \},
\]
endowed with the usual seminorm and norm, 
denoted by~\(|\cdot|_{m,p,\mathcal{D}}\) and~\(\|\cdot\|_{m,p,\mathcal{D}}\), respectively.
The space~\(L^2(\mathcal{D})\) is a Hilbert space endowed with the usual inner product \((u,v)_{\mathcal{D}} \).
When~\(p=2\),
we use the shorter notation \(H^s(\mathcal{D}):=W^{s, 2}(\mathcal{D})\).
As before, its seminorm and norm are denoted by~\(|\cdot|_{s,\mathcal{D}}\) 
and~\(\|\cdot\|_{s,\mathcal{D}}\), respectively.
Let~\(H_0^1(\mathcal{D})\) denote the subspace of~\(H^1(\mathcal{D})\) 
consisting of functions with zero trace on~\(\partial\mathcal{D}\). 
We denote by $W^{-m,\frac{p}{p-1}}(\mathcal{D})$ \corr{the dual of the space} $W^{m,p}(\mathcal{D})$ with the usual convention $H^{-m}(\mathcal{D})$ for the case $p=2$.

\corr{Furthermore,} given~\(p\in [1,\infty]\), ~\(s \geq 1\), a Banach space~\((X, \|\cdot\|_X)\), 
and a time interval~\((a,b)\), 
the corresponding Lebesgue-Bochner space is defined by
\[
L^p(a,b;X) = \{ u : (a,b) \to X \text{ strongly measurable},\ \|u\|_{L^p(a,b;X)} < \infty \},
\]
with
\[
\|u\|_{L^p(a,b;X)} =
\begin{cases}
\left( \displaystyle\int_a^b \|u(t)\|_X^p\,dt \right)^{1/p}, & 1 \le p < \infty, \\[0.6em]
\operatorname*{ess\,sup}_{t \in (a,b)} \|u(t)\|_X, & p = \infty,
\end{cases}
\]
and by 
\[
W^{s,p}(a,b;X) = \{ u \in L^p(a,b;X) : \partial_t^j u \in L^p(a,b;X),\text{ for all }1 \le j \le s \},
\]
equipped with the norm
\[
\|u\|_{W^{s,p}(a,b;X)} =
\left( \sum_{j=0}^{k} \| \partial_t^j u \|_{L^p(a,b;X)}^p \right)^{1/p}.
\]
When \(p=2\), we use the Hilbertian notation~\(H^s(a,b;X) := W^{s,2}(a,b;X)\).

%\corr{Put a sentence for scalar, vectors and multispecies refer also to the $[H(\Omega)]^{n_c}$!!!}

\corr{
To facilitate the reader's navigation through the complex structure of the problem, we adopt a consistent typographical convention for the different types of variables involved. Throughout this work, standard italic letters (e.g., $u, w$) denote scalar-valued functions defined in $\mathbb{R}$. Boldface letters (e.g., $\mathbf{u}, \mathbf{n}$) are reserved for vector fields in the physical space $\mathbb{R}^d$. Finally, to distinguish the multi--component nature of certain variables, specifically the presence of $n_c$ species, Fraktur symbols (e.g., $\mathfrak{v}, \mathfrak{w}$) are employed for vectors belonging to the generalized space $\mathbb{R}^{n_c}$. This distinction is systematically maintained to ensure clarity across the different mathematical dimensions of the model.
}

\corr{
Regarding functional spaces, we extend the previously defined notations to the multi--component framework. 
For instance, the space of functions whose components are in $H^1(\Omega)$ for each of the $n_c$ components is denoted by $[H^1(\Omega)]^{n_c}$. 
Similarly, an element $\mathfrak{v} \in [H^1(\Omega)]^{n_c}$ represents the vector $(v_1, \dots, v_{n_c})$, 
where each $v_i \in H^1(\Omega)$. 
}

\corr{Finally, we use the notation \(a\lesssim b\)
to indicate the existence of a positive constant \(C\), 
independent of the mesh size, the element diameters, and the edge lengths,
such that \(a \leq C\,b\). 
Furthermore, we use the notation \(a\approx b\) to indicate that \(a\lesssim b\) and \(b\lesssim a\).
}

%%%%%%%%%%%%%%%%%%%%%%%%%%%% english checked with a deep analysis with specific instrument

\section{Model problem}\label{sec:model}

In this section, we describe the multi-species model problem considered in this article.
In Sections~\ref{sec:darcyMod} and~\ref{sec:tran}, we introduce the strong formulations of the Darcy and transport equations,
respectively.
Then, in Section~\ref{sec:simp}, we analyze in more detail how these two equations are coupled in this class of problems
and describe the simplifications adopted in this work.
Based on these assumptions, we define the weak formulation of the problem in Section~\ref{sec:weakForm}.
Finally, \corr{Section~\ref{sec:teo} is devoted} to establishing the theoretical framework
and to proving the existence and uniqueness of the solution.

\subsection{Darcy equations}\label{sec:darcyMod}

We consider the classical Darcy equation in mixed form that describes the flow of a fluid through a porous medium. 
In this model, the velocity field, $\bu:\Omega \rightarrow \mathbb{R}^2$, 
satisfies the conservation of mass equation of the form
\begin{equation}\label{continuityeq}
\div(\bu) =f,
\end{equation}
where $\bu$ is the velocity field, 
and $f$ is \corr{an external source or sink term}.
Specifically, sources are characterized by positive values of~\(f\),
while sinks correspond to negative ones. 

Darcy's law for the flow of a viscous fluid in a permeable medium is expressed as follows:
\begin{equation}\label{darcyslaw}
\bu=\dfrac{K}{\mu (c_1,...,c_{n_c})}\nabla p, \text{ in }\Omega,
\end{equation}
where $p:\Omega \rightarrow \mathbb{R}$ denotes the pressure,~$K$ is the permeability coefficient, 
and~$\mu$ is the viscosity of the fluid which generally depends on the species concentrations.
The boundary \corr{of~\(\Omega\), denoted by~\(\Gamma\),} is partitioned into two parts $\overline{\Gamma_N}\cup\overline{\Gamma_D}$  
with $\Gamma_N\cap \Gamma_D=\emptyset$.
We impose the following boundary conditions:
\begin{align}\label{presureboundary}
p&=g_D, \quad\text{on }\Gamma_D,\\
\label{velocityboundary}
\bu\cdot \bn&=g_N, \quad \text{on } \Gamma_N, 
\end{align}
here $\bn$ denotes the outward pointing unit normal to $\Gamma$, 
$g_D\in \mathrm{H}_{00}^{1/2}(\Gamma_D)$ and $g_N\in \mathrm{H}_{00}^{-1/2}(\Gamma_N)$ \corr{are the Dirichlet} and Neumann boundary data, respectively.

\subsection{Transport equations}\label{sec:tran}

Given a set of~\(n_c\) species,
whose initial distribution in the domain~\(\Omega\) is a known function~\(c_i^0\) for~\(i=1,\dots,n_c\),
we consider the following transport equations for each species
\begin{equation}
\partial_t {c}_i+\div(\bu\,{c_i}-D_i(\bu)\nabla c_i)=f c_i^{\ast}+R_i(c_1,\dots,c_{n_c})\quad\text{for }i=1,\dots,n_c.
\label{Mtransport}
\end{equation}
Here~$c_i$ denotes the concentration of the \(i\)-th species,
and the symmetric positive semi-definite tensor $D_i(\bu)$ 
represents the diffusion-dispersion of the \(i\)-th species which,
in general, depends on the Darcy flow~\(\bu\).
The function~\(c_i^\ast\) represents the prescribed concentration of the \(i\)-th species at sources or sinks.
Specifically, it is defined according to the sign of the Darcy force term~\(f\):
for sources~\(f>0\),~\(c_i^\ast\) is a given function~\(\tilde{c}_i\),
whereas for sinks~\(f<0\) it coincides with the concentration, i.e., \(c_i^\ast=c_i\).
Finally,~\(R_i\) is the reaction term which may depend on \emph{all} species, including the~\(i\)-th one. 

The boundary~\(\Gamma\) for the transport system is split into two disjoint parts: 
\begin{equation}\label{binoutflow}
\Gamma_{I}:=\{ \bx\in \Gamma:\, \bu\cdot\bn<0\}\quad\text{and}\quad\Gamma_{O}:=\{ \bx\in \Gamma:\, \bu\cdot\bn\geq 0\},
\end{equation}
representing the inflow and the outflow boundary, respectively. 
On these boundaries, we impose the following boundary conditions 
\begin{subequations}
\begin{equation}\label{bcinflow}
(c_i\bu -D_i(\bu)\nabla c_i)\cdot \bn =c_i^I\bu\cdot \bn, \qquad \text{on }\Gamma_{I}\times (0,T],
\end{equation}
\begin{equation}\label{bcoutflow}
\qquad  \qquad D_i(\bu)\,\nabla c_i \cdot \bn=0, \, \qquad \qquad \text{on }\Gamma_{O}\times (0,T].
\end{equation}
\end{subequations}
Here $c_i^I$ is the inflow concentration of the~\(i\)-th species which may depend on both space and time.  

\subsection{Simplified model and assumptions}\label{sec:simp}

This class of problems presents several challenges due to the nonlinear relations, the time dependency and 
strong coupling of these two problems.
Given an initial distribution of species concentrations, 
since the porosity~\(\mu\) generally depends on the species distributions, 
one can compute the Darcy flow at the initial time.
Then, using this velocity field~\(\bu\),
the new concentration distribution can be obtained by solving the nonlinear time-dependent problem~\eqref{Mtransport}.
However, since the species concentrations evolve in time,
the Darcy problem must be solved again to update the velocity field, 
and then the transport problem must be resolved accordingly.
In summary,
to obtain the concentration distribution over time,
one must repeatedly alternate between solving the Darcy problem and the nonlinear time-dependent transport problem until the end of the simulation.

\begin{rem}
A possible strategy would be a monolithic approach, 
that is solving a single coupled system including the Darcy and the transport equations for all species.
However, such an approach is \corr{computationally expensive}, 
since the linear system becomes extremely large and,
more importantly, involves two nonlinearities \(\mu (c_1,...,c_{n_c})\) and \(R_i(c_1,\dots,c_{n_c})\),
which requires a nonlinear iterative scheme, such as fixed-point iterations at \emph{each} time step.
\end{rem}

Dealing with such system, with all these difficulties, 
is extremely demanding both from a theoretical and the \corr{practical viewpoint}. 
For this reason, in this preliminary work we introduce \corr{a series of simplifications}
in order to establish a foundation for future studies,
where some of these complexities will be progressively reintroduced.

\corr{Specifically, the model is analyzed under the assumptions of a multi-species transport system with a first-order reaction network, the particularities of which, relative to the general model, are grounded in references~\cite{Marian2008,Jones2021,SUK2016,Chen2012,QUEZADA2004507}. In accordance with the cited literature, the assumptions governing the proposed model are as follows:}
%Specifically, in this paper we study the model under the following assumptions:
\begin{enumerate}[(i)]
\item more than one species may be present, i.e.,~\(n_c>1\);
\item the coefficient~\(\mu\) of the Darcy flow is constant throughout the domain~\(\Omega\) and is not species dependent;
\item the diffusion tensors \(D_i(\bu)\) are constant and isotropic, i.e.,~they can be reduced to scalar coefficients~\(D_i\);
\item\label{ass:forn} \corr{the transport of the $n_c$ species undergoes} a first-order reaction network with liquid-phase degradation; 
in other words, the reaction terms take the form
\begin{equation}
R_i(c_1,\dots,c_{n_c}) = -\gamma_i c_i + \sum_{\substack{j=1\\ j\neq i}}^{n_c} y_{i/j}\,\gamma_j\,c_j,
\label{reactionterm}
\end{equation}
where $\gamma_j$ denotes the first-order degradation (or decay) rate constant of the~\(j\)-th species, 
and  $y_{i/j}$ is the effective yield coefficient 
representing the mass of~\(i\)-th species produced from the degradation of~\(j\)-th species.
\end{enumerate}

\begin{rem}
A direct consequence of assumption (\ref{ass:forn}) is 
that the reaction terms \(R_i(c_1,\dots,c_{n_c})\)  can be expressed as a matrix–vector product, 
where the reaction matrix~\(\R\) has entries
\begin{equation}
    R_{ij} = 
    \begin{cases}
        -\gamma_i, & \text{if } i = j, \\[6pt]
        y_{i/j}\,\gamma_j, & \text{if } i \neq j.
    \end{cases}
\end{equation}
\label{eqn:forn}
\end{rem}

\subsection{Weak formulations}\label{sec:weakForm}

In this section, we introduce the weak formulations of both the Darcy and the transport equations
(see Sections~\ref{sec:darcy} and~\ref{sec:tran}).
Let us consider the Darcy problem. 
First, we define the following functional spaces 
\begin{eqnarray*}
\mathbf{H}(\div,\Omega)&:=&\left\{\bv \in [L^2(\Omega)]^2:\, \div\,\bv \in L^2(\Omega)\right\},\\[0.3em]
\mathbf{H}_{N,g_N}(\div,\Omega)&:=&\left\{\bv \in \mathbf{H}(\div,\Omega):\, \bv\cdot \bn =g_N \, \text{on} \, \Gamma_N \right\}, 
\end{eqnarray*}
and
\[
\mathbf{H}_{N,0}(\div,\Omega):=\left\{\bv \in \mathbf{H}(\div,\Omega):\, \bv\cdot \bn =0 \, \text{on} \, \Gamma_D \right\}. 
\]
We equip these spaces with the usual norm defined as 
\[
\Vert \bv \Vert_{\vdiv}^2:=\Vert \bv \Vert_{0,\Omega}^2+\Vert \div\, \bv \Vert_{0,\Omega}^2.
\]
Starting from these functional spaces,
to get the variational formulation of the Darcy problem,
we proceed as usual.
We multiply equations~\eqref{continuityeq} and~\eqref{darcyslaw} 
by suitable test functions.
Then, integrating by parts over~\(\Omega\),
we get the variational form of Darcy's problem:
find $(\bu,p)\in \mathbf{H}_{N,g_N}(\div,\Omega)\times \mathrm{L}^2(\Omega)$, such that 
\begin{align}
\mathcal{M}(\bu,\bv)+ \mathcal{N}(\bv,p)&=\mathcal{G}_N(\bv), 
&\forall\bv\in \mathbf{H}_{N,0}(\div,\Omega), \label{darcyequ}\\
\mathcal{N}(\bu,q)&=\mathcal{G}_D(q),
&\forall q\in  \mathrm{L}^2(\Omega),\label{darcyeqp}
\end{align}
where we have defined the following bilinear and linear forms
\begin{align*}
\mathcal{M}(\bu,\bv)&:=\mu K^{-1}\int_{\Omega} \bu\cdot \bv\, d\bx, \\
\mathcal{N}(\bv,q)&:=\int_{\Omega}q\,\div \bv\, d\bx,\\
\mathcal{G}_N(\bv)&:=\int_{\Gamma_N}g_N(\bv \cdot \bn)\, dS\\ \mathcal{G}_D(q)&:=\int_{\Omega} f\,q\, d\bx.
\end{align*}

We move now to the transport equations.
In this case, we adopt as the functional space for the solution the Hilbert space~\(\mathrm{H}^1(0, T;\mathrm{H}^1(\,\Omega))\).
We now turn our attention to the transport equations. For simplicity, let us focus on the \(i\)-th species.
Multiplying \eqref{Mtransport} by a test function \(w_i\in \mathrm{H}^1(\Omega)\) and integrating by parts over \(\Omega\), 
we obtain

\begin{equation}\label{eq1}
\int_{\Omega} \partial_t \,c_i\, w_i \,d\bx-\int_{\Omega}(\bu c_i-D_i\nabla c_i )\cdot \nabla w_i \, d\bx+\int_{\Gamma}(\bu c_i-D_i\nabla c_i )\cdot\bn\, w_i\, dS
=\int_{\Omega}(f c_i^{\ast}+R_i(\bc))w_i\,d\bx
\end{equation}

Here, we introduce the vector~\(\bc\) which collects all the species 
so that we have a more compact notation of the reaction term~\(R_i(\bc)\) instead 
of~\(R_i(c_1,\dots,c_{n_c})\).

Starting from this variational formulation,
we manipulate these integrals in such a way 
that both boundary conditions and the source term of the Darcy problem appear.
To do this, let~\(\xi,\eta\in \mathrm{H}^1(\Omega)\) 
and~\(\bg\in \mathbf{H}(\div,\Omega)\cap \mathrm{\boldsymbol{L}}^4(\Omega)\cap \gamma_{\bn}^{-1}(\mathrm{L}^2(\Gamma))\),
\corr{exploiting by parts integration},
the following identity holds 
\begin{equation}
\int_{\Omega}\div(\xi\,\bg)\,\eta\,d\bx=-\int_{\Omega}\xi\,(\bg\cdot \nabla \eta)\,d\bx+\int_{\Gamma}\xi\,\eta \,(\bg\cdot \bn)\, dS\,.
\label{eqn:byPart1}
\end{equation}
Then, if we compute the divergence of the left-hand side,
Equation~\eqref{eqn:byPart1} becomes  
\[
\int_{\Omega}(\xi\, \div\,\bg+\bg\cdot \nabla \xi)\,\eta\,d\bx=-\int_{\Omega}\xi\,(\bg\cdot \nabla \eta)\,d\bx+\int_{\Gamma}\xi\,\eta \,(\bg\cdot \bn)\, dS,
\]
\corr{that gives the following identity;} 
\begin{equation}
\int_{\Omega}(\bg\cdot \nabla \xi)\,\eta\,d\bx=\int_{\Omega}\xi\,\eta\, \div\,\bg\,\, d\bx-\int_{\Omega}\xi\,(\bg\cdot \nabla \eta)\,d\bx+\int_{\Gamma}\xi\,\eta \,(\bg\cdot \bn)\, dS.
\end{equation}
Assuming that $\bu$ has sufficient regularity, and setting \(\bg=\bu\), \(\eta=c_i\) and \(\xi=w_i\) in the previous identity, we obtain
\begin{align}\nonumber
\int_{\Omega}c_i\,(\bu\cdot \nabla w_i)\,d\bx&=\frac{1}{2}\left( \int_{\Omega}c_i\,(\bu\cdot \nabla w_i)\,d\bx+\int_{\Omega}c_i\,(\bu\cdot \nabla w_i)\,d\bx \right)\\
&=\frac{1}{2}\left( \int_{\Omega}c_i\,(\bu\cdot \nabla w_i)\,d\bx-\int_{\Omega}w_i\,(\bu\cdot \nabla c_i)\,d\bx-\int_{\Omega}w_i\,c_i\, \div\,\bu\,d\bx+\int_{\Gamma}w_i\,c_i\,(\bu\cdot \bn)\, dS \right). \label{identityf}
\end{align}

Recalling that~\(\div\,\bu=f\),
since~\(\bu\) satisfies~\eqref{continuityeq},
and substituting Equation~\eqref{identityf} into~\eqref{eq1}, 
we obtain
\begin{align*}
   \int_{\Omega}(f c_i^{\ast}+R_i(\bc))w_i\,d\bx&= \int_{\Omega} \partial_t \,c_i\, w_i \,d\bx-\int_{\Omega}(\bu c_i-D_i\nabla c_i )\cdot \nabla w_i \, d\bx+\int_{\Gamma}(\bu c_i-D_i\nabla c_i )\cdot\bn\, w_i\, dS\\   
   &=  \int_{\Omega} (\partial_t \,c_i\, w_i+D_i\nabla c_i \cdot \nabla w_i) \, d\bx+\int_{\Gamma}(\bu c_i-D_i\nabla c_i )\cdot\bn\, w_i\, dS - 
   \int_\Omega c_i(\bu\cdot \nabla w_i)\,d\bx\\
   &=  \int_{\Omega} (\partial_t \,c_i\, w_i+D_i\nabla c_i \cdot \nabla w_i) \, d\bx+\int_{\Gamma}(\bu c_i-D_i\nabla c_i )\cdot\bn\,w_i\, dS\\
   &\qquad -\frac{1}{2}\left( \int_{\Omega}c_i\,(\bu\cdot \nabla w_i)\,d\bx-\int_{\Omega}w_i\,(\bu\cdot \nabla c_i)\,d\bx-\int_{\Omega}f\,w_i\,c_i\,d\bx+\int_{\Gamma}w_i\,c_i \,(\bu\cdot \bn)\, dS \right)\\
   &= \mathcal{A}_i(c_i,\, w_i) + \underbrace{\int_{\Gamma}(\bu c_i-D_i\nabla c_i )\cdot\bn\,w_i\, dS-\frac{1}{2}\int_{\Gamma}c_i \,w_i\,(\bu\cdot \bn)\,dS}_{(a)} + 
   \corr{\underbrace{\frac{1}{2}\int_{\Omega}f\,w_i\,c_i\,d\bx,}_{(b)}}\\
\end{align*}
where, to have a more readable expression,
we have defined the following bilinear form
\begin{equation}
\mathcal{A}_i(c_i,\, w_i):=
\int_{\Omega} (\partial_t \,c_i\, w_i+D_i\nabla c_i \cdot \nabla w_i) \, d\bx + \frac{1}{2}\int_{\Omega}(\bu\cdot \nabla c_i)\,w_i\,d\bx-\frac{1}{2}\int_{\Omega}c_i\,(\bu\cdot \nabla w_i)\,d\bx.
\label{eqn:aForm}
\end{equation}
\corr{In the weak formulation of the problem, the integrals \((a)\) and \((b)\) are modified by 
incorporating the boundary conditions and 
accounting for the behavior of the function~\(c_i^*\).
Regarding the term~\((a)\),
we proceed by applying the inflow and outflow boundary conditions~\eqref{bcinflow}--\eqref{bcoutflow}} 
\begin{align*}
(a)&=\int_{\Gamma_I}(\bu c_i-D_i\nabla c_i )\cdot\bn\,w_i\, dS
+\int_{\Gamma_O}(\bu c_i-D_i\nabla c_i )\cdot\bn\,w_i\, dS-\frac{1}{2}\int_{\Gamma}c_i \,(\bu\cdot \bn)\,w_i\,dS\\
&=\int_{\Gamma_I}\,c_i^{I}(\bu\cdot \bn)\,w_i\,dS+\int_{\Gamma_O}(\bu \cdot \bn)\,c_i\,w_i\,dS -\frac{1}{2}\int_{\Gamma_I}(\bu\cdot \bn)\,c_i\,w_i\, dS
-\frac{1}{2}\int_{\Gamma_O}(\bu\cdot \bn)\,c_i\,w_i\, dS\\
&=\int_{\Gamma_I}\,c_i^{I}(\bu\cdot \bn)\,w_i\,dS \underbrace{-\frac{1}{2}\int_{\Gamma_I}(\bu\cdot \bn)\,c_i\,w_i\, dS
+\frac{1}{2}\int_{\Gamma_O}(\bu\cdot \bn)\,c_i\,w_i\, dS}_{(c)}\\
\end{align*}
According to the sign of~\(\bu\cdot\bn\) on~\(\Gamma_I\) and~\(\Gamma_O\),
\corr{it is possible to rewrite (c) in a more compact form as
\[
 \frac{1}{2}\int_{\Gamma}|\bu\cdot \bn|\,c_i\,w_i\, dS.
\]
Then, since \(c_i^{I}\) is a data, we can move its integral to the right hand side and 
the variational formulation of the problem becomes}
\[
\mathcal{A}_i(c_i,\, w_i) + \mathcal{B}(\bu; c_i,\, w_i) + \frac{1}{2} \int_\Omega f\,w_i\,c_i\,d\bx = \int_{\Omega}(f c_i^{\ast}+R_i(\bc))w_i\,d\bx-\int_{\Gamma_I}\,c_i^{I}(\bu\cdot \bn)\,w_i\,dS,
\]
where we defined the following bilinear form 
\begin{equation}
\mathcal{B}(\bu;c_i,\, w_i) := \frac{1}{2}\int_{\Gamma}|\bu\cdot \bn|\,c_i\,w_i\, dS\,.    
\label{eqn:bForm}
\end{equation}

\corr{Finally, for the behavior of the function~\(c_i^*\),}
let~\(\Omega^+\) and~\(\Omega^-\) denote the subdomains where~\(f\) is positive or negative, respectively.
Recalling the definition of~\(c^\ast_i\), 
we split the integral at the right hand side as follows
\[
\mathcal{A}_i(c_i,\, w_i) + \mathcal{B}(\bu;c_i,\, w_i) + \frac{1}{2} \int_\Omega f\,w_i\,c_i\,d\bx = \int_{\Omega^+} f\,\tilde{c}_i\,w_i\,d\bx + \underbrace{\int_{\Omega^-} f\,c_i\,w_i\,d\bx}_{(d)} 
+ \int_\Omega R_i(\bc)\,w_i\,d\bx - \int_{\Gamma_I}\,c_i^{I}(\bu\cdot \bn)\,w_i\,dS.
\]
We move the integral~\((d)\) at the left hand side
since it depends on both the trial and test functions,~\(c_i\) and~\(w_i\),
and obtain the final variational formulation of the problem
\begin{equation}
\sum_{i=1}^{n_c}\left(\mathcal{A}_i(c_i,\, w_i) + \mathcal{B}(\bu;c_i,\, w_i) + \mathcal{C}_i(\bc,\, w_i) \right)
= \sum_{i=1}^{n_c}\mathcal{F}_i(w_i),\quad\forall w_i\in H^1(\Omega),
\label{eqn:varForm}
\end{equation}
where we have defined the bilinear and linear form 
\begin{align}
\mathcal{C}_i(\bc,\, w_i) &:= \frac{1}{2}\int_\Omega |f|\,c_i\,w_i\,d\bx-\int_\Omega R_i(\bc)\,w_i\,d\bx, \label{eqn:cForm}\\
\mathcal{F}_i(w_i) &:=\int_{\Omega^+} f\,\tilde{c}_i\,w_i\,d\bx -\int_{\Gamma_I}\,c_i^{I}(\bu\cdot \bn)\,w_i\,dS.\label{eqn:fForm}
\end{align}

\begin{comment}
Summing over $j=1,\dots,n_c$, we obtain that
\begin{gather} \nonumber
\int_{\Omega}(f\bc^{\ast}\cdot \bw+R(\bc)\bw)\,d\bx=\int_{\Omega} \partial_t \,\bc\cdot \bw\, d\bx+\int_{\Omega}N(\bc,\bw)\, d\bx+\int_{\Gamma_I}\,(\bu\cdot \bn)(\bc^{I}\cdot \bw)\,dS\\
+\frac{1}{2}\left( \int_{\Omega}(\bw\cdot\bt(\bc,\bu)-\bc\cdot \bt(\bw,\bu) )\,d\bx+\int_{\Omega}f\, \bc\cdot \bw\,d\bx+\int_{\Gamma} \,|\bu\cdot \bn|(\bc\cdot \bw)\, dS \right),\label{cvf}
\end{gather}
where $N(\bc,\bw):=\mathrm{Tr}( (\nabla \bc)^{\top} \mathbf{D} (\nabla \bw) )=\sum_{j=1}^{n_c}D_j \nabla \bc_j \cdot \nabla \bw_j$.

Hence, by the definition of $c_i^{\ast}$, we have
\begin{equation}\label{eq2}
    \int_{\Omega}f\bc^{\ast}\cdot \bw\,d\bx=\int_{\Omega^{+}}f\bc^{\ast}\cdot \bw\,d\bx+\int_{\Omega^{-}}f\bc^{\ast}\cdot \bw\,d\bx,
\end{equation}
where $\Omega^{\pm}:=\{\pm f >0\}\cap \Omega$.

Using \eqref{eq2} we rewrite \eqref{cvf} and obtain the variational formulation:

Find $\bc\in L^2((0,T);H^1(\Omega)^{n_c})\cap H^1((0,T);H^{-1}(\Omega)^{n_c})$, such that
\begin{equation}\label{varfcont}
(\partial_t \bc\cdot \bw)_{\Omega}+(N(\bc,\bw),1)_{\Omega}+K(\bu;\bc,\bw)=(f\tilde{\bc},\bw)_{\Omega^{+}}-((\bu \cdot \bn)\bc_{I}, \bw)_{\Gamma_{I}}+(\bR(\bc)\bw,1)_{\Omega},
\end{equation}
for all $\bv\in  \bc\in L^2((0,T);H^1(\Omega)^{n_c})$, where 
\[ 
K(\bu;\bc,\bw):=\frac{1}{2}\left( (\bw,\bt(\bc,\bu))_{\Omega}-(\bc,\bt(\bw,\bu) )_{\Omega}+(|f|\bc, \bw)_{\Omega}+(|\bu\cdot \bn|,\bc\cdot \bw)_{\Gamma}\right)
\]
\end{comment}

\subsection{Functional Analysis Framework and Well-posedness}\label{sec:teo}

In this section we introduce the functional framework needed for the analysis of
the Darcy equations coupled with multi-species transport involving a 
first-order reaction network.
We first recall the abstract setting for linear parabolic problems and the
inequalities that will be used throughout the manuscript.
We then apply these tools to establish the well-posedness of the Darcy problem.

\noindent\textit{Abstract setting.} 
We consider the Hilbert spaces $\mathscr{V}$ and $\mathscr{H}$, 
$\mathscr{V}\subseteq \mathscr{H}$, $\mathscr{V}$ dense in $\mathscr{H}$. 
We identify $\mathscr{H}$ with its dual space $\mathscr{H}'$. 
Let $a:\mathscr{V}\times \mathscr{V}\rightarrow \mathbb{R}$ be 
a continuous bilinear form satisfying the following coercivity condition: 
there exist $\alpha>0$ and $\mu \geq 0$ such that 
\[
a(w,w)+ \mu \Vert w \Vert_{\mathscr{H}}^2\geq \alpha \Vert w\Vert_{\mathscr{V}}^2\qquad\forall w\in \mathscr{V}.
\]

\noindent\textit{Variational formulation and abstract existence result.}
The variational formulation of the parabolic problem 
(where $(\cdot,\cdot)$ denotes the scalar product in $\mathscr{H}$) 
is given by the following.

\noindent Let $T>0$, $f:(0,T)\rightarrow \mathscr{V}'$, and $c_0\in \mathscr{H}$.
For almost every $t\in (0,T)$ find $c(t)\in \mathscr{V}$ such that 
\begin{equation}\label{abstractpe}
\partial_t (c(t),w)+a(c(t),w)=_{\mathscr{V}'}\langle f(t),w\rangle_{\mathscr{V}}, \text{  }\forall w\in \mathscr{V};  \text{  }c(0)=c_0.
\end{equation}
The following existence and uniqueness result for problem~\eqref{abstractpe} is well known (see, e.g.,~\cite{lions2012non}).
\begin{thm}\label{Funtional}
Assume that the bilinear form $a$ is continuous and coercive on~$\mathscr{V}\times \mathscr{V}$. Then, given~$f\in L^2((0,T);\mathscr{V}')$ and~$c_0\in \mathscr{H}$, there exist a unique solution~$c\in L^2((0,T);\mathscr{V})\cap C^0([0,T];\mathscr{H})$ to~\eqref{abstractpe}, with~$\partial_t c \in L^2((0,T);\mathscr{V}')$.
Moreover, the following energy estimate holds true:
\begin{equation}\label{energype}
    \max_{t\in [0,T]}\Vert c(t)\Vert_{\mathscr{H}}^2+\alpha \int_{0}^T \Vert c\Vert_{\mathscr{V}}^2\, dt \leq \Vert c_0\Vert_{\mathscr{H}}^2+C\int_{0}^T \Vert f\Vert_{\mathscr{V}'}^2\, dt.
\end{equation}
\end{thm}

\noindent\textit{Auxiliary inequalities.}
The inequalities below will be used in several sections of the manuscript. 
Proofs of these results can be found in~\cite{adams2003sobolev,brezis2010functional,tartar2007introduction,evans2010partial}.

\begin{lemm}
Let $S\subset \mathbb{R}^d$, $\epsilon \in (0,\infty)$, $1\leq r,s\leq \infty$ such that $\frac{1}{r}+\frac{1}{s}=1$, for $\phi \in L^r(S)$ and $\varphi \in L^s(S)$, the inequality below is valid
\begin{equation}\label{holderyungiq}
    \left|\int_{S} \phi \varphi \,d\bx \right|\leq \Vert \phi \Vert_{L^r(S)}\Vert \varphi \Vert_{L^s(S)} \leq \frac{\epsilon}{r}\Vert \phi \Vert_{L^r(S)}^r+\frac{\epsilon^{-s/r}}{s}\Vert \varphi \Vert_{L^s(S)}
\end{equation}
\end{lemm}

\begin{lemm}\label{lem:embedingLs} Let $S\subset \mathbb{R}^d$ be an open bounded Lipschitz subset. Then, there exists a constant $C_{s,r}>0$ such that, for every $\phi \in W^{1,r}(S)$, the following inequality holds
\begin{equation}\label{embedingLs}
\Vert \phi \Vert_{L^s(S)}\leq C_{s,r} \Vert \phi \Vert_{W^{1,r}(S)}
\end{equation}
for all $1\leq s < r^{\ast}$, where $r^{\ast}=\frac{dr}{d-r}$ for $r<d$ and $r^{\ast}=\infty$ for $p=d$.
\end{lemm}

\begin{lemm}\label{boundaryembeding} Let $S\subset \mathbb{R}^d$ be an open bounded Lipschitz subset.
Assume that $1\leq l <d$, $1\leq r < \frac{d}{l}$, and $r\leq s\leq \frac{(d-1)r}{d-lr}$. Then, there exists a constant $C_{s,l,r}^{\partial S}>0$ such that
\begin{equation}\label{boundarylq}
    \Vert \phi \Vert_{0,s,\partial S}\leq C_{s,l,r}^{\partial S} \Vert \phi \Vert_{l,r, S}, \quad \forall \phi \in W^{l,r}(S).
\end{equation}
\end{lemm}

\noindent\textit{Well-posedness of the Darcy problem.}
For the mixed formulation \eqref{darcyequ}--\eqref{darcyeqp}, the following result,
adapted from~\cite{ern2021finite}, ensures the well-posedness of the problem.

\begin{prop}
Assume that $f\in \mathrm{L}^2(\Omega)$, $g_D\in \mathrm{H}_{00}^{1/2}(\Gamma_D)$ and $g_N\in \mathrm{H}_{00}^{-1/2}(\Gamma_N)$. Then the problem \eqref{darcyequ}-\eqref{darcyeqp} is well-posed.
\end{prop}

\noindent\textit{Well-posedness of the transport problem.} For the well-posedness of equations \eqref{eqn:varForm}, we assume the velocity $\bu$ in the equations \eqref{darcyequ}-\eqref{darcyeqp} possesses a normal trace in the space $\mathrm{L}^2(\Gamma)$ and satisfies
\begin{equation}\label{velocitycond}
\int_{\Gamma}|\bu\cdot \bn|^{1/2}\,d\bx \geq c_0>0.
\end{equation}

The preceding assumption allows us to establish the following equivalence of the norms.
\begin{lemm}\label{ineq:normeqc}
We define the norm $\Vert \cdot \Vert_{d}$ by
\[
\Vert c \Vert_{d}^2:= \Vert |f|^{1/2}c\Vert_{0,\Omega}^2+|c|_{1,\Omega}^2+\Vert |\bu\cdot \bn|^{1/2}c\Vert_{0,\Gamma}^2.
\]
Under this definition, the norms $\Vert\cdot \Vert_{1,\Omega}$ and $\Vert \cdot \Vert_{d}$ \corr{are equivalent} in $\mathrm{H}^1(\Omega)$.
\end{lemm}
\begin{proof}
Let $G:\mathrm{H}^{1}(\Omega)\rightarrow \R$ be given by 
\[
G(w):=\int_{\Gamma}|\bu\cdot \bn|^{1/2}w\, dS.
\]
Then $G$ is clearly a linear operator. Moreover, we have $G(w)\leq \Vert \bu \cdot \bn \Vert_{0,1,\Gamma}^{1/2}\Vert c\Vert_{0,\Gamma}\leq C \Vert \bu \cdot \bn \Vert_{0,1,\Gamma}^{1/2}\Vert c\Vert_{1,\Omega} $ and $G(1)=\int_{\Gamma}|\bu\cdot \bn|^{1/2}\,d\bx \geq c_0$, where $C$ is a trace constant only depending on $\Gamma$. \corr{Then, by the generalized Poincar\'e's inequality from~\cite{gatica2024introduccion}}, there exist constants $c_1,c_2>0$ such that, for all $c\in \mathrm{H}^{1}(\Omega)$,  
\begin{equation}\label{eqnorm0}
    c_1 \Vert c\Vert_{1,\Omega}^2 \leq (G(c))^2+|c|_{1,\Omega}^2 \leq c_2 \Vert c\Vert_{1,\Omega}^2.
\end{equation}
Then, applying H\"older's inequality to $G(c)$, we obtain
\begin{equation}\label{eqnorm1}
\Vert c\Vert_{1,\Omega}^2\leq c_1^{-1}\left[ (G(c))^2+|c|_{1,\Omega}^2 \right] \leq c_1^{-1}\left[ |\Gamma|\Vert |\bu\cdot \bn|^{1/2}c\Vert_{0,\Gamma}^2+|c|_{1,\Omega}^2 \right]\leq c_1^{-1}\max\{1,|\Gamma|\}\Vert c\Vert_{d}^2.    
\end{equation}

Applying H\"ollder's inequality to $\Vert |f|^{1/2}c\Vert_{0,\Omega}$ and using inequality~\eqref{embedingLs}, we obtain 
\corr{
\begin{equation}\label{eqfH}
\Vert |f|^{1/2}c\Vert_{0,\Omega}^2\leq \Vert f\Vert_{0,1,\Omega} \Vert c\Vert_{0,4,\Omega}^2\leq C_{4,2}^2 \Vert f\Vert_{0,1,\Omega}\Vert c\Vert_{1,\Omega}^2.
\end{equation}
}
Using H\"ollder's inequality to $\Vert |\bu\cdot \bn|^{1/2}c\Vert_{0,\Gamma}$ and using inequality~\eqref{boundarylq}, we obtain for the following two cases:
\corr{
\begin{itemize} 
    \item For $d=2$, we set $s=4$, $r=8/5$ and $l=1$. In this case, the trace operator maps $W^{1,8/5}(\Omega)$ continuously into $L^4(\Gamma)$ since the conditions
    \[
    l=1<2=d, \quad 1<\frac{8}{5}=r<2=\frac{d}{l} \quad \text{and} \quad r=\frac{8}{5}<4=s=\frac{(2-1)\frac{8}{5}}{2-\frac{8}{5}}=\frac{(d-1)r}{d-lr}.
    \]
    \item For $d=3$, we set $s=4$, $r=2$ and $l=1$. In this case, the trace operator maps $H^{1}(\Omega)$ continuously into $L^4(\Gamma)$ since the conditions
    \[
    l=1<3=d, \quad 1<2=r<3=\frac{d}{l} \quad \text{and} \quad r=2<4=s=\frac{(3-1)2}{3-2}=\frac{(d-1)r}{d-lr}.
    \] 
\end{itemize}
From the two previous cases and H\"ollder's inequality, the following inequality follows:
    \begin{equation}\label{equnH}
\Vert |\bu\cdot \bn|^{1/2}c\Vert_{0,\Gamma}^2\leq \Vert \bu\cdot \bn\Vert_{0,1,\Gamma} \Vert c\Vert_{0,4,\Gamma}^2\leq \Vert \bu\cdot \bn \Vert_{0,1,\Gamma}\left( C_{4,1,r}^{\Gamma}\right)^2 \Vert c\Vert_{1,r,\Omega}^2\leq \left( C_{4,1,r}^{\Gamma}\right)^2 |\Omega|^{\frac{2-r}{r}}\Vert \bu\cdot \bn \Vert_{0,1,\Gamma} \Vert c\Vert_{1,\Omega}^2.
\end{equation}
}

\begin{comment}
\begin{equation}\label{equnH}
\Vert |\bu\cdot \bn|^{1/2}c\Vert_{0,\Gamma}^2\leq \Vert \bu\cdot \bn\Vert_{0,\Gamma} \Vert c\Vert_{0,4,\Gamma}^2\leq \Vert \bu\cdot \bn \Vert_{0,\Gamma}\left( C_{4,1,\frac{3}{2}}^{\Gamma}\right)^2 \Vert c\Vert_{1,\frac{3}{2},\Omega}^2\leq \left( C_{4,1,\frac{3}{2}}^{\Gamma}\right)^2 |\Omega|^{1/3}\Vert \bu\cdot \bn \Vert_{0,\Gamma} \Vert c\Vert_{1,\Omega}^2.
\end{equation} 
\end{comment}

Then, as a consequence of~\eqref{eqfH} and~\eqref{equnH}, we obtain the following inequality:
\corr{
\begin{equation}\label{eqnorm2}
\Vert c\Vert_{d}^2\leq \left(1+ C_{4,2}^2\Vert f\Vert_{0,1,\Omega}+\left( C_{4,1,r}^{\Gamma}\right)^2 |\Omega|^{\frac{2-r}{r}}\Vert \bu\cdot \bn \Vert_{0,1,\Gamma}\right) \Vert c\Vert_{1,\Omega}^2.
\end{equation}
}
The result of the proof is an immediate consequence of inequalities \eqref{eqnorm1} and \eqref{eqnorm2}.
\end{proof}

In what follows, we shall employ the norm $\Vert \cdot \Vert_{\bu}$ in the space $[\mathrm{H}^1(\Omega)]^{n_c}$, defined as 
\[
\Vert \bw  \Vert_{\bu}^2:= \sum_{j=1}^{n_c}\Vert w_j\Vert_{d},
\]
where $\bw:=(w_1,\dots,w_{n_c})$.
Henceforth, vectors in $\R^{n_c}$ will be denoted using Fraktur letters.
By Lemma \ref{ineq:normeqc}, this norm is equivalent to the standard product norm in $[\mathrm{H}^1(\Omega)]^{n_c}$.

\begin{prop} Let $(\bu, p)$ be the solution to the problem \eqref{darcyequ}-\eqref{darcyeqp}, and assume that \( \bu\in \boldsymbol{\mathrm{L}}^{4}(\Omega)\) and that it satisfies~\eqref{velocitycond}. Then, there exists a unique solution \( \bc:=(c_1,\dots,c_{n_c})\in [\mathrm{H}^1(\Omega)]^{n_c}\) to the problem~\eqref{eqn:varForm}.  
\end{prop}
\begin{proof}
For the purposes of the proof, we rewrite~\eqref{eqn:varForm} using the following vector notation:
\begin{equation}\label{revarf}
(\partial_t \bc,\bw)_{\Omega}+\mathcal{D}(\bc,\bw)=\mathcal{H}(\bw),
\end{equation}
\corr{where we have defined the operators $\mathcal{D}$ and $\mathcal{H}$.
Specifically, given the generic test functions \( \vv:=(v_1,\dots,v_{n_c})\) and \(\bw:=(w_1,\cdots,w_{n_c})\) in \([\mathrm{H}^1(\Omega)]^{n_c}\),
the operator $\mathcal{D}$ is defined as
\[
\mathcal{D}(\vv,\bw):=\sum_{i=1}^{n_c}\left(\mathcal{A}_i^{\ast}(v_i,\, w_i) + \mathcal{B}(\bu;v_i,\, w_i) + \mathcal{C}_i(\vv,\, w_i) \right),
\]
where
\[
\mathcal{A}_i^{\ast}(v_i,\, w_i):=\int_{\Omega} D_i\nabla v_i \cdot \nabla w_i \, d\bx + \frac{1}{2}\int_{\Omega}(\bu\cdot \nabla v_i)\,w_i\,d\bx-\frac{1}{2}\int_{\Omega}v_i\,(\bu\cdot \nabla w_i)\,d\bx,
\]
while the other operator is given by
\[
\mathcal{H}(\bw):=\sum_{i=1}^{n_c}\mathcal{F}_i(w_i).
\]}

To establish the well-posedness of the problem, we shall invoke \corr{Theorem \ref{Funtional}}.
For this reason, we partition the proof into the following steps:
\paragraph{Continuity of \(\mathcal{D}\).}
Since 
\begin{align*}
    \int_{\Omega}(\bu\cdot\nabla v)w\,d\bx&\leq \Vert \bu\Vert_{0,4,\Omega}|v|_{1,\Omega}\Vert w\Vert_{0,4,\Omega}\leq C_{4,2}\Vert \bu\Vert_{0,4,\Omega}|v|_{1,\Omega}\Vert w\Vert_{1,\Omega}\\
    &\leq C_{4,2} c_1^{-1/2}\max\{1,|\Gamma|\}^{1/2}\Vert \bu\Vert_{0,4,\Omega}\Vert v\Vert_{d}\Vert w\Vert_{d},
\end{align*}
and 
\[
\int_{\Omega}D_i \nabla v_i\cdot \nabla w_i\,d\bx \leq D_i\Vert v_i\Vert_{d}\Vert w_i\Vert_{d}, 
\]
we may conclude that
\begin{equation}\label{Aast}
\mathcal{A}_i^{\ast}(v_i,\, w_i)\leq \alpha_i \Vert v_i\Vert_{d}\Vert w_i\Vert_{d},
\end{equation}
where $\alpha_i:=D_i+C_{4,2}c_1^{-1/2}\max\{1,|\Gamma|\}^{1/2}$.
One readily verifies that 
\begin{equation}\label{BF}
\mathcal{B}(\bu;v_i,\, w_i) \leq \dfrac{1}{2}\Vert|\bu \cdot \bn|^{1/2} v_i \Vert_{0,\Gamma} \Vert|\bu \cdot \bn|^{1/2} w_i\Vert_{0,\Gamma} \quad \text{and} \quad \int_{\Omega}|f|v_i\,w_i\,d\bx\leq \Vert|f|^{1/2} v_i \Vert_{0,\Omega} \Vert|f|^{1/2} w_i\Vert_{0,\Omega}. 
\end{equation}
\corr{On the other hand}, 
\corr{
\begin{equation}\label{Ri}
\int_{\Omega}R_i(\vv)w_i\,d\bx\leq \gamma_i \Vert v_i\Vert_{0,\Omega}\Vert w_i \Vert_{0,\Omega}+\omega\sum_{\substack{j=1\\ j\neq i}}^{n_c}  \Vert v_j\Vert_{0,\Omega}\Vert w_i \Vert_{0,\Omega}\leq \beta_i\Vert \vv\Vert_{\bu}\Vert w_i\Vert_{d}, 
\end{equation}
}
where \(\beta_i:=(\gamma_i+(n_c-1)^{1/2}\omega)c_1^{-1/2}\max\{1,|\Gamma|\}^{1/2}\) and \(\omega:=\max_{i,j=1,\dots,n_c}\{y_{i/j}\gamma_j\}\).
Combining~\eqref{BF} and~\eqref{Ri}, we deduce that
\begin{equation}\label{C}
\corr{\mathcal{C}_i(\vv,\, w_i)\leq (1+\beta_i) \Vert \vv\Vert_{\bu}\Vert w_i\Vert_{d}.}
\end{equation}
By combining~\eqref{Aast}, \eqref{BF}, and~\eqref{C}, 
it follows that \(\mathcal{D}\) defines a continuous bilinear \corr{form in the space $[\mathrm{H}^1(\Omega)]^{n_c}$}.

\paragraph{Continuity of \(\mathcal{H}\).}
Employing the same reasoning, we conclude that
\[
\mathcal{F}_i(w_i)\leq \left[\Vert |f|^{1/2}\tilde{c}_i\Vert_{0,\Omega^{+}}+\Vert |\bu\cdot\bn|^{1/2}c_i^{I}\Vert_{0,\Gamma_{I}}\right] \Vert w_i\Vert_{d},
\]
and therefore \(\mathcal{H}\) defines a continuous linear form on the space $[\mathrm{H}^1(\Omega)]^{n_c}$.

\paragraph{Coercivity of \(\mathcal{D}\).}   
It can be immediately verified by direct evaluation that
\[
\mathcal{D}(\vv,\vv)\geq \alpha \Vert \vv\Vert_{\bu}- \sum_{i=1}^{n_c}\int_\Omega R_i(\vv)\,v_i\,d\bx,
\]
where \(\alpha:=\min\left\{\frac{1}{2},D_1,\dots,D_{n_c}\right\}\).
Indeed, using the Cauchy–Schwarz inequality, we deduce that:
\begin{equation}\label{Ric}
\sum_{i=1}^{n_c}\int_\Omega R_i(\vv)\,v_i\,d\bx\leq \mu \Vert \vv\Vert_{0,\Omega}^2,
\end{equation}
where \(\mu:=\gamma+n_c^{1/2}(n_c-1)^{1/2}\omega\), \(\omega:=\max_{i,j=1,\dots,n_c}\{y_{i/j}\gamma_j\}\) and \(\gamma:=\max_{i=1,\dots,n_c}\{\gamma_i\}\).
Consequently, we obtain
\[
\mathcal{D}(\vv,\vv)+\mu \Vert \vv\Vert_{0,\Omega}^2 \geq \alpha \Vert \vv\Vert_{\bu}.
\]

Gathering the previous steps and applying \corr{Theorem~\ref{Funtional}}, \corr{we conclude the proof}.
\end{proof}
\begin{rem}\label{dependeceondata}
Additionally, Theorem \ref{abstractpe} allows us to obtain the following estimate:
\[
\max_{t\in (0,T]}\Vert \bc(t)\Vert_{0,\Omega}^2+\alpha \int_{0}^T\Vert \bc\Vert_{\bu}^2\, dt\leq \Vert \bc_0\Vert_{0,\Omega}^2+C\sum_{i=1}^{n_c}\int_{0}^{T}\left[\Vert |f|^{1/2}\tilde{c}_i\Vert_{0,\Omega^{+}}+\Vert |\bu\cdot\bn|^{1/2}c_i^{I}\Vert_{0,\Gamma_{I}}\right]^2\, dt.
\]
\end{rem}

%%%%%%%%%%%%%%%%%%%%%%%%%%%% english checked with a deep analysis with specific instrument

%--------
% Franco
%--------
\section{Virtual element method}\label{sec:vem}

In this section, we introduce the VE spaces used to discretize both the Darcy equations, Section~\ref{sec:darcy}, 
and the species transport equations, Section~\ref{sec:spec}.
The theory underlying these spaces is well established in the literature. 
In the following subsections, we focus on describing the local and global spaces 
as well as the projection operators required for the definition of the VEM.
We refer to~\cite{beirao:2016:VEM} for the Darcy flow discretization, 
and to~\cite{beirao:2016:MVE} for the species transport discretization.
Finally, in Section~\ref{sec:res}, we recall some results 
which will be used in the theoretical analysis.

\subsection{Space-time Discretization}

Given a polygon~\(K\), we denote by~\(h_K\) and~\(|K|\) its diameter and area, respectively.
Let~\(\partial K\) be the boundary of~\(K\) and
let~\(e\) be a generic edge of~\(\partial K\).
We denote by~\(\mathbf{n}_K^e\) the unit outward normal to \(K\)
and by~\(h_e\) the length of the edge~\(e\).

We consider a spatial mesh~\(\Omega_h\) composed by polygonal elements,
with mesh size defined as~\(h:=\max_{K\in\Omega_h} h_K\).
Let~\(\mathcal{E}_h(\Omega_h)\) be the set of all edges of the polygons.
We denote the boundary edges by~\(\mathcal{E}_h\), the outflow edges by~\(\mathcal{E}_h^{\mathcal{O}}\),
and the set of inflow edges by~\(\mathcal{E}_h^{\mathcal{I}}\).

\begin{assu}[Space-mesh regularity]
For a given a mesh~\(\Omega_h\), 
we assume that each element~\(K\in\Omega_h\) satisfies:
\begin{enumerate}[i)]
\item \(K\) is star-shaped with respect to a ball \(B_K\) of radius less or equal to \(\rho h_K\),
\item the distance between any two vertexes if \(K\) is greater or equal to  \(\rho h_K\),
\item \(\Omega_h\) is regular, that is, \(h_K\) greater or equal to \(\rho_0 h\),
\end{enumerate}
where \(\rho\) and~\(\rho_0\) are positive constants.
\label{ass:meshH}
\end{assu}

For the time discretization, 
let~\(\mathcal{T}_\tau\) be a uniform partition of the time interval~\((0, T)\) into~\(N\) subintervals, 
given by
\[
0 = t_0 < t_1 < \dots < t_N = T,
\]
where~\(t_i-t_{i-1}=\tau\) for all~\(i = 1, 2,\dots N\).
We denote by~\(I_n:=(t_{n-1}, t_n)\) the~\(n\)-th interval of this time discretization.
\begin{comment}
Finally, for each~$K \in \Omega_h$ and~\(n = 1, \dots, N\), 
we define the space–time prism
\[
K_n := K \times I_n.
\]
\end{comment}

A key aspect in the construction of the VEM are local polynomial projection operators.
They play a fundamental role both in the definition of the discrete VE spaces and, 
more importantly,  in the assembly of the linear system.

Given a mesh element \(K\) and a positive integer \(k\), 
we define the following polynomial projection operators:
\begin{itemize}
\item the \(\nabla\)-\emph{projection}, \(\Pi^\nabla_k:H^1(K)\to\mathbb{P}_k(K)\) defined for all \(v\in H^1(K)\)
\[
\left\{
\begin{array}{rrl}
\mathlarger{\int_K} \nabla \Pi_k^\nabla v\cdot\nabla q_k~\text{d}K &=& \mathlarger{\int_K} \nabla v\cdot\nabla q_k~\text{d}K\qquad \forall q_k\in\mathbb{P}_k(K),\\[1em]
\mathlarger{\int_{\partial K}} \Pi_k^\nabla v~\text{d}e &=& \mathlarger{\int_{\partial K}} v~\text{d}e,   
\end{array}
\right.
\]
\item the \(L^2\)-\emph{scalar projection}, \(\Pi^{0}_k:L^2(K)\to\mathbb{P}_k(K)\) defined for all \(v\in L^2(K)\)
\[
\int_K \Pi_k^0 v\, q_k~\text{d}K = \int_K v\, q_k~\text{d}K\qquad \forall q_k\in\mathbb{P}_k(K)\,,
\]
\item the \(L^2\)-\emph{vector projection},
\(\bpi^0_k:[L^2(K)]^2\to[\mathbb{P}_k(K)]^2\) defined for all \(\bv\in [L^2(K)]^2\)
\[
\int_K \bpi_k^0 \bv\cdot\mathbf{q}_k~\text{d}K = \int_K \bv\, \mathbf{q}_k~\text{d}K\qquad \forall \mathbf{q}_k\in[\mathbb{P}_k(K)]^2\,.
\]
\end{itemize}

\subsection{Spaces for Darcy equations}\label{sec:darcy}

In this section, we briefly introduce the local and global spaces used in the VEM disretization of the Darcy problem.
Let \(k\) be a positive integer.
For any element \(K\in\Omega_h\), we define the following space
\begin{equation*}
\mathbf{U}_h^k(K) := \left\{\bv_h\in H(\div,\Omega)\cap H(\rot,\Omega) : \div(\bv_h)\in\P_k(K),\ 
\rot(\bv_h)\in\P_{k-1}(K),\ \bv_h\cdot\bn_K^e\in\P_k(e)\ \forall e\in\partial K\right\}\,.
\label{eqn:darcyFlow}
\end{equation*}
It is worth noting that \(\mathbf{U}_h^k(K)\) is a typical example of virtual element space.
Indeed, it contains polynomial vector fields of degree \(k\) 
as well as non-polynomial functions.
However, all these functions can be uniquely determined by a suitable set of degrees of freedom.
Among the possible choices, in this work we consider the following:
\begin{itemize}
\item edge moments of the vector normal components
\[
\frac{1}{|e|}\int_e \bv_h\cdot\bn_e\,q_k\,\text{d}e\qquad\forall q_k\in\mathbb{P}_k(e);
\]
\item face moments associated with the divergence of the vector field \(\bv_h\)
\[
\frac{h_K}{|K|}\int_K \div(\bv_h)\,p_k\,\text{d}K\qquad\forall p_k\in\mathbb{P}_k(K)\backslash\mathbb{P}_0(K);
\]
\item moments of the vector field 
\[
\frac{1}{|K|}\int_K \bv_h\cdot \bp_k\,\text{d}K\qquad\forall \bp_k\in\mathcal{G}_k^\perp(K),
\]
where \(\mathcal{G}_k^\perp(K)\) is the \(L^2\) orthogonal space of \(\nabla\mathbb{P}_{k+1}(K)\) in 
\([\mathbb{P}_k(K)]^2\).
\end{itemize}
The global velocity space is then defined by assembling the local spaces across all elements:
\begin{equation*}
\mathbf{U}_h^k(\Omega) := \left\{\bv_h\in H(\div,\Omega) : \bv_h|_K\in\mathbf{U}_h^k(K)\ \forall K\in\Omega_h\right\}.
\end{equation*}
From the properties of the local spaces \(\mathbf{U}_h^k(K)\), 
it follows that the global space \(\mathbf{U}_h^k(\Omega)\) consists of vector fields in \(H(\div,\Omega)\) 
whose normal component is continuous across internal mesh edges.

For the pressure, the virtual element space is not required. 
As in the standard finite element setting, we consider the polynomial space
\begin{equation*}
Q_h^k(\Omega) := \left\{q_h\in L^2(\Omega) : p_h|_K\in\P_{k}(K)\ \forall K\in\Omega_h\right\}.
\end{equation*}

Among the possible choices for the degrees of freedom of \(Q_h^k(\Omega)\), 
in this work we simply use the coefficients of the polynomials.

\subsection{Spaces for transport equations}\label{sec:spec}

For the transport equations,
we employ the enhanced VE spaces introduced in~\cite{beirao:2016:VEM}.  
Given a positive integer~\(k\),
and a polygon~\(K\in\Omega_h\),
we define the local space
\begin{equation}
\begin{array}{lrl}
V_h^k(K):=\Bigg\{v_h\in H^1(K)\cap C^0(\partial K)&:&\Delta v_h\in\mathbb{P}_k(K),\ v_h|_e\in\mathbb{P}_k(e)\ \forall e\in\partial K,\\[-0.7em]
&&\mathlarger{\int_K} \Pi_k^\nabla v_h\,p_k~\text{d}K= \mathlarger{\int_K} v_h\,p_k~\text{d}K\ \forall p_k\in\mathbb{P}_k(K)\backslash\mathbb{P}_{k-2}(K) \Bigg\}.
\nonumber
\end{array}
\end{equation}
Notice that, differently from standard VE spaces~\cite{Beirao:2013:BPV}, 
the Laplacian of the virtual function belongs to \(\mathbb{P}_k(K)\).
This is the key feature of the enhanced VE spaces, 
as it allows the exact construction of an \(L^2\)-projection onto the polynomial space of degree \(k\).

Similarly as the space \(\mathbf{U}_h^k(K)\),
\(V_h^k(K)\) contains polynomials of degree \(k\) as well as virtual functions,
which are uniquely determined by a set of degrees of freedom.
A possible set of degrees of freedom for \(V_h^k(K)\) is
\begin{itemize}
\item vertex values, the values of the virtual function \(v_h\) at each element vertex;
\item internal node values, values of the virtual function at \(k-1\) points along each edge 
(in this work we use the~\(k-1\) internal Gauss-Lobatto points associated 
with a quadrature that exactly integrate polynomials of degree \(2k-3\));
\item internal moments, 
\[
\frac{1}{|K|}\int_K v_h\,p_k\,\text{d}K\qquad\forall p_k\in\mathbb{P}_{k-2}(K).
\]
\end{itemize}
The global space is obtained by gluing the local spaces along the mesh edges with \(C^0\) continuity:
\begin{equation*}
V_h^k(\Omega) := \left\{v_h\in H^1(\Omega) : v_h|_K\in V_k^h(K)\ \forall K\in\Omega_h\right\}.
\end{equation*}

\begin{comment}
As a consequence, when considering \(n_c\) species,
the concentration vector has components in \(V_h^k(\Omega)\), i.e.,
\(\bc\in [V_h^k(\Omega)]^{n_c}\).
\end{comment}

\subsection{Auxiliary results on projection operators}\label{sec:res}

In this subsection we collect several technical results concerning polynomial
and virtual element projection operators. 
These estimates will be used throughout the error analysis.

\corr{The following lemma is essential for certain parts of the analysis and corresponds to}
Lemma~4.5.3 of~\cite{ern2021finite}.

\begin{lemm}[Inverse inequality] \label{inverseineq}
Let~\(\rho_0\) be the parameter introduced in Assumption~\ref{ass:meshH},
$1 \le q \le \infty$ and $0 \le m \le l$. 
Then there exists 
$C := C(m,l, p, q, \rho_0)$ 
such that for all $v_h \in \mathbb{P}_{k}(K).\cap W_p^l(K) \cap W_q^m(K)$, we have
\begin{equation}
\label{eq:4.5.4}
\|v_h\|_{W_p^l(K)} \le C\, h_K^{\,m - l + n/p - n/q} \|v_h\|_{W_q^m(K)}.
\end{equation}
\end{lemm}

\corr{ The following inequality for Virtual Element spaces can be found in \cite{Beirao:2025:SST}.
\begin{lemm}[Inverse inequality VEM]\label{inverseineqV}
Let $K \in \Omega_h$ be a spatial element with diameter $h_K$. There exists a positive constant $C$, independent of the local mesh size $h_K$, such that the following inverse estimate holds for all functions in the discrete space $V_{h}^k(K)$:
$$\|\nabla v\|_{0,K} \le C h_K^{-1} \|v\|_{0,K}.$$    
\end{lemm}
}

The following lemma is a local version of the standard trace inequality, which can be found in~\cite{ern2021finite}.
\begin{lemm}[Trace inequality]\label{traceineq}
Let \(p\in[1,\infty)\). Then, for all \(K\in \mathcal{T}_h\) and all \(v\in \mathrm{W}^{1,p}(K)\), there exist a positive constant $C_{TR}$ independent of $h_K$ such that 
\begin{equation}\label{traceineqq}
\Vert v\Vert_{\mathrm{L}^p(\partial K)}\leq C_{TR}\left(h^{-\frac{1}{p}}\Vert v\Vert_{\mathrm{L}^{p}(K)}+h^{1-\frac{1}{p}}\Vert v\Vert_{\mathrm{W}^{1,p}(K)}\right).
\end{equation}
\end{lemm}

\corr{
The following lemma establishes a fundamental estimate required for our subsequent analysis. While related results have been discussed in the literature (see, e.g., \cite{L2Crouzeeix1978, L2Diening2021}), we provide an alternative proof specifically adapted to the hypotheses of our current framework.
}
\begin{lemm}\label{lem:L2PLp0}
Let $1< r\leq \infty$ and $n\in \mathbb{N}$. Then for any $w\in L^r(\Omega)\cap L^2(\Omega)$ we have
\begin{equation}\label{L2PLp}
\Vert \Pi^0_k w \Vert_{L^r(K)}\lesssim \Vert w \Vert_{L^r(K)}.
\end{equation}
\end{lemm}
\begin{proof}
We treat the two ranges of $r$ separately.
\begin{itemize}

\item {\bf Case $r\in[2,\infty)$:}

From~\eqref{eq:4.5.4} and the continuity of the $L^2$-projection, we obtain the following local estimate
\[
\Vert \Pi^0_k w \Vert_{L^r(K)}\lesssim h_{K}^{2\left(\frac{1}{p}-\frac{1}{2}\right)}\Vert \Pi^0_k w \Vert_{L^2(K)} \lesssim h_{K}^{2\left(\frac{1}{p}-\frac{1}{2}\right)}\Vert  w \Vert_{L^2(K)}.
\]
Apply H\"older's inequality on the previous inequality and using $|K|\approx h_K^2$ we have
\[
\Vert \Pi^0_k w \Vert_{L^r(K)}\lesssim h_{K}^{2\left(\frac{1}{p}-\frac{1}{2}\right)}|K|^{\frac{1}{2}-\frac{1}{p}}\Vert  w \Vert_{L^r(K)}\lesssim h_{K}^{2\left(\frac{1}{p}-\frac{1}{2}\right)}h_K^{2\left(\frac{1}{2}-\frac{1}{p}\right)}\Vert  w \Vert_{L^r(K)}=\Vert  w \Vert_{L^r(K)}.
\]
\item  {\bf Case $r\in(1,2]$:}

For this case, we choose $r':=\frac{r}{r-1}\geq 2$. Using the previous case, H\"older's inequality and the dual definition of the $L^r$-norm, we obtain
\begin{align*}
\Vert \Pi^0_k w \Vert_{L^r(K)}&:=\sup_{v\in L^{r'}(K)}\frac{\mathlarger{\int}_{K} \Pi^0_k w\, v~\text{d}K}{\Vert v\Vert_{L^{r'}(K)}} =\sup_{v\in L^{r'}(K)}\frac{\mathlarger{\int}_{K}  w\, \Pi^0_k v~\text{d}K}{\Vert v\Vert_{L^{r'}(K)}}\\
&\lesssim \sup_{v\in L^{r'}(K)}\frac{\Vert w\Vert_{L^{r}(K)} \Vert\Pi^0_k v\Vert_{L^{r'}(K)}}{\Vert v\Vert_{L^{r'}(K)}}\\
&\lesssim \sup_{v\in L^{r'}(K)}\frac{\Vert w\Vert_{L^{r}(K)} \Vert v\Vert_{L^{r'}(K)}}{\Vert v\Vert_{L^{r'}(K)}} \\
&\lesssim \Vert  w \Vert_{L^r(K)}.
\end{align*}
\end{itemize} 
Since both cases yield, the theorem follows.
\end{proof}

The following approximation result for polynomials, taken from
\cite{Beirao:2013:BPV}, will also be needed during the error analysis of the problem here proposed.

\begin{lemm}\label{polyproyect}
Let $K\in \mathcal{T}_h$ and $v\in W^{s,p}(K)$, where $1\leq s\leq k+1$. Then, there exists a $v_{\pi}\in \mathbb{P}_k(K)$ such that
$$
\Vert v-v_{\pi} \Vert_{0,p,K}+h_{K}|v-v_{\pi}|_{1,p,K}\lesssim h_K^s |v|_{s,p,K}.
$$
\end{lemm}

We conclude with an approximation property of the virtual element space
from~\cite{Beirao:2013:BPV}.

\begin{lemm}\label{vemproyect}
Let $K\in \mathcal{T}_h$ and $v\in H^s(K)$, where $2\leq s\leq k+1$. Then, there exists a $\mathcal{P}_h v\in V_k(K)$ such that 
$$
\Vert v-\mathcal{P}_h v \Vert_{0,2,K}+h_{K}|v-\mathcal{P}_h v|_{1,2,K}\lesssim h_K^s |v|_{s,2,K}.
$$
\end{lemm}

%%%%%%%%%%%%%%%%%%%%%%%%%%%% english checked!
%--------
% Ruben
%--------
\section{The discrete problem}\label{sec:disc}
In the following, we derive the discrete scheme associated with our problem. 
To this end, we introduce the corresponding discrete bilinear and linear forms, 
establish the existence and uniqueness of the discrete solutions, 
and finally provide the convergence orders for the error estimates in each case.

We now introduce the discrete bilinear forms employed in the method. 
Since much of the analysis is conducted on a generic element $K$, 
we denote by 
$\mathcal{M}^{K}(\cdot,\cdot)$, 
$\mathcal{N}^{K}(\cdot,\cdot)$, 
$\mathcal{A}_i^{K}(\cdot,\cdot)$, 
$\mathcal{B}^{K}(\cdot,\cdot)$, 
$\mathcal{C}_i^{K}(\cdot,\cdot)$, 
$\mathcal{G}_N^{K}(\cdot)$, 
$\mathcal{G}_D^{K}(\cdot)$, and $\mathcal{F}_i^{K}(\cdot)$ 
the restriction of the corresponding forms to $K$. 
Furthermore, 
let $S^{K}(\cdot,\cdot)$, $S_{\mathcal{M}}^{K}(\cdot,\cdot)$, and $S_{\mathcal{A}_i}^{K}(\cdot,\cdot)$ 
be symmetric bilinear forms that scale as $(\cdot,\cdot)_{0,K}$, $\mathcal{M}^K$, and $\mathcal{A}_i^K$ on the kernels of $\Pi_k^{0}$, $\boldsymbol{\Pi}_k^{0}$, and $\Pi_k^{\nabla}$, respectively; 
these will serve as stabilization terms. 
Specifically, we define the following discrete bilinear and linear forms:

\begin{equation}\label{stabilitation}
\begin{aligned}
(c_h,c_h)_{0,K} &\approx S^{K}(c_h,c_h), && \forall c_h \in \mathrm{V}_h^k(K) \quad \text{with } \Pi_k^{0}c_h=0, \\
\mathcal{M}^{K}(\bv_h,\bv_h) &\approx S_{\mathcal{M}}^{K}(\bv_h,\bv_h), && \forall \bv_h \in \boldsymbol{\mathrm{U}}_h^k(K) \quad \text{with } \boldsymbol{\Pi}_k^{0}\bv_h=0, \\
\mathcal{A}_i^{K}(c_h,c_h) &\approx S_{\mathcal{A}_i}^{K}(c_h,c_h), && \forall c_h \in \mathrm{V}_h^k(K) \quad \text{with } \Pi_k^{\nabla}c_h=0.
\end{aligned}
\end{equation}

We refer the reader to~\cite{Beirao:2013:BPV,beirao:2016:VEM} for a more detailed construction of these stabilization terms. 
In the numerical experiments of Section~\ref{sec:numExe}, 
we employ the standard \texttt{dofi-dofi} stabilization for each $S^{K}_*$, 
where the subscript $*$ stands for any of the labels introduced above.

We now have all the ingredients to define the local bilinear forms 
as well as the local loading terms on each element~$K$ to establish the discrete formulation of the problem:
\begin{align*}
m_h^{K}(c_h, w_h) &:= (\Pi_k^{0}c_h, \Pi_k^{0}w_h)_{0,K} + S^{K}(c_h - \Pi_k^{0}c_h, w_h - \Pi_k^{0}w_h), \\
\mathcal{M}_h^{K}(\bu_h, \bv_h) &:= \mathcal{M}^{K}(\boldsymbol{\Pi}_k^{0}\bu_h, \boldsymbol{\Pi}_k^{0}\bv_h) + S_{\mathcal{M}}^{K}(\bu_h - \boldsymbol{\Pi}_k^{0}\bu_h, \bv_h - \boldsymbol{\Pi}_k^{0}\bv_h), \\
\mathcal{A}_{i,h}^{K}(\phi_h, \psi_h) &:= D_i(\nabla \Pi_k^{\nabla}\phi_h, \nabla \Pi_k^{\nabla}\psi_h)_{0,K} + S_{\mathcal{A}_{i}}^{K}(\phi_h - \Pi_k^{\nabla}\phi_h, \psi_h - \Pi_k^{\nabla}\psi_h), \\
\mathcal{K}_h^{K}(\bu_h; \phi_h, \psi_h) &:= (\boldsymbol{\Pi}_k^{0}\bu_h \cdot \boldsymbol{\Pi}_k^{0}\nabla \phi_h, \Pi_k^{0}\psi_h)_{0,K}, \\
\mathcal{K}_h^{\ast,K}(\bu_h; \phi_h, \psi_h) &:= \dfrac{1}{2}\mathcal{K}_h^{K}(\bu_h; \phi_h, \psi_h) - \dfrac{1}{2}\mathcal{K}_h^{K}(\bu_h; \psi_h, \phi_h), \\
\mathcal{N}_h^{K}(\bv_h, q_h) &:= \mathcal{N}^{K}(\bv_h, q_h), \\
\mathcal{C}_{i,h}^{K}(\Phi, \psi) &:= \mathcal{C}_{i}^{K}(\boldsymbol{\Pi}_k^{0}\Phi, \Pi_k^{0}\psi), \\
\mathcal{B}_h^{e}(\bu_h; \phi_h, \psi_h) &:= \frac{1}{2}(|\bu_h|\phi_h, \psi_h)_{0,e}, \\
\mathcal{G}_{h,N}^{K}(\bv_h) &:= \mathcal{G}_{N}^{K}(\bv_h), \\
\mathcal{G}_{h,D}^{K}(q_h) &:= \mathcal{G}_{D}^{K}(q_h), \\
\mathcal{F}_{i,h}^{K}(\phi_h) &:= (\max\{0,f\} \tilde{c}_i, \Pi_k^{0}\phi_h)_{0,K}, \\
\mathcal{F}_{i,h}^{e}(\bu_h; \phi_h) &:= -(\min\{0, \bu_h \cdot \bn_e\} \tilde{c}_i^{I}, \phi_h)_{0,e}.
\end{align*}

Finally, the global forms, denoted by omitting the superscript $K$ or $e$, are obtained by assembling the local contributions over the mesh $\mathcal{T}_h$. For instance, we define
\begin{equation}
m_h(\cdot, \cdot) := \sum_{K \in \mathcal{T}_h} m_h^K(\cdot, \cdot), \quad \mathcal{A}_{i,h}(\cdot, \cdot) := \sum_{K \in \mathcal{T}_h} \mathcal{A}_{i,h}^K(\cdot, \cdot),
\end{equation}
with an analogous summation convention for all other bilinear forms.
\corr{
However, since the data of the problem may involve either the entire domain $\Omega$ or the boundary $\Gamma$, 
we define the global linear forms as the sum of their local contributions:
\begin{equation}
\mathcal{F}_{i,h}^{\Omega}(\phi_h) := \sum_{K \in \mathcal{T}_h}\mathcal{F}_{i,h}^{K}(\phi_h)\quad\text{and}\quad 
\mathcal{F}_{i,h}^{\Gamma}(\phi_h) := \sum_{e \in \Gamma} \mathcal{F}_{i,h}^{e} (\bu_h; \phi_h).
\end{equation}}

The remainder of this section is organized as follows.
First, we address the discretization of the Darcy~equations. 
Second, we focus on the transport equations, 
for which the analysis is further split into two levels: 
the semi-discrete and the fully discrete formulations.
\corr{
\begin{rem}
It is worth noting that we have applied stabilization to both the diffusive and the mass terms. 
Depending on the physical regime of the transport equation, 
one of these two stabilization terms can be omitted~\cite{Beirao:2014:THG}. 
For instance, if the diffusion coefficient $D_i$ is very small, 
we can neglect the stabilization term associated with $\mathcal{A}_{i,h}^{K}(\phi_h, \psi_h)$.
To ensure greater code flexibility and avoid the use of heuristic thresholds for $D_i$, 
we consistently employ both stabilizations.
\end{rem}}

\subsection{Discrete formulation for Darcy equations}

Based on the continuous variational formulation and the bilinear and linear forms introduced before, 
we define the following discrete variational formulation for the Darcy problem:
find $(\bu_h,p_h)\in \mathbf{\mathrm{U}}_h^k(\Omega)\times Q_h^k(\Omega)$, such that 
\begin{align}
\mathcal{M}_h(\bu_h,\bv_h)+ \mathcal{N}_h(\bv_h,p_h)&=\mathcal{G}_{h,N}^{\Gamma}(\bv_h), 
&\forall\bv_h\in \mathbf{\mathrm{U}}_h^k(\Omega), \label{darcyequsemi}\\
\mathcal{N}_h(\bu_h,q_h)&=\mathcal{G}_{h,D}^{\Gamma}(q_h),
&\forall q_h\in  Q_h^k(\Omega).\label{darcyeqpsemi}
\end{align}
From~\cite{Beirao:2013:BPV,beirao:2016:VEM}, 
we report two fundamental results concerning the well--posedness of the discrete problem 
and the corresponding error estimates.

\begin{thm}
Problem \eqref{darcyequsemi}--\eqref{darcyeqpsemi} has a unique solution \((\bu_h,p_h)\in \mathbf{\mathrm{U}}_h^k(\Omega)\times Q_h^k(\Omega)\), verifying the estimate
\[
\Vert \bu_h\Vert_{0,\Omega}+\Vert \mathrm{div}\, \bu_h\Vert_{0,\Omega}+\Vert p_h\Vert_{0,\Omega}\lesssim \Vert g_N\Vert_{\mathrm{H}_{00}^{-1/2}(\Gamma_N)}+\Vert g_D\Vert_{\mathrm{H}_{00}^{1/2}(\Gamma_D)}.
\]
\end{thm}

\begin{thm}
Let \((\bu,p)\in \mathbf{H}_{N,g_N}(\mathrm{div},\Omega)\times \mathrm{L}^2(\Omega) \) be the solution of the problem \eqref{darcyequ}-\eqref{darcyeqp} and \((\bu_h,p_h)\in \mathbf{\mathrm{U}}_h^k(\Omega)\times Q_h^k(\Omega)\) be the solution of problem \eqref{darcyequsemi}-\eqref{darcyeqpsemi}.Then it holds 
\begin{equation}\label{errorDarcy}
\Vert \bu-\bu_h\Vert_{0,\Omega}+h\Vert\mathrm{div}\, (\bu-\bu_h)\Vert_{0,\Omega}\lesssim h^{k+1}|\bu|_{k+1,\Omega} \quad \text{and}\quad \Vert p-p_h\Vert_{0,\Omega}\lesssim h^{k+1}(|\bu|_{k+1,\Omega}+|p|_{k,\Omega}).    
\end{equation}
\end{thm}

\subsection{Semi-discrete formulation for the transport equations}

Similarly, based on the continuous analogues, we define the following semi-discrete variational formulation for the transport problem: 
find $\bc_h := (c_{1,h}, \dots, c_{n_c,h}) \in \mathrm{H}^1(0,T; V_{h}^{k}(\Omega)^{n_c})$ such that, for all $t \in (0,T]$ and \corr{for all $w_{i,h}\in V_h^k(\Omega)$,} it holds:
\begin{equation}
\label{eqn:varFormsemi}
\begin{aligned}
\sum_{i=1}^{n_c} \left[ m_h(\partial_t c_{i,h}, w_{i,h}) + \mathcal{D}_{i,h}(\bu_h; \bc_h, w_{i,h}) \right] &= \sum_{i=1}^{n_c} \left[ \mathcal{F}_{i,h}^{\Omega}(w_{i,h}) + \mathcal{F}_{i,h}^{\Gamma}(w_{i,h}) \right], \\
c_{i,h}(0) &= \Pi_{k}^{0}c_{i}^0, s
\end{aligned}
\end{equation}
where the operator $\mathcal{D}_{i,h}$ is defined as:
\begin{equation*}
\mathcal{D}_{i,h}(\bu_h; \bc_h, w_{i,h}) := \mathcal{A}_{i,h}(c_{i,h}, w_{i,h}) + \mathcal{K}_h^{\ast}(\bu_h; c_{i,h}, w_{i,h}) + \mathcal{B}_h(\bu_h; c_{i,h}, w_{i,h}) + \mathcal{C}_{i,h}(\bc_h, w_{i,h}).
\end{equation*}
We observe that, although not explicitly indicated in the notation, 
the boundary source term $\mathcal{F}_{i,h}^{\Gamma}(\cdot)$ also 
depends on the velocity field $\bu_h$ through the inflow boundary conditions, 
as defined in the local contributions. 
For the sake of simplicity in the notation, 
we will omit this explicit dependence in the remainder of this section.
\corr{Moreover, with a slight abuse of notation, 
we denote the species--wise mass term as $m_h(\bc,\bw)$ 
to provide a more compact representation, 
omitting the explicit summation for brevity.}

\subsubsection{Existence, uniqueness and error estimates}

Before addressing the well-posedness and the convergence of the semi--discrete formulation for the transport equations,
we introduce two lemmas that play a key role in the subsequent analysis.

\begin{lemm}\label{eqnormd} Let $\bu$ and $\bu_h$ denote the velocity solutions of problems \eqref{darcyequ}-\eqref{darcyeqp} and \eqref{darcyequsemi}-\eqref{darcyeqpsemi}, respectively. Assume that \(\bu\in \mathbf{\mathrm{H}}^{1+s}(\Omega)\), for some \(s>0\), and define the mesh-dependent norm
\[
\Vert\psi_h \Vert_{h,d}^2:= \sum_{K\in \Omega_h}(|\psi_h|_{1,K}^2+\Vert |f|^{1/2}\Pi_{k}^0\psi_h \Vert_{0,K}^2)+\sum_{e\in \mathcal{E}_h}\Vert |\bu_h\cdot \bn|^{1/2}\psi_h\Vert_{0,e}^2.
\]
Then, there exists $h_0 > 0$ such that for all $ h \leq h_0$ such that the norms \(\Vert\cdot \Vert_{h,d}\) and \(\Vert\cdot \Vert_{1,\Omega}\) are \corr{equivalent} in $V_h^{k}(\Omega)$.
\end{lemm}
\begin{proof}
\corr{
Let $\Vert \cdot \Vert_{\ast}$ given by 
\[
\Vert\psi_h \Vert_{\ast}^2:= \sum_{K\in \Omega_h}|\psi_h|_{1,K}^2+\sum_{e\in \mathcal{E}_h}\Vert |\bu_h\cdot \bn|^{1/2}\psi_h\Vert_{0,e}^2.
\]
By Lemma \ref{ineq:normeqc} we have $\Vert\cdot \Vert_{1,\Omega}^2\approx |\cdot|_{1,\Omega}^2+\Vert |\bu\cdot \bn|^{1/2}\cdot \Vert_{0,\Gamma}^2$. Furthermore, note that 
\[
\left| \Vert \phi_h\Vert_{\ast}^2- |\phi_h|_{1,\Omega}^2+\Vert |\bu\cdot \bn|^{1/2}\phi_h \Vert_{0,\Gamma}^2\right|=\left|(|\bu\cdot\bn|\phi_h,\phi_h)_{0,\Gamma}-\sum_{e\in \mathcal{E}_h}(|\bu_h\cdot\bn|\phi_h,\phi_h)_{0,e}\right|.
\]
We now estimate $\mathrm{I}_{\Gamma}:=\left|(|\bu\cdot\bn|\phi_h,\phi_h)_{0,\Gamma}-\sum_{e\in \mathcal{E}_h}(|\bu_h\cdot\bn|\phi_h,\phi_h)_{0,e}\right|$:
}
\corr{
\begin{equation}\label{boundaryidentity}
\mathrm{I}_{\Gamma}\lesssim \sum_{e\in \mathcal{E}_h}|(|\bu\cdot\bn|-|\bu_h\cdot\bn|\phi_h,\phi_h)_{0,e}|\lesssim \sum_{e\in \mathcal{E}_h}\Vert (\bu-\bu_h)\cdot\bn\Vert_{0,e} \Vert\phi_h\Vert_{0,4,e}^2.
\end{equation}
Employing~\eqref{traceineqq} and Cauchy–Schwarz inequality, we obtain that
\begin{equation}\label{boundaryinequality}
\sum_{e\in \mathcal{E}_h} \Vert (\bu-\bu_h)\cdot\bn\Vert_{0,e}\Vert\phi_h\Vert_{0,4,e}^2 \lesssim  \left[ \left(\sum_{e\in \mathcal{E}_h} h_{K_e}^{-1}\Vert \bu-\bu_h\Vert_{0,K_e}^2\right)^{\frac{1}{2}}+\left( \sum_{e\in \mathcal{E}_h}h_{K_e}|\bu-\bu_h|_{1,K_e}^2\right)^{\frac{1}{2}}\right]\left( \sum_{e\in \mathcal{E}_h} \Vert\phi_h\Vert_{0,4,e}^4\right)^{\frac{1}{2}},
\end{equation}
}
\corr{
Using~\eqref{errorDarcy}, we may conclude that 
\begin{equation}\label{boundaryinequality2}
    \sum_{e\in \mathcal{E}_h} \Vert (\bu-\bu_h)\cdot\bn\Vert_{0,e}\Vert\phi_h\Vert_{0,4,e}^2\lesssim  \left[ h^{s+\frac{1}{2}}|\bu|_{1+s,\Omega}+\left( \sum_{e\in \mathcal{E}_h}h_{K_e}|\bu-\bu_h|_{1,K_e}^2\right)^{\frac{1}{2}}\right]\left( \sum_{e\in \mathcal{E}_h} \Vert\phi_h\Vert_{0,4,e}^4\right)^{\frac{1}{2}}
\end{equation}
}

\corr{
where \(K_e\in \mathcal{T}_h\) is such that \(e\subset \partial K_e\).
Using triangular inequality together with Lemmas~\ref{polyproyect} and~\ref{inverseineqV} , we obtain that
\[
|\bu-\bu_h|_{1,K_e}\lesssim |\bu-\bu_{\pi}|_{1,K_e}+|\bu_{\pi}-\bu_h|_{1,K_e}\lesssim h_{K_e}^{s}|\bu|_{1+s,K_e}+h_{K_e}^{-1}\Vert\bu_{\pi}-\bu_h\Vert_{0,K_e}.
\]
The previous inequality, together with~\eqref{boundaryidentity}, ~\eqref{boundaryinequality}, ~\eqref{boundaryinequality2} and~\eqref{errorDarcy}, allows us to deduce that
\begin{equation*}
\mathrm{I}_{\Gamma}\lesssim h^{s+1/2}|\bu|_{1+s,\Omega}\left( \sum_{e\in \mathcal{E}_h} \Vert\phi_h\Vert_{0,4,e}^4\right)^{\frac{1}{2}} =h^{s+1/2}|\bu|_{1+s,\Omega}\Vert\phi_h\Vert_{0,4,\Gamma}^2.
\end{equation*}
Then, applying Lemma~\ref{boundaryembeding} as in inequality~\eqref{equnH} in the term \(\Vert \phi_h\Vert_{0,4,e}\)  to this inequality, we obtain 
\begin{equation}\label{estimtestar}
  \mathrm{I}_{\Gamma}\lesssim h^{s+\frac{1}{2}}|\bu|_{1+s,\Omega}\Vert\phi_h\Vert_{1,\Omega}^2 
\end{equation}
Inequalities~\eqref{estimtestar}, \eqref{eqnorm0} and \eqref{eqnorm2} implies that
\[
\Vert \phi_h \Vert_{1,\Omega}^2 \lesssim |\phi_h|_{1,\Omega}^2+\Vert |\bu\cdot\bn|\phi_h\Vert_{0,\Gamma}^2 \lesssim \Vert \phi_h\Vert_{\ast}^2 +\mathrm{I}_{\Gamma}\lesssim \Vert \phi_h\Vert_{\ast}^2+ h^{s}|\bu|_{1+s,\Omega}\Vert\phi_h\Vert_{1,\Omega}^2
\]
and
\[
\Vert \phi_h \Vert_{\ast}^2 \lesssim |\phi_h|_{1,\Omega}^2+\Vert |\bu\cdot\bn|\phi_h\Vert_{0,\Gamma}^2+\mathrm{I}_{\Gamma} \lesssim \Vert \phi_h\Vert_{1,\Omega}^2 +\mathrm{I}_{\Gamma}\lesssim (1+ h^{s}|\bu|_{1+s,\Omega})\Vert\phi_h\Vert_{1,\Omega}^2.
\]
Hence, we obtain
\[
(1-h^{s}|\bu|_{1+s,\Omega})\Vert \phi_h \Vert_{1,\Omega}^2 \lesssim \Vert \phi_h \Vert_{\ast}^2 \lesssim (1+h^{s}|\bu|_{1+s,\Omega})\Vert \phi_h \Vert_{1,\Omega}^2.
\]
Thus, for a sufficiently small $h > 0$, it holds that 
\begin{equation}\label{hd1}
    \Vert \phi_h \Vert_{1,\Omega}^2 \approx \Vert \phi_h \Vert_{\ast}^2.
\end{equation}
}

\corr{
On the other hand, by applying the vector Cauchy-Schwarz inequality to $|f|$ and $(\Pi_{k}^0\psi_h)^2$, and invoking Lemma \ref{lem:L2PLp0}, we obtain
\[
\sum_{K\in \Omega_h}\Vert |f|^{1/2}\Pi_{k}^0\psi_h \Vert_{0,K}^2:= \sum_{K\in \Omega_h}\int_{K} |f|(\Pi_{k}^0\psi_h)^2  \leq  \sum_{K\in \Omega_h} \Vert f\Vert_{0,K} \Vert \Pi_{k}^0\psi_h \Vert_{0,4,K}^2 \lesssim \sum_{K\in \Omega_h} \Vert f\Vert_{0,K} \Vert \psi_h \Vert_{0,4,K}^2.
\]
Applying the vector Cauchy-Schwarz inequality to $\sum_{K\in \Omega_h} \Vert f\Vert_{0,K} \Vert \psi_h \Vert_{0,4,K}^2$, we have that 
\[
\sum_{K\in \Omega_h}\Vert |f|^{1/2}\Pi_{k}^0\psi_h \Vert_{0,K}^2 \lesssim \left( \sum_{K\in \Omega_h} \Vert f\Vert_{0,K}^2\right)^{1/2}\left( \sum_{K\in \Omega_h} \Vert \psi_h \Vert_{0,4,K}^4\right)^{1/2}= \Vert f\Vert_{0,\Omega}\Vert \psi_h \Vert_{0,4,\Omega}^2.
\]
From the previous inequality and by applying Lemma \ref{lem:embedingLs}, we can deduce that
\begin{equation}\label{hd2}
\sum_{K\in \Omega_h}\Vert |f|^{1/2}\Pi_{k}^0\psi_h \Vert_{0,K}^2 \lesssim \Vert f\Vert_{0,\Omega}\Vert \psi_h \Vert_{1,\Omega}^2
\end{equation}
}

\corr{
Finally, \eqref{hd1} and \eqref{hd2} allow us to conclude the desired equivalence of the norms.
}
\end{proof} 

To state the following result, we employ the norm $\Vert \cdot \Vert_{h}$ 
in the space $[V_h^k(\Omega)]^{n_c}$, defined for any vector $\vv_h:= (v_{1,h}, \dots, v_{n_c,h})$ as
\begin{equation}
\Vert \vv_h \Vert_{h}^2 := \sum_{j=1}^{n_c} \Vert v_{j,h} \Vert_{h,d}^2,
\end{equation}
in order to have a more compact notation that takes into account the presence of multiple species.

\begin{lemm}\label{PropertiesBF} Let $\mathcal{F}_{h}^{\Omega}:V_h^k(\Omega)^{n_c}\rightarrow \R$, $\mathcal{F}_{h}^{\Gamma}:V_h^k(\Omega)^{n_c}\rightarrow \R$ and $\mathcal{D}_h:V_h^k(\Omega)^{n_c}\times V_h^k(\Omega)^{n_c}\rightarrow \R$ be given by 
\[
    \corr{\mathcal{F}_{h}^{\Omega}(\bw_h):=\sum_{i=1}^{n_c}\mathcal{F}_{i,h}^{\Omega}(w_{i,h})}, \quad \corr{\mathcal{F}_{h}^{\Gamma}(\bw_h):=\sum_{i=1}^{n_c}\mathcal{F}_{i,h}^{\Gamma}(w_{i,h})}\quad \text{and}\quad
    \mathcal{D}_{h}(\bc_h,\bw_h):=\sum_{i=1}^{n_c}\mathcal{D}_{i}(\bu_h;\bc_h,w_{i,h}),
\]
where $\bc_h:=(c_{1,h},\dots,c_{n_c,h})$ and $\bw_h:=(w_{1,h},\dots,w_{n_c,h})$.

Then, the operators $\mathcal{F}_{i,h}^{\Omega}$ and $\mathcal{F}_{i,h}^{\Gamma}$ are continuous in the VEM space $V_h^k(\Omega)$ with respect to the norm $\Vert\cdot \Vert_{h}$. Furthermore, the form $\mathcal{D}_h$ satisfies the conditions of continuity and coercivity.
\end{lemm}
\begin{proof}
For simplicity, we denote $\bc_h:=(c_{1,h},\dots,c_{n_c,h})$ and $\bw_h:=(w_{1,h},\dots,w_{n_c,h})$.

The definitions of $\mathcal{F}_{i,h}^{\Omega}$ and $\mathcal{F}_{i,h}^{\Gamma}$ immediately imply that
\begin{equation*}
\mathcal{F}_{i,h}^{\Omega}(w_{i,h})\lesssim \Vert w_{i,h}\Vert_{h,d}\, \text{and}\, \mathcal{F}_{i,h}^{\Gamma}(w_{i,h})\lesssim \Vert w_{i,h}\Vert_{h,d},\, \forall \bw_h \in V_h^k(\Omega)^{n_c},
\end{equation*}
and hence
\begin{equation*}
\mathcal{F}_{h}^{\Omega}(\bw_{h})\lesssim \Vert w_{h}\Vert_{h}\, \text{and}\, \mathcal{F}_{h}^{\Gamma}(w_{h})\lesssim \Vert w_{h}\Vert_{h},\, \forall \bw_h \in V_h^k(\Omega)^{n_c}.
\end{equation*}
Regarding the continuity and coercivity of the form $\mathcal{D}_h$,
we observe that 
\begin{equation*}
|(\boldsymbol{\Pi}_k^{0}\bu_h\cdot \boldsymbol{\Pi}_k^{0}\nabla c_{i,h},\Pi_k^{0} w_{i,h})_{0,K}|\leq \Vert\boldsymbol{\Pi}_k^{0}\bu_h \Vert_{0,4,K} \Vert\boldsymbol{\Pi}_k^{0}\nabla c_{i,h} \Vert_{0,K} \Vert\Pi_k^{0} w_{i,h} \Vert_{0,4,K}.
\end{equation*}
Then, using \eqref{L2PLp}, H\"older's inequality the embedding $H^1\hookrightarrow L^4$ in \eqref{embedingLs}, we get 
\begin{align*}
    |\mathcal{K}_h(\bu_h;c_{i,h},w_{i,h})|&\lesssim \sum_{K\in \mathcal{T}_h} \Vert\bu_h \Vert_{0,4,K} \Vert\nabla c_{i,h} \Vert_{0,K} \Vert w_{i,h} \Vert_{0,4,K} \lesssim \Vert\bu_h \Vert_{0,4,\Omega} \Vert\nabla c_{i,h} \Vert_{0,\Omega} \Vert w_{i,h} \Vert_{0,4,\Omega}\\
    &\lesssim \Vert\bu_h \Vert_{0,4,\Omega} \Vert\nabla c_{i,h} \Vert_{0,\Omega} \Vert w_{i,h} \Vert_{1,\Omega}\lesssim \Vert\bu_h \Vert_{0,4,\Omega} \Vert c_{i,h} \Vert_{h,d} \Vert w_{i,h} \Vert_{h,d}.
\end{align*}
This last inequality implies that
\begin{equation*}
|\mathcal{K}_h^{\ast}(\bu_h;c_{i,h},w_{i,h})|\lesssim \Vert\bu_h \Vert_{0,4,\Omega} \Vert c_{i,h} \Vert_{h,d} \Vert w_{i,h} \Vert_{h,d}.
\end{equation*}
It is easy to deduce from the definitions of $\mathcal{A}_{i,h}$, $\mathcal{B}$ that they are continuous on $V_h^k(\Omega)\times V_h^k(\Omega)$ and that $\mathcal{C}_{i,h}$ is continuous on $V_h^k(\Omega)^{n_c}\times V_h^k(\Omega)$.
The continuity of $\mathcal{D}_h$ follows from the respective continuities of $\mathcal{A}_{i,h}$, $\mathcal{B}$, $\mathcal{C}_{i,h}$ and $\mathcal{K}_h^{\ast}$.

Regarding coercivity, note that by the construction of the VEM form and the definition of $\mathcal{K}_h^{\ast}$, we have
\begin{align*}
\mathcal{D}_h(\bw_h,\bw_h)&=\sum_{i=1}^{n_c}\Big[\mathcal{A}_{i,h}(w_{i,h},w_{i,h})+\mathcal{K}_h^{\ast}(\bu_h;w_{i,h},w_{i,h})+\mathcal{B}(\bu_h;w_{i,h},w_{i,h})+\mathcal{C}_{i,h}(\bw_h,w_{i,h})\Big] \\
&=\sum_{i=1}^{n_c}[\mathcal{A}_{i,h}(w_{i,h},w_{i,h})+\mathcal{B}(\bu_h;w_{i,h},w_{i,h})+\mathcal{C}_{i,h}(\bw_h,w_{i,h})]\\
&=\sum_{i=1}^{n_c}\sum_{K\in \Omega_h}(\mathcal{A}^{K}_{i,h}(w_{i,h},w_{i,h})+\dfrac{1}{2}(|f|\Pi_k^0 w_{i,h}, \Pi_k^0 w_{i,h})_{0,K}-(R_i(\boldsymbol{\Pi}_k^{0}\bw_h),\Pi_k^0 w_{i,h})_{0,K})\\
& \qquad +\sum_{i=1}^{n_c}\sum_{e\in \mathcal{E}_h}(|\bu_h\cdot \bn|w_{i,h},w_{i,h})_{0,K}.
\end{align*}

\corr{Therefore, employing inequality~\eqref{Ric}, in estimating the term containing $R_i$, it follows that}
\begin{align*}
\mathcal{D}_h(\bw_h,\bw_h)&\geq \sum_{i=1}^{n_c}\left(\sum_{K\in \Omega_h} |w_{i,h}|_{1,K}^2+\dfrac{1}{2}\Vert |f|^{1/2}\Pi_k^0 w_{i,h}\Vert_{0,K}^2+\sum_{e\in \mathcal{E}_h}\Vert|\bu_h\cdot \bn|^{1/2}w_{i,h}\Vert_{0,e}^2\right)\\
& \qquad  - \sum_{K\in \Omega_h} \sum_{i=1}^{n_c}(R_i(\boldsymbol{\Pi}_k^{0}\bw_h),\Pi_k^0 w_{i,h})_{0,K}\\
&\geq \frac{1}{2}\Vert \bw_h\Vert_h^2-\mu\sum_{K\in \Omega_h}\Vert \boldsymbol{\Pi}_k^{0}\bw_h\Vert_{0,K}^2\geq  \frac{1}{2}\Vert \bw_h\Vert_h^2-\mu \mu^{\ast}\sum_{K\in \Omega_h}\Vert\bw_h\Vert_{0,K}^2,
\end{align*}

where $\mu^{\ast}$ \corr{is the continuity constant of $\Pi^0$, see Lemma~\ref{lem:L2PLp0}}. 
From this last inequality, we deduce that
\[
\mathcal{D}_h(\bw_h,\bw_h)+\mu \mu^{\ast}\sum_{K\in \Omega_h}\Vert\bw_h\Vert_{0,K}^2\geq \frac{1}{2}\Vert \bw_h\Vert_h^2.
\]
This establishes coercivity and therefore completes the proof.
\end{proof}

We are now ready to establish the existence and uniqueness of the semi-discrete solution and, 
subsequently, to derive the associated convergence rates. 
In particular, by leveraging Theorem \ref{Funtional} and the properties established in Lemma \ref{PropertiesBF}, 
we obtain the existence and uniqueness of the solution to Problem \eqref{eqn:varFormsemi}, 
as stated in the following theorem.

\begin{thm} There exist a unique solution $\bc_h:=(c_{1,h},\dots,c_{n_c,h})$ to \corr{Problem \eqref{eqn:varFormsemi}} stisfying
\[
\max_{t\in (0,T]} \sum_{K\in \Omega_h}\hspace{-0.3em}\Vert\bc_{h}(t)\Vert_{0,K}^2+\int_{0}^T\hspace{-0.8em}\Vert \bc_h\Vert_h^2 \, d\bx \lesssim  \hspace{-0.2em}\sum_{K\in \Omega_h}\hspace{-0.4em}\Vert\bc_{h}(0)\Vert_{0,K}^2+\sum_{i=1}^{n_c}\int_{0}^T\hspace{-0.1em}\left(\hspace{-0.1em}\sum_{K\in \Omega_h}\Vert |f|^{1/2}\Tilde{c}_{i,h}\Vert_{0,K}^2+\hspace{-0.2em}\sum_{e\in \mathcal{E}_h^I}\Vert|\bu\cdot \bn|^{1/2}c_{i}\Vert_{0,e}^2\right).
\]
\end{thm}

\corr{The main result regarding the convergence of the semi--discrete problem is summarized 
in the following theorem, which provides an upper bound on the discretization error.}

\begin{thm}\label{errorsemi}
Let $\bc$ be the solution of \eqref{eqn:varForm} and let $\bc_h$ be the solution of \eqref{eqn:varFormsemi}. 
Furthermore, we assume that $\bc_I\in L^{4}((0,T);L^{4}(\Gamma_I))^{n_c}$, $f\tilde{\bc}\in L^2((0,T);H^s(\Omega))^{n_c}$, $\bc\in H^1((0,T);H^{1+s}(\Omega))^{n_c}$ and $\bu \in [L^4((0,T);H^{1+s}(\Omega))]^d$. Then,
\begin{equation}\label{errorsemd}
\sup_{t\in(0,T]}\Vert (\bc-\bc_h)(t,\cdot)\Vert_{L^2(\Omega)}^2+\int_{0}^T||| \bc-\bc_h|||^2  \lesssim C(\bu,\bc)h^{2s},
\end{equation}
where $C(\bu,\bc)$ denotes a positive constant that depends on the solution and the problem data.
\end{thm}
\begin{proof}
Let $\bc_{\pi}:=((c_{1})_{\pi},\dots,(c_{n_c})_{\pi})$, $\zz_h:=(z_{1,h},\dots,z_{n_c,h}):=\bc_h-\bc_{\pi}$ and $\epsilon>0$, then 
\[
m_h(\partial_t \zz_h,\zz_h)+\mathcal{D}_h(\bu_h;\zz_h,\zz_h)=\mathcal{F}_h^{\Omega}(\zz_h)+\mathcal{F}_h^{\Gamma}(\zz_h)-(\partial_t \bc_{\pi},\boldsymbol{\Pi}_k^{0}\zz_h)_{0,\Omega}-\mathcal{D}(\boldsymbol{\Pi}_k^{0}\bu_h;\bc_{\pi},\boldsymbol{\Pi}_k^{0}\zz_h).
\]
Using equation \eqref{revarf}, we can rewrite the previous identity as
\begin{equation}\label{decompotition}
m_h(\partial_t \zz_h,\zz_h)+\mathcal{D}_h(\bu_h;\zz_h,\zz_h)=T^{\Omega}+T^{\Gamma}+T^{m}+T^{\mathcal{D}},
\end{equation}
where
\begin{align*}
T^{\Omega}&:=\mathcal{F}_h^{\Omega}(\zz_h)-\mathcal{F}^{\Omega}(\zz_h),\quad
T^{\Gamma}:=\mathcal{F}_h^{\Gamma}(\zz_h)-\mathcal{F}^{\Gamma}(\zz_h),\\
T^{m}&:=(\partial_t \bc,\zz_h)_{0,\Omega}-(\partial_t \bc_{\pi},\boldsymbol{\Pi}_k^{0}\zz_h)_{0,\Omega},\quad \text{and}\quad T^{\mathcal{D}}:= \mathcal{D}(\bu;\bc,\zz_h)-\mathcal{D}(\boldsymbol{\Pi}_k^{0}\bu_h;\bc_{\pi},\boldsymbol{\Pi}_k^{0}\zz_h).
\end{align*}
To conclude the proof, it remains to bound each term in identity \eqref{decompotition} appropriately.
\begin{itemize}
    \item {\bf Term $T^{m}$:}
    Using the properties established in \corr{Lemmas \ref{lem:L2PLp0}} and \ref{polyproyect}, we obtain that
\begin{align*}
T^{m}&=( \partial_t\bc-(\partial_t\bc)_{\pi},\boldsymbol{\Pi}_k^{0}\zz_h)_{0,\Omega}+((I-\boldsymbol{\Pi}_k^{0})\partial_t\bc,\zz_h)_{0,\Omega}\lesssim \big[ \Vert \partial_t\bc-(\partial_t\bc)_{\pi}\Vert_{0,\Omega}+\Vert(I-\boldsymbol{\Pi}_k^{0})\partial_t\bc \Vert_{0,\Omega}\big] \Vert \zz_h\Vert_{0,\Omega}\\
&\lesssim h^{s}|\partial_t \bc|_{s,\Omega}\Vert \zz_h\Vert_{0,\Omega}.
\end{align*}
Then, applying Young’s inequality, we obtain
\begin{equation}\label{Tm}
T^{m}\lesssim h^{2s}|\partial_t \bc|_{s,\Omega}^2+\Vert \zz_h\Vert_{0,\Omega}^2
\end{equation}

\item {\bf Term $T^{\Omega}$:}
Using the approximation properties of the $L^2$-projection, we get
\begin{align*}
T^{\Omega}&=-(\tilde{f}\tilde{\bc},\zz_h)_{\Omega}+(\tilde{f}\tilde{\bc},\boldsymbol{\Pi}_{k}^{0}\zz_h)_{\Omega}=\sum_{K\in \Omega_h}(\tilde{f}\tilde{\bc},(-I+\boldsymbol{\Pi}_{k}^{0})\zz_h)_{K}\\
  &\lesssim \sum_{K\in \Omega_h}\Vert (I-\boldsymbol{\Pi}_{k}^{0})(\tilde{f}\tilde{\bc})\Vert_{0,K}\Vert \zz_h\Vert_{0,K}\lesssim \sum_{K\in \Omega_h} h_{K}^s|\tilde{f}\tilde{\bc}|_{s,K}\Vert \zz_h\Vert_{0,K}\lesssim h^s||f|\tilde{\bc}|_{s,\Omega}\Vert \zz_h\Vert_{0,\Omega},
\end{align*}
where $\tilde{f}:=\max\{0,f\}$. 
Applying Young’s inequality, we obtain
\begin{equation}\label{Tomega}
T^{\Omega}\lesssim h^{2s}|\tilde{f}\tilde{\bc}|_{s,\Omega}^2+\Vert \zz_h\Vert_{0,\Omega}^2.
\end{equation}

\item{\bf Term $T^{\Gamma}$:}

\corr{
Employing H\"older's inequality and the same reasoning used in inequalities~\eqref{boundaryinequality} and~\eqref{boundaryinequality2} to derive estimate~\eqref{estimtestar}, we obtain
} 
\begin{align*}
T^{\Gamma}&= -\sum_{e\in \mathcal{E}^{I}}([|\bu_h\cdot \bn|-|\bu\cdot \bn|]\bc_I,\zz_h)_e \lesssim \sum_{e\in \mathcal{E}^{I}}\Vert \bu_h-\bu\Vert_{0,e} \Vert \bc_I\Vert_{0,4,e} \Vert \zz_h \Vert_{0,4,e} \\
&\corr{\lesssim \left(\sum_{e\in \mathcal{E}^{I}}\Vert \bu_h-\bu\Vert_{0,e}^2\right)^{1/2} \left(\sum_{e\in \mathcal{E}^{I}} \Vert \bc_I\Vert_{0,4,e}^4\right)^{1/4}\left(\sum_{e\in \mathcal{E}^{I}} \Vert \zz_h \Vert_{0,4,e}^4\right)^{1/4}}\\
& \corr{\lesssim  h^{s+\frac{1}{2}}\Vert \bc_I\Vert_{0,4,\Gamma_{I}} |\bu|_{s+1,\Omega} \Vert \zz_h\Vert_{1,\Omega}\lesssim h^{s+\frac{1}{2}}\Vert \bc_I\Vert_{0,4,\Gamma_{I}} |\bu|_{s+1,\Omega} \Vert \zz_h\Vert_{h}}.
\end{align*}
Applying Young’s inequality, we obtain
\begin{equation}\label{Tgamma}
T^{\Gamma}-\epsilon \Vert \zz_h\Vert_{h}^2\lesssim h^{2s+1}\Vert \bc_I\Vert_{0,4,\Gamma_{I}}^2 |\bu|_{s+1,\Omega}^2
\end{equation}

\item {\bf Term $T^{\mathcal{D}}$:}

In order to carry out the analysis of this term, 
we rewrite it in the following form: 
\[
T^{\mathcal{D}}=T^{\mathcal{D}}_1+T^{\mathcal{D}}_2+T^{\mathcal{D}}_3+T^{\mathcal{D}}_4+T^{\mathcal{D}}_5,
\]
where 
\begin{align*}
T^{\mathcal{D}}_1&:=\sum_{i=1}^{n_c}D_i\big[(\nabla c_{i}, \nabla z_{i,h})_{\Omega}-(\nabla (c_{i})_{\pi}, \nabla \Pi_{k}^{\nabla} z_{i,h})_{\Omega}\big], \,\, T^{\mathcal{D}}_2:= \sum_{i=1}^{n_c}\big[ -(R_i(\bc), z_{i,h})_{\Omega}+(R_i(\bc_{\pi}),\Pi_k^0 z_{i,h})_{\Omega}\big],\\
T^{\mathcal{D}}_3&:= (|\bu \cdot \bn|\bc,\zz_{h})_{\Gamma}-(|\bu_h\cdot \bn|;\bc_{\pi},\zz_{h})_{\Gamma},\,\, T^{\mathcal{D}}_4:=(|f|\bc,\zz_h)_{\Omega}-(|f|\bc_{\pi},\boldsymbol{\Pi}_{k}^{0}\zz_h)_{\Omega} \,\,\, \text{and} \\
T^{\mathcal{D}}_5&:= \dfrac{1}{2}\sum_{i=1}^{n_c}\big[(\bu \cdot \nabla c_i,z_{i,h})_{\Omega}-(\bu \cdot \nabla z_{i,h},c_i)_{\Omega}-(\boldsymbol{\Pi}_k^{0}\bu_h \cdot  \nabla (c_i)_{\pi},\Pi_k^{0}z_{i,h})_{\Omega}+(\boldsymbol{\Pi}_k^{0}\bu_h \cdot \boldsymbol{\Pi}_k^{0} \nabla z_{i,h},(c_i)_{\pi})_{\Omega}\big].\\
\end{align*}

We begin by estimating the term $T^{\mathcal{D}}_1$. 
The following estimate follows from H\"older’s inequality, the approximation properties of the $\Pi^{\nabla}_k$-projection, and Lemma \ref{polyproyect}.
\begin{align*}
T^{\mathcal{D}}_1&=\sum_{i=1}^{n_c}D_i\big[(\nabla (I-\Pi_{k}^{\nabla})c_{i}, \nabla z_{i,h})_{\Omega}+(\nabla (c_i-(c_{i})_{\pi}), \nabla \Pi_{k}^{\nabla} z_{i,h})_{\Omega}\big]\\
&\lesssim \sum_{i=1}^{n_c}D_i\big[|(I-\Pi_{k}^{\nabla})c_{i}|_{1,\Omega}+|(c_i-(c_{i})_{\pi}|_{1,\Omega}\big]|z_{i,h}|_{1,\Omega}\\
&\lesssim \sum_{i=1}^{n_c}D_i h^s |c_i|_{s+1,\Omega}|z_{i,h}|_{1,\Omega} \lesssim \sum_{i=1}^{n_c}D_i h^s |c_i|_{s+1,\Omega}\Vert z_{i,h}\Vert_{h,d}\\
&\lesssim  h^s |\bc|_{s+1,\Omega}\Vert \zz_{h}\Vert_{h},
\end{align*}
and therefore
\[
T^{\mathcal{D}}_1-\epsilon\Vert\zz_{h}\Vert_{h}^2\lesssim  h^{2s} |\bc|_{s+1,\Omega}^2
\]

Regarding the term $T^{\mathcal{D}}_2$, by virtue of H\"older's inequality, 
the approximation properties of the $L^2$-projection  and Lemma \ref{polyproyect}, we infer that
\begin{align*}
T^{\mathcal{D}}_2&=\sum_{i=1}^{n_c}\big[ (R_i(\bc_{\pi}-\bc), \Pi_k^0 z_{i,h})_{K}+((\boldsymbol{\Pi}_{k}^{0}-I)R_i(\bc_{\pi}),z_{i,h})_{K}\big]\\
&\lesssim  \sum_{i=1}^{n_c}\big[ \Vert\bc_{\pi}-\bc\Vert_{0,\Omega}+\Vert (\boldsymbol{\Pi}_{k}^{0}-I)R_i(\bc_{\pi})\Vert_{0,\Omega}\big]\Vert z_{i,h}\Vert_{0,\Omega}\\
&\lesssim h^{s+1}|\bc|_{s+1,\Omega}\Vert \zz_{h}\Vert_{0,\Omega}.
\end{align*}
Consequently,
\[
T^{\mathcal{D}}_2\lesssim  h^{2s+2}|\bc|_{s+1,\Omega}^2+\Vert \zz_{h}\Vert_{0,\Omega}^2.
\]

Concerning the term $T^{\mathcal{D}}_3$, proceeding analogously to the treatment of $T^{\Gamma}$, we infer that 
\begin{align*}
T^{\mathcal{D}}_3&=\sum_{e\in \mathcal{E}_h}([|\bu_h\cdot \bn|-|\bu\cdot \bn|]\bc_{\pi},\zz_h)_e+(|\bu\cdot \bn|(\bc-\bc_{\pi}),\zz_h)_e]\\
&\lesssim \corr{h^{s+\frac{1}{2}}}[\Vert \bc\Vert_{0,4,\Gamma} |\bu|_{s+1,\Omega}+\Vert \bu\Vert_{0,4,\Gamma} |\bc|_{s+1,\Omega}] \Vert \zz_h\Vert_{h}.
\end{align*}
Hence,
\[
T^{\mathcal{D}}_3-\epsilon \Vert \zz_h\Vert_{h}^2\lesssim \corr{h^{2s+1}}[\Vert \bc\Vert_{0,4,\Gamma}^2 |\bu|_{s+1,\Omega}^2+\Vert \bu\Vert_{0,4,\Gamma}^2 |\bc|_{s+1,\Omega}^2].
\]
For the term $T^{\mathcal{D}}_4$, using H\"older's inequality, the approximation properties of the $L^2$-projection, and Lemma \ref{polyproyect}, we obtain that
\begin{align*}
T_4^{\mathcal{D}}&=\sum_{K\in \Omega_h}[(|f|\bc,\zz_h)_{K}-(|f|\bc_{\pi},\boldsymbol{\Pi}_{k}^{0}\zz_h)_{K}]=\sum_{K\in \Omega_h}[(|f|(\bc-\bc_{\pi}),\zz_h)_{K}+(|f|\bc_{\pi},(I-\boldsymbol{\Pi}_{k}^{0})\zz_h)_{K}]\\
       &=\sum_{K\in \Omega_h}[(|f|(\bc-\bc_{\pi}),\zz_h)_{K}+(|f|(\bc_{\pi}-\bc),(I-\boldsymbol{\Pi}_{k}^{0})\zz_h)_{K}+(|f|\bc,(I-\boldsymbol{\Pi}_{k}^{0})\zz_h)_{K}]\\
       &=\sum_{K\in \Omega_h}[(|f|(\bc-\bc_{\pi}),\zz_h)_{K}+(|f|(\bc_{\pi}-\bc),(I-\boldsymbol{\Pi}_{k}^{0})\zz_h)_{K}+((I-\boldsymbol{\Pi}_{k}^{0})(|f|\bc),\zz_h)_{K}]\\
       &\lesssim \sum_{K\in \Omega_h}[\Vert f\Vert_{0,4,K}\Vert \bc-\bc_{\pi}\Vert_{0,K}\Vert \zz_h\Vert_{0,4,K}+\Vert (I-\boldsymbol{\Pi}_{k}^{0})(|f|\bc)\Vert_{0,K}\Vert \zz_h\Vert_{0,K}]\\
       & \lesssim \sum_{K\in \Omega_h}[h_{K}^{s}\Vert f\Vert_{0,4,K}|\bc|_{s,K}\Vert \zz_h\Vert_{0,4,K}+h_{K}^s||f|\bc|_{s,K}\Vert \zz_h\Vert_{0,K}] \\
       & \lesssim h^{s}[\Vert f\Vert_{0,4,\Omega}|\bc|_{s,\Omega}\Vert \zz_h\Vert_{0,4,\Omega}+||f|\bc|_{s,\Omega}\Vert \zz_h\Vert_{0,\Omega}] \lesssim h^{s}[\Vert f\Vert_{0,4,\Omega}|\bc|_{s,\Omega}+||f|\bc|_{s,\Omega}]\Vert \zz_h\Vert_{1,\Omega}\\
       &\lesssim h^{s}[\Vert f\Vert_{0,4,\Omega}|\bc|_{s,\Omega}+||f|\bc|_{s,\Omega}]\Vert \zz_h\Vert_{h}
\end{align*}
From the previous inequality, we have that
\[
T_4^{\mathcal{D}}-\epsilon\Vert \zz_h\Vert_{h}^2 \lesssim h^{2s}[\Vert f\Vert_{0,4,\Omega}^2|\bc|_{s,\Omega}^2+||f|\bc|_{s,\Omega}^2].
\]
To bound the term $T^{\mathcal{D}}_5$, we begin by deriving estimates for the term 
\[
T^{\ast}:=\sum_{i=1}^{n_c}\big[(\bu\cdot \nabla c_{i},z_{i,h})_{\Omega}-(\boldsymbol{\Pi}_{k}^0\bu_h\cdot \nabla (c_{i})_{\pi},\Pi_{k}^{0}z_{i,h})_{\Omega}\big],
\]
namely,
\begin{align*}
T^{\ast}&=\sum_{i=1}^{n_c}\big[((I-\Pi_{k}^0)(\bu\cdot \nabla c_{i}),z_{i,h})_{\Omega}+(\bu\cdot \nabla( c_{i}-(c_i)_{\pi}),\Pi_k^0 z_{i,h})_{\Omega}\\
&\qquad +((I-\boldsymbol{\Pi}_{k}^0)\bu\cdot \nabla (c_{i})_{\pi},\Pi_{k}^0 z_{i,h})_{\Omega}+(\boldsymbol{\Pi}_{k}^0(\bu-\bu_h)\cdot \nabla c_{i}),\Pi_k^0 z_{i,h})_{\Omega}\big]
\end{align*}
Applying H\"older's inequality, the Lemma \ref{traceineq}, the Lemma \ref{polyproyect} ,the approximation properties of the $L^2$-projection, and standard Sobolev embeddings, we obtain
\begin{align*}
T^{\ast}&\lesssim \sum_{i=1}^{n_c}\big[\Vert (I-\Pi_{k}^0)(\bu\cdot \nabla c_{i})\Vert_{0,\Omega}+\Vert \bu \Vert_{0,4,\Omega}|c_i-(c_i)_{\pi}|_{1,\Omega}\\
\qquad &+\Vert (I-\boldsymbol{\Pi}_{k}^0)\bu\Vert_{0,\Omega}\Vert \nabla (c_i)_{\pi}\Vert_{0,4,\Omega}+\Vert \bu-\bu_h\Vert_{0,4,\Omega}\Vert \nabla (c_i)_{\pi}\Vert_{0,\Omega}\big] \Vert z_{i,h}\Vert_{1,\Omega}\\
&\lesssim \sum_{i=1}^{n_c}\big[h^s |\bu\cdot \nabla c_{i}|_{s,\Omega}+h^s\Vert \bu \Vert_{0,4,\Omega}|c_i|_{s+1,\Omega}+h^{s+1}h^{-\frac{d}{4}}|\bu|_{s+1,\Omega}|(c_i)_{\pi}|_{1,\Omega}+h^s|\bu|_{s+1,\Omega}|c_i|_{1,\Omega}\big]\Vert z_{i,h}\Vert_{h,d}\\
&\lesssim \sum_{i=1}^{n_c} h^s \big[|\bu\cdot \nabla c_{i}|_{s,\Omega}+\Vert \bu \Vert_{0,4,\Omega}|c_i|_{s+1,\Omega}+h^{\frac{4-d}{4}}|\bu|_{s+1,\Omega}|(c_i)_{\pi}|_{1,\Omega}+|\bu|_{s+1,\Omega}|c_i|_{1,\Omega}\big]\Vert z_{i,h}\Vert_{h,d}\\
&\lesssim h^s \big[| (\nabla \bc)\bu|_{s,\Omega}+\Vert \bu \Vert_{0,4,\Omega}|\bc|_{s+1,\Omega}+h^{\frac{4-d}{4}}|\bu|_{s+1,\Omega}|\bc|_{1,\Omega}+|\bu|_{s+1,\Omega}|\bc|_{1,\Omega}\big]\Vert \zz_{h}\Vert_{h}
\end{align*}
By employing arguments analogous to those used for the term $T^{\ast}$, we deduce that 
\[
T^{\circ}\lesssim h^s\big[ |\bu \bc^{\mathrm{t}}|_{s,\Omega}+|\bu|_{s+1,\Omega}\Vert \bc\Vert_{1,\Omega}+\Vert \bu\Vert_{0,4,\Omega}|\bc|_{s+1,\Omega}+|\bu|_{s+1,\Omega}\Vert \bc\Vert_{1,\Omega}\big]\Vert \zz_{h}\Vert_{h},
\]
where 
\[
T^{\circ}:=\sum_{i=1}^{n_c}\big[(\bu\cdot \nabla z_{i,h},c_{i})_{\Omega}-(\boldsymbol{\Pi}_{k}^0\bu_h\cdot \boldsymbol{\Pi}_{k}^{0}\nabla z_{i,h},(c_{i})_{\pi})_{\Omega}\big].
\]
Using the estimates derived for the terms $T^{\ast}$ and $T^{\circ}$, we can infer that
\[
T_5^{\mathcal{D}}\leq h^s\big[ | (\nabla \bc)\bu|_{s,\Omega}+|\bu \bc^{\mathrm{t}}|_{s,\Omega}+|\bu|_{s+1,\Omega}\Vert \bc\Vert_{1,\Omega}+\Vert \bu\Vert_{0,4,\Omega}|\bc|_{s+1,\Omega}+|\bu|_{s+1,\Omega}\Vert \bc\Vert_{1,\Omega}\big]\Vert \zz_{h}\Vert_{h}
\]
From the above, it follows that
\[
T_5^{\mathcal{D}}-\epsilon\Vert \zz_{h}\Vert_{h}^2\lesssim h^{2s}\big[ | (\nabla \bc)\bu|_{s,\Omega}^2+|\bu \bc^{\mathrm{t}}|_{s,\Omega}^2+|\bu|_{s+1,\Omega}^2\Vert \bc\Vert_{1,\Omega}^2+\Vert \bu\Vert_{0,4,\Omega}^2|\bc|_{s+1,\Omega}^2+|\bu|_{s+1,\Omega}^2\Vert \bc\Vert_{1,\Omega}^2\big].
\]
\end{itemize}
By employing all the estimates previously derived for each term in the decomposition shown in \eqref{decompotition}, we can deduce that
\[
m_h(\partial_t \zz_h,\zz_h)+\mathcal{D}_h(\bu_h;\zz_h,\zz_h)-4\epsilon \Vert \zz_{h}\Vert_{h}^2 \lesssim \tilde{C}(\bu,\bc)h^{2s}+\Vert \zz_{h}\Vert_{0,\Omega}^2,
\]
where 
\begin{gather*}
\tilde{C}(\bu,\bc):=| (\nabla \bc)\bu|_{s,\Omega}^2+|\bu \bc^{\mathrm{t}}|_{s,\Omega}^2+|\bu|_{s+1,\Omega}^2\Vert \bc\Vert_{1,\Omega}^2+\Vert \bu\Vert_{0,4,\Omega}^2|\bc|_{s+1,\Omega}^2+|\bu|_{s+1,\Omega}^2\Vert \bc\Vert_{1,\Omega}^2\\
\qquad \qquad\qquad\qquad +|\tilde{f}\tilde{\bc}|_{s,\Omega}^2+\Vert \bc \Vert_{0,4,\Gamma}^2|\bu|_{s+1,\Omega}^2+\Vert \bu\Vert_{0,4,\Gamma}^2|\bc|_{s+1,\Omega}^2+\Vert f\Vert_{0,4,\Omega}^2|\bc|_{s,\Omega}^2+||f|\bc|_{s,\Omega}^2 
\end{gather*}
Utilizing the coercivity properties of $\mathcal{D}_h$ and integrating over $t$ \corr{ with the choice $\epsilon=\frac{1}{8}$}, it follows that
\[
\Vert \zz_h(t)\Vert_{0,\Omega}+\int_{0}^t\Vert \zz_{h}(s)\Vert_{h}^2\,ds \lesssim h^{2s}\int_{0}^{t}\tilde{C}(\bu,\bc)\,ds+\int_{0}^{t}\Vert \zz_{h}(s)\Vert_{0,\Omega}^2\, ds+\Vert \zz_h(0)\Vert_{0,\Omega}
\]
Subsequently, applying Gr\"onwall's inequality to the previous estimate, it follows that
\[
\Vert \zz_h(t)\Vert_{0,\Omega}+\int_{0}^t\Vert \zz_{h}(s)\Vert_{h}^2\,ds \lesssim h^{2s}\int_{0}^{t}\tilde{C}(\bu,\bc)\,ds+\Vert \zz_h(0)\Vert_{0,\Omega}
\]
Ultimately, the preceding inequality enables us to conclude the theorem.
\end{proof}

%%%%%%%%%%%%%%%%%%%%%%%%%%%%%% english checked!!
\subsection{Fully discrete problem}

In this subsection, we present the fully discrete formulation of the proposed problem. 
Since we employ a Time-DG discretization based on Gauss--Radau quadrature points, 
we first briefly introduce the necessary notation and the key properties of this temporal scheme.
Then, in Section~\ref{sec:fully}, we establish the existence and uniqueness of the solution and 
derive the corresponding error estimates for the fully discrete system.

\subsubsection{Gauss--Radau quadrature and space--time discretization}

Given an interval $J\subset \mathbb{R}$, we denote by $\mathbb{P}_q(J)$ the space of polynomials of degree at most $q$.
We consider the following Gauss–Radau quadrature formula on the reference interval $I_{-1}:=(0,1]$:
\corr{
\begin{equation}\label{GR01}
GR_{I_{-1}}(\varphi):= \sum_{i=1}^{q+1}\omega_i \varphi(\xi_i),
\end{equation}
}
where $0<\xi_1<\cdots<\xi_{q+1}=1$ are the Radau integration points and $\omega_i>0$ are the Radau weights. 
Using the isomorphism of $(0,1]$ with $I_n = (t_{n-1},t_n]$, the formula \eqref{GR01} can be rewritten as
\corr{
\begin{equation}\label{GR}
GR_{I_n}(\varphi):=\tau_n \sum_{i=1}^{q+1}\omega_i \varphi(t_{n,i}),
\end{equation}
}
where $t_{n,i}:=t_{n-1}+\tau_{n}\xi_i$.
It is worth noticing that the quadrature rules provided in Equations~\eqref{GR01} and, 
consequently,~\eqref{GR} integrate exactly all polynomials of degree at most $2q$; \corr{that is, 
\[
GR_{I_n}(\varphi)=\int_{I_n} \varphi\, dt, \qquad \forall \varphi \in \mathrm{P}_{2q}(I_n).
\]
}
Starting from this time discretization, 
we define the time-discrete spaces 
$$
 V_{\tau}^q(I):=\left\{ v_{\tau}\in L^2(I): v_{\tau}|_{I_n}\in \mathbb{P}_q(I_n)\, \forall n=1,...,N\right\},
$$
\corr{
and the global space-time VE--DG space 
$$
V_{\tau,h}^{q,k}(I\times \Omega):=\left\{ v_{\tau,h}=\sum_{i=0}^{r}\theta_{i,\tau}\beta_{i,h}:I\times \Omega \rightarrow \mathbb{R}: r\in \mathbb{N},\,(\theta_{i,\tau},\beta_{i,h}) \in  V_h^{k}(\Omega)\times  V_{\tau}^q(I),\,\, i=1,\dots,l\right\}.
$$
\begin{rem}
Given a basis $\{\varphi_i\}_{i=1}^l$ for $V_h^{k}(\Omega)$, any function $v_{\tau,h} \in V_{\tau,h}^{q,k}(I\times \Omega)$ can be uniquely expanded as
$$
v_{\tau,h}(t,\bx)=\sum_{i=0}^{l}\theta_{i,\tau}(t)\varphi_i(\bx),
$$
where $\theta_{i,\tau} \in V_{\tau}^q(I)$ for $i=1,\dots,l$.
\end{rem}
}
Denote by $v^n$ the restriction of $v$ to the time slab $I_n \times \Omega $. 
Note that these functions may be discontinuous from one slab to the next. 
Therefore, we denote the limits by $v^{\mp,n}=\lim_{s\to 0^{\mp}}v({\bf x},t_n+s)$ and 
the jump by $\{ v\}^{n}=v^{+,n}-v^{-,n}$.

The following results and notations for the Time-DG scheme are essential for the subsequent analysis; 
for the proofs, we refer the reader to~\cite{refId0}.

\begin{lemm}\label{identitiesGR}
Let $v\in \mathbb{P}_{q}(I_n)$ and $\mathcal{L}_{\tau}v\in \mathbb{P}_q(I_n)$ be the Lagrange interpolation of the function $\tau_n(t-t_{n-1})^{-1}v$ at the points $t_{n,i}$, $i=1,...,q+1$:
$$
\mathcal{L}_{\tau}v(t_{n,i})=\tau_n(t_{n,i}-t_{n-1})^{-1}v(t_{n,i}), \qquad i=1,...,q+1.
$$
Then
\begin{equation}\label{Lagrangeint}
    \int_{I_n}\partial_t v \mathcal{L}_{\tau}v\, dt+v(t_{n-1})\mathcal{L}_{\tau}v(t_{n-1})=\frac{1}{2}\left( v^2(t_n)+\sum_{i=1}^{q+1}\omega_i\xi_{i}^{-1}v^2(t_{n,i})\right).
\end{equation}    
\end{lemm}

\begin{lemm}
For all $v\in \mathbb{P}_{q}(I_n)$, the following inequalities holds:
\begin{align}\label{minusineq}
    (\mathcal{L}_{\tau}v^{+,n-1})^2\lesssim \frac{1}{\tau_n}\int_{I_n}v^2\,dt\quad \text{and}\quad \corr{(v^{+,n-1})^2\lesssim \frac{1}{\tau_n}\int_{I_n}v^2\,dt}
\end{align}
\end{lemm}

\corr{
By exploiting the aforementioned properties of the Gauss–-Radau quadrature, 
we establish the norm equivalence presented in the following lemma, 
which will be fundamental for the subsequent analysis of the fully discrete problem.
}
\begin{lemm}
Suppose $\bu \in  L^{\infty}((0,T);(H^{1+s}(\Omega))^{d})$ and $f\in L^{\infty}((0,T);L^4(\Omega))$. Then there exist $h_0>0$ such that for all $h\leq h_0$, we have
\corr{
\begin{equation}\label{eqnormtime}
    \Vert v_{\tau,h}\Vert_{L^2(I_n;H^{1}(\Omega))}^2 \approx \tau_n\sum_{i=1}^{q+1}\omega_i \Vert v_{\tau,h}(t_{n,i})\Vert_{d,h}^2,
\end{equation}
for all $v_{\tau,h}\in V_{\tau,h}^{q,k}(I\times \Omega)$, $n=1,...,N$.
}
\end{lemm}
\begin{proof}
Using lemma~\ref{eqnormd} we have that
\corr{
\begin{equation}\label{localeqnorm}
\Vert v_{\tau,h}(t_{n,i})\Vert_{1,\Omega}^2 \approx \Vert v_{\tau,h}(t_{n,i})\Vert_{d,h}^2.
\end{equation}
Since $v_{\tau,h}$ belongs to the space $ V_{\tau,h}^{q,k}(I\times \Omega)$, we have 
$v_{\tau,h} = \sum_{j=1}^l \theta_{j,\tau} \varphi_j$, 
where $l:=dim(V_h^{k}(\Omega))$, $\{\varphi_j\}_{j=1}^l$ is basis for $V_h^{k}(\Omega)$, and $\theta_{j,\tau} \in V_{\tau}^q(I)$ for $j=1,\dots,l$.
}

\corr{
Since the Gauss quadrature rule is exact for polynomials of degree less than or equal to $2q$, we have
\begin{equation}\label{exacttheta}
GR_{I_n}(\theta_{i,\tau}\theta_{j,\tau})=\int_{I_n}\theta_{i,\tau}\theta_{j,\tau}\, dt, \qquad  i,j=1,\dots,l.
\end{equation}
Hence, employing equality \eqref{exacttheta}, we have 
\begin{align}\nonumber
  \Vert v_{\tau,h}\Vert_{L^2(I_n;L^{2}(\Omega))}^2&=\int_{I_n}\int_{\Omega}  \left(\sum_{i,j=1}^l \theta_{i,\tau}\theta_{j,\tau} \varphi_i \varphi_j\right)\,d\bx\,dt = \int_{\Omega}  \left( \sum_{i,j=1}^l \left(\int_{I_n}\theta_{i,\tau}\theta_{j,\tau}\,dt\right)\; \varphi_i \varphi_j\right)\,d\bx\\
  \nonumber
  &=\int_{\Omega}  \left( \sum_{i,j=1}^l \left(\tau_n\sum_{r=1}^{q+1}\omega_{r}\theta_{i,\tau}(t_{n,r})\theta_{j,\tau}(t_{n,r})\right)\; \varphi_i \varphi_j\right)\,d\bx\\
  \nonumber
  &=\int_{\Omega}   \left(\tau_n\sum_{r=1}^{q+1}\omega_{r}\sum_{i,j=1}^l \theta_{i,\tau}(t_{n,r})\theta_{j,\tau}(t_{n,r}) \varphi_i \varphi_j\right)\,d\bx\\
  \nonumber
  &=\int_{\Omega}   \left(\tau_n\sum_{r=1}^{q+1}\omega_{r}\left(\sum_{i=1}^l \theta_{i,\tau}(t_{n,r}) \varphi_i\right) \left(\sum_{j=1}^l \theta_{j,\tau}(t_{n,r}) \varphi_j\right)\right)\,d\bx =\int_{\Omega}   \left(\tau_n\sum_{r=1}^{q+1}\omega_{r}v_{\tau,h}(t_{n,r})^2\right)\,d\bx\\
  \label{idequnorm1}
  &= \tau_n\sum_{r=1}^{q+1}\omega_{r}\int_{\Omega}v_{\tau,h}(t_{n,r})^2\,d\bx=\tau_n\sum_{r=1}^{q+1}\omega_{r}\Vert v_{\tau,h}(t_{n,r})\Vert_{0,\Omega}^2
\end{align}
Furthermore, because $\nabla v_{\tau,h}=\sum_{j=1}^l \theta_{j,\tau} \nabla\varphi_j$, by applying arguments similar to identity \eqref{idequnorm1}, it follows that
\begin{equation}\label{idequnorm2} 
\Vert \nabla v_{\tau,h}\Vert_{L^2(I_n;L^{2}(\Omega))}^2=\tau_n\sum_{r=1}^{q+1}\omega_{r} |v_{\tau,h}(t_{n,r})|_{1,\Omega}^2.
\end{equation}
By employing \eqref{idequnorm1}, \eqref{idequnorm2}, and the equivalence \eqref{localeqnorm}, we obtain that
\[
\Vert v_{\tau,h}\Vert_{L^2(I_n;H^{1}(\Omega))}^2=\tau_n\sum_{r=1}^{q+1}\omega_{r} \Vert v_{\tau,h}(t_{n,r})\Vert_{1,\Omega}^2 \approx \tau_n\sum_{r=1}^{q+1}\omega_{r}\Vert v_{\tau,h}(t_{n,r})\Vert_{d,h}^2,
\]
which completes the proof.
}
\end{proof}

\subsubsection{Space--time projection and commutation properties}

For $v\in C^{0}((t_{n-1},t_n])$, 
we define the projector $\Pi_{n}^{\tau}$ in the following way:
\begin{enumerate}[(i)] 
    \item $\Pi_{n}^{\tau}v \in \mathbb{P}_{q}(I_n)$, 
    \item  $\int_{I_n}(\Pi_{n}^{\tau}v-v)w\, dt=0$, for all  $w\in \mathbb{P}_{q-1}(I_n)$,
    \item $\left( \Pi_{n}^{\tau}v\right)^{-,n}=v^{-,n}$.
\end{enumerate}
Let $v\in \cap_{n=1}^{N} C^0((t_{n-1},t_{n}]; L^2(\Omega))$ be given. 
The operator $\Pi^{\tau}$ is then defined as $\Pi^{\tau} v(t,\bx)=\Pi_{n}^{\tau}v(t,\bx)$ in $I_n\times \Omega$. 

\corr{
Given a function $v_h:=\sum_{j=1}^l \theta_{j} \varphi_j\in  \cap_{n=1}^{N} C^0((t_{n-1},t_{n}];V_h^k(\Omega))$ ($\theta_j\in \cap_{n=1}^{N} C^0((t_{n-1},t_{n}]) $, $j=1,\dots,l$), we have that
\[
m_h(\Pi^{\tau}v_{h}-v_{h},w_h)=\sum_{j=1}^l (\Pi^{\tau}\theta_{j}-\theta_{j})m_h(\varphi_j,w_h),
\]
for all $w_h\in V_h^k(\Omega)$.
}

\corr{
Consequently, we obtain
\begin{equation}\label{idmhint}
\int_{I_n}m_h(\Pi^{\tau}v_{h}-v_{h},w_h)\theta_{\tau}^{\ast}\,dt=0,
\end{equation}
for all  $v_h\in  \cap_{n=1}^{N} C^0((t_{n-1},t_{n}];V_h^k(\Omega))$,  $w_h\in V_h^k(\Omega)$, $\theta_{\tau}^{\ast}\in \mathbb{P}_{q-1}(I_n)$ and $n=1,...,N$.
} 

\corr{
Let $v_h\in \cap_{n=1}^{N} C^0((t_{n-1},t_{n}]; L^2(\Omega))$ and $w_{\tau,h}\in V_{\tau,h}^{q,k}(I\times\Omega)$ be given by their respective representations in the basis $\{\varphi_j\}_{j=1}^l$ of $V_h^k(\Omega)$, namely
$$v_h = \sum_{j=1}^l \theta_j \varphi_j \quad \text{and} \quad w_{\tau,h} = \sum_{j=1}^l \theta_{j,\tau}^{\ast} \varphi_j,$$
Using integration by parts, we can deduce that
\begin{align}
\nonumber
\int_{I_n}m_h(\partial_t v_h,w_{\tau,h})\, dt&=\int_{I_n}\sum_{i,j=1}^l \partial_t \theta_j \theta_{j,\tau}^{\ast}m_h(\varphi_i,\varphi_j)\,dt=\sum_{i,j=1}^l \int_{I_n}\partial_t \theta_j \theta_{j,\tau}^{\ast}\, dt\; m_h(\varphi_i,\varphi_j)\\
\nonumber
&=\sum_{i,j=1}^l m_h(\varphi_i,\varphi_j)\left[ \theta_j^{-,n} (\theta_{j,\tau}^{\ast})^{-,n}-\theta_j^{+,n-1} (\theta_{j,\tau}^{\ast})^{+,n-1}-\int_{I_n} \theta_j \partial_t \theta_{j,\tau}^{\ast}\, dt \right]\\
\nonumber
&=\sum_{i,j=1}^l\left[ \theta_j^{-,n} (\theta_{j,\tau}^{\ast})^{-,n} m_h(\varphi_i,\varphi_j)-\theta_j^{+,n-1} (\theta_{j,\tau}^{\ast})^{+,n-1} m_h(\varphi_i,\varphi_j)-\int_{I_n} \theta_j \partial_t \theta_{j,\tau}^{\ast} m_h(\varphi_i,\varphi_j)\, dt \right]\\
\label{identityDt}
&=m_h({v_{h}}^{-,n},{w_{\tau,h}}^{-,n})-m_h({v_{h}}^{+,n-1},{w_{\tau,h}}^{+,n-1})-\int_{I_n} m_h({v_{h}}, \partial_t {w_{\tau,h}})\, dt,
\end{align}
}

\corr{
By combining \eqref{idmhint}, \eqref{identityDt} and property (iii) of $\Pi^{\tau}$
\begin{align} \nonumber
\int_{I_n}m_h(\partial_t\Pi^{\tau}v_{h},w_{\tau,h})\,dt&=m_h((\Pi^{\tau}v_{h})^{-,n},w_{\tau,h}^{-,n})-m_h((\Pi^{\tau}v_{h})^{+,n-1},w_{\tau,h}^{+,n-1})-\int_{I_n}m_h(\Pi^{\tau}v_{h},\partial_t w_{\tau,h})\,dt \\
\nonumber
&=m_h((\Pi^{\tau}v_{h})^{-,n},w_{\tau,h}^{-,n})-m_h((\Pi^{\tau}v_{h})^{+,n-1},w_{\tau,h}^{+,n-1})-\int_{I_n}m_h(v_{h},\partial_t w_{\tau,h})\,dt \\
\label{intbyparts}
&=m_h((v_{h})^{+,n-1},(w_{\tau,h})^{+,n-1})-m_h((\Pi^{\tau}v_h)^{+,n-1},(w_{\tau,h})^{+,n-1})+\int_{I_n}m_h(\partial_t v_{h}, w_{\tau,h})\,dt.
\end{align}
}

Furthermore, the projector satisfies the following approximation property and
commutation results with spatial derivatives.
\begin{lemm}\label{Pitau}
If $v\in H^{q+1}(I_n)$, then 
$$
\Vert \Pi^{\tau}v-v\Vert_{L^2(I_n)}\lesssim \tau_n^{q+1}\Vert \partial_t^{q+1} v\Vert_{L^2(I_n)}.
$$
\end{lemm}
For $v \in L^2((0,T);H^k(\tilde{\Omega}))$, we also have:
$$
 \partial_{x_j}^{k} \Pi^{\tau} v=\Pi^{\tau} \partial_{x_j}^{k} v, \text{ a.e. in }\tilde{\Omega}.
$$

\subsubsection{Fully discrete formulation}\label{sec:fully}
\corr{
Given the continuity of the exact solution, the identity can be deduced from \eqref{revarf}
\begin{equation*}
\int_{I_n}[(\partial_t \bc,\bw)_{\Omega}+\mathcal{D}(\bc,\bw)]+(\bc^{+,n-1}-\bc^{-,n-1},\bw^{+,n-1})_{\Omega}=\int_{I_n}\mathcal{H}(\bw),
\end{equation*}
By employing a VEM spatial discretization, as in the semidiscrete case, and a Gauss-Radau quadrature rule for the time integral, we can obtain the following variational formulation:}
 
find $\bc_{\tau,h}:=(c_1^{\tau,h},\dots,c_{n_c}^{\tau,h})\in V_{\tau,h}^{q,k}(I\times \Omega)^{n_c}$ such that, 
for $n=1,...,N$:
\begin{align}
m_h(\bc_{\tau,h}^{+,n-1},\bw_{\tau,h}^{+,n-1})+\mathcal{D}_{\tau_{n},h}(\bc_{\tau,h},\bw_{\tau,h})&=\mathcal{F}_{\tau_n,h}^{\Omega}(\bw_{\tau,h})+\mathcal{F}_{\tau_n,h}^{\Gamma}(\bw_{\tau,h})+m_h(\bc_{\tau,h}^{-,n-1},\bw_{\tau,h}^{+,n-1}), \label{fullyproblem}\\
\bc_{\tau,h}^{-,0} &=\boldsymbol{\Pi}_{\tau,h}\bc_0, \nonumber
\end{align}
for all $\bw_{\tau,h} \in V_{\tau,h}^{q,k}(I\times \Omega)^{n_c}$, 
where the operators are defined via Gauss--Radau quadrature:
\begin{align*}
\mathcal{D}_{\tau_{n},h}(\bc_{\tau,h},\bw_{\tau,h})&:=\tau_n \sum_{i=1}^{q+1}\omega_i(\corr{m_h(\partial_t \bc_{\tau,h}(t_{n,i}),\bw_{\tau,h}(t_{n,i}))}+\mathcal{D}_{h}(\bc_{\tau,h}(t_{n,i}),\bw_{\tau,h}(t_{n,i}))\\
\mathcal{F}_{\tau_n,h}^{\Omega}(\bw_{\tau,h})&:=\tau_n \sum_{i=1}^{q+1}\omega_i \mathcal{F}_{h}^{\Omega}(\bw_{\tau,h}(t_{n,i})) \quad \text{and}\\
\mathcal{F}_{\tau_n,h}^{\Gamma}(\bw_{\tau,h})&:=\tau_n \sum_{i=1}^{q+1}\omega_i \mathcal{F}_{h}^{\Gamma}(\bw_{\tau,h}(t_{n,i})).
\end{align*}

\paragraph{Existence and uniqueness.}
To carry out the existence and uniqueness analysis, 
we define the following norm:
\[
\Vert \bw_{\tau,h}\Vert_{n,h}^2:=\tau_n \sum_{i=1}^{q+1}\omega_i \Vert \bw_{\tau,h}(t_{n,i})\Vert_h^2.
\]

\begin{thm} 
Suppose $\bu \in  L^{\infty}((0,T);(H^{1+s}(\Omega))^{d})$ for some $s>0$ and $f\in L^{\infty}((0,T);L^4(\Omega))$.
Then:
\begin{gather}
m_h(\bw_{\tau,h}^{+,n-1},\mathcal{L}_{\tau}\bw_{\tau,h}^{+,n-1})+\mathcal{D}_{\tau_n,h}(\bw_{\tau,h},\mathcal{L}_{\tau}\bw_{\tau,h})
 \gtrsim \Vert (\bw_{\tau,h})^{-,n}\Vert_{0,2,\Omega}^2+\frac{1}{\tau_n}\Vert \bw_{\tau,h} \Vert_{L^2(I_n;L^2(\Omega))}^2+\Vert \bw_{\tau,h}\Vert_{n,h}^2. \label{inf-sup}
\end{gather}
\end{thm}

\begin{proof}
\corr{
By Lemma \ref{identitiesGR}, we have $\mathcal{L}_{\tau}\bw_{\tau,h}(t_{n,i})=\tau_n(t_{n,i}-t_{n-1})^{-1}\bw_{\tau,h}(t_{n,i})$. Since $t_{n,i} \in I_{n}$, it follows that $\tau_n(t_{n,i}-t_{n-1})^{-1} > 1$. Hence, by the coercivity of  $\mathcal{D}_h$ it follows that:
\begin{align*}
\tau_n\sum_{i=1}^{q+1}\omega_i\mathcal{D}_{h}(\bw_{\tau,h}(t_{n,i}),\mathcal{L}_{\tau}\bw_{\tau,h}(t_{n,i}))+\Vert \bw_{\tau,h} \Vert_{L^2(I_n;L^2(\Omega))} &\gtrsim \Vert \bw_{\tau,h}\Vert_{n,h}.
\end{align*}
}

\corr{
Let $\{\varphi_j\}_{j=1}^l$ be an $m_h$-orthogonal basis of the space $V_h^k(\Omega)$, i.e., $m_h(\varphi_i, \varphi_j) = 0$ for $i \neq j$, any function $v_{\tau,h} \in V_{\tau,h}^{q,k}(I \times \Omega)$ can be uniquely represented as
$$v_{\tau,h}(t,x) =  \sum_{j=1}^l  \theta_{j,\tau} (t) \varphi_j(x),$$
where $\theta_{j,\tau} \in V_{\tau}^q(I)$ for $j=1,\dots,l$. Next, we define the functional
$$I(v_{\tau,h}) := \tau_n \sum_{i=1}^{q+1} \omega_i m_h(\partial_t v_{\tau,h}(t_{n,i}), \mathcal{L}_{\tau} v_{\tau,h}(t_{n,i})) + m_h(v_{\tau,h}^{+,n-1}, \mathcal{L}_{\tau}v_{\tau,h}^{+,n-1}).$$
By virtue of identities \eqref{Lagrangeint} and \eqref{exacttheta}, and utilizing the $m_h$-orthogonality of the spatial basis, we then have
\begin{align*}
I(v_{\tau,h})&=\tau_n \sum_{i=1}^{q+1} \omega_i m_h(\partial_t v_{\tau,h}(t_{n,i}), \mathcal{L}_{\tau} v_{\tau,h}(t_{n,i})) + m_h(v_{\tau,h}^{+,n-1}, \mathcal{L}_{\tau}v_{\tau,h}^{+,n-1}) \\
&= \sum_{j=1}^l m_h(\varphi_j, \varphi_j) \left[ \tau_n \sum_{i=1}^{q+1} \omega_i \partial_t\theta_{j,\tau}(t_{n,i}) \mathcal{L}_{\tau} \theta_{j,\tau}(t_{n,i}) + \theta_{j,\tau}^{+,n-1} \mathcal{L}_{\tau} \theta_{j,\tau}^{+,n-1} \right]\\
&=\sum_{j=1}^l \frac{1}{2}m_h(\varphi_j, \varphi_j)\left[(\theta_{j,\tau}^{-,n})^2+\sum_{i=1}^{q+1}\omega_i \xi_{i}^{-1}(\theta_{j,\tau}(t_{n,i}))^2\right]\\
&\geq \sum_{j=1}^l \frac{1}{2}m_h(\varphi_j, \varphi_j)\left[(\theta_{j,\tau}^{-,n})^2+\sum_{i=1}^{q+1}\omega_i (\theta_{j,\tau}(t_{n,i}))^2\right]=\sum_{j=1}^l \frac{1}{2}m_h(\varphi_j, \varphi_j)\left[(\theta_{j,\tau}^{-,n})^2+\frac{1}{\tau_n}\int_{I_n} \theta_{j,\tau}^2\,dt\right]\\
&=\frac{1}{2}\left[m_h(v_{\tau,h}^{-,n},v_{\tau,h}^{-,n})+\frac{1}{\tau_n}\int_{I_n} m_h(v_{\tau,h},v_{\tau,h})\,dt\right]\gtrsim \frac{1}{2}\left[ \Vert v_{\tau,h}^{-,n}\Vert_{0,\Omega}^2+ \frac{1}{\tau_n}\int_{I_n}\Vert v_{\tau,h}(t) \Vert_{0,\Omega}^2\,dt\right].
\end{align*}
}

By combining these inequalities, we obtain
\corr{
\[
m_h(\bw_{\tau,h}^{+,n-1},\mathcal{L}_{\tau}\bw_{\tau,h}^{+,n-1})+\mathcal{D}_{\tau_n,h}(\bw_{\tau,h},\mathcal{L}_{\tau}\bw_{\tau,h})
 \gtrsim \Vert (\bw_{\tau,h})^{-,n}\Vert_{0,2,\Omega}^2+\left(\frac{1}{\tau_n}-1\right)\Vert \bw_{\tau,h} \Vert_{L^2(I_n;L^2(\Omega))}^2+\Vert \bw_{\tau,h}\Vert_{n,h}^2.
\] 
}
\corr{Finally, the conclusion is obtained by choosing $\tau_n$ sufficiently small.}
\end{proof}

\paragraph{Error estimates.}

\begin{lemm}\label{errorfullysemi}
Let $\bc$, $\tilde{\bc}$, $\bc_I$, $f$ and $\bu$ be functions satisfying the assumptions of Theorem \ref{errorsemi}. Additionally, assume that $\bc\in H^{q+1}(I_n;H^1(\Omega))^{n_c}$. Then 
\begin{equation}
    \Vert \bc_{\tau,h}^{-,n}-\bc_h(t_n)\Vert_{0,\Omega}^2+\Vert \bc_{\tau,h}-\boldsymbol{\Pi}^{\tau}\bc_h\Vert_{n,h}^2 \, dt \lesssim (\tau_n^{2q+1}+h^{2s})M(\bu,\bc)+\Vert \bc_{\tau,h}^{-,n-1}-\bc_{h}(t_{n-1})\Vert_{0,\Omega}^2,
\end{equation}
$n=1,...,N$, where $M(\bu,\bc)$ denotes a positive constant that depends on the solution and the problem data.
\end{lemm}

\begin{proof}
\corr{
Setting $\zz_{\tau,h}:=\bc_{\tau,h}-\boldsymbol{\Pi}^{\tau}\bc_h$ and employing \eqref{inf-sup}, we deduce that
\begin{equation}\label{eqerror1}
E:=m_h(\zz_{\tau,h}^{+,n-1},\mathcal{L}_{\tau}\zz_{\tau,h}^{+,n-1})+\mathcal{D}_{\tau_n,h}(\zz_{\tau,h},\mathcal{L}_{\tau}\zz_{\tau,h})
 \gtrsim \Vert (\zz_{\tau,h})^{-,n}\Vert_{0,2,\Omega}^2+\frac{1}{\tau_n}\Vert \zz_{\tau,h} \Vert_{L^2(I_n;L^2(\Omega))}^2+\Vert \zz_{\tau,h}\Vert_{n,h}^2.  
\end{equation}
By linearity, we can rewrite the term $E$ in \eqref{eqerror1}  as
\[
E=m_h(\bc_{\tau,h}^{+,n-1},\mathcal{L}_{\tau}\zz_{\tau,h}^{+,n-1})+\mathcal{D}_{\tau_n,h}(\bc_{\tau,h},\mathcal{L}_{\tau}\zz_{\tau,h})-m_h((\boldsymbol{\Pi}^{\tau}\bc_{h})^{+,n-1},\mathcal{L}_{\tau}\zz_{\tau,h}^{+,n-1})+\mathcal{D}_{\tau_n,h}(\boldsymbol{\Pi}^{\tau}\bc_{h},\mathcal{L}_{\tau}\zz_{\tau,h}).
\]
Using the fact that $\bc_{\tau,h}$ is a solution to Problem \eqref{fullyproblem}, we obtain
\[
E=\mathcal{F}_{\tau_n,h}^{\Omega}(\mathcal{L}_{\tau}\zz_{\tau,h})+\mathcal{F}_{\tau_n,h}^{\Gamma}(\mathcal{L}_{\tau}\zz_{\tau,h})+m_h(\bc_{\tau,h}^{-,n-1},(\mathcal{L}_{\tau}\zz_{\tau,h})^{+,n-1})-m_h((\boldsymbol{\Pi}^{\tau}\bc_{h})^{+,n-1},\mathcal{L}_{\tau}\zz_{\tau,h}^{+,n-1})+\mathcal{D}_{\tau_n,h}(\boldsymbol{\Pi}^{\tau}\bc_{h},\mathcal{L}_{\tau}\zz_{\tau,h}),
\]
By virtue of identity \eqref{eqn:varFormsemi}, we deduce that
\[
\mathcal{F}_{\tau_n,h}^{\Omega}(\mathcal{L}_{\tau}\zz_{\tau,h})+\mathcal{F}_{\tau_n,h}^{\Gamma}(\mathcal{L}_{\tau}\zz_{\tau,h})=\tau_n\sum_{i=1}^{q+1}\omega_i[m_h(\partial_t \bc_h(t_{n,i}),\mathcal{L}_{\tau}\zz_{\tau,h}(t_{n,i}))+\mathcal{D}_{h}(\bc_h(t_{n,i}),\mathcal{L}_{\tau}\zz_{\tau,h}(t_{n,i}))].
\]
Using the aforementioned equality, linearity, and the definition of $\mathcal{D}_{\tau_n,h}$, we can rewrite the term $E$ as follows
\begin{gather*}
E=\tau_n\sum_{i=1}^{q+1}\omega_i[m_h(\partial_t (I-\boldsymbol{\Pi}^{\tau})\bc_h(t_{n,i}),\mathcal{L}_{\tau}\zz_{\tau,h}(t_{n,i}))+\mathcal{D}_{h}((I-\boldsymbol{\Pi}^{\tau})\bc_h(t_{n,i}),\mathcal{L}_{\tau}\zz_{\tau,h}(t_{n,i}))]\\
+m_h(\bc_{\tau,h}^{-,n-1},\mathcal{L}_{\tau}\zz_{\tau,h}^{+,n-1})-m_h((\boldsymbol{\Pi}^{\tau}\bc_{h})^{+,n-1},\mathcal{L}_{\tau}\zz_{\tau,h}^{+,n-1}).
\end{gather*}
Since $\mathcal{L}_{\tau}\zz_{\tau,h}(t_{n,i})=\tau_n(t_{n,i}-t_{n-1})^{-1}\zz_{\tau,h}(t_{n,i})=\tau_n(\tau_n \xi_i)^{-1}\zz_{\tau,h}(t_{n,i})=\xi_i^{-1}\zz_{\tau,h}(t_{n,i})$, we have that
\begin{gather*}
E=\tau_n\sum_{i=1}^{q+1}\omega_i\xi_i^{-1}[m_h(\partial_t (I-\boldsymbol{\Pi}^{\tau})\bc_h(t_{n,i}),\zz_{\tau,h}(t_{n,i}))+\mathcal{D}_{h}((I-\boldsymbol{\Pi}^{\tau})\bc_h(t_{n,i}),\zz_{\tau,h}(t_{n,i}))]\\
+m_h(\bc_{\tau,h}^{-,n-1},\mathcal{L}_{\tau}\zz_{\tau,h}^{+,n-1})-m_h((\boldsymbol{\Pi}^{\tau}\bc_{h})^{+,n-1},\mathcal{L}_{\tau}\zz_{\tau,h}^{+,n-1}).
\end{gather*}
In order to estimate the term $E$, we proceed by decomposing it into $E = E_1 + E_2 + E_3$, where 
\begin{align*}
E_1 &:=\tau_n\sum_{i=1}^{q+1}\omega_i\xi_i^{-1}m_h(\partial_t (I-\boldsymbol{\Pi}^{\tau})\bc_h(t_{n,i}),\zz_{\tau,h}(t_{n,i})), \quad E_2:= \tau_n\sum_{i=1}^{q+1}\omega_i\xi_i^{-1}\mathcal{D}_{h}((I-\boldsymbol{\Pi}^{\tau})\bc_h(t_{n,i}),\zz_{\tau,h}(t_{n,i})) \\
& \text{and} \qquad E_3:=m_h(\bc_{\tau,h}^{-,n-1},\mathcal{L}_{\tau}\zz_{\tau,h}^{+,n-1})-m_h((\boldsymbol{\Pi}^{\tau}\bc_{h})^{+,n-1},\mathcal{L}_{\tau}\zz_{\tau,h}^{+,n-1}).
\end{align*}
Letting $I^\tau$ denote the Lagrange polynomial interpolant with respect to the temporal variable at the points $t_{n,i}$ for $i=1,\dots,q+1$, the exactness of the Gauss–Radau quadrature for polynomials of degree at most $2q$ implies that
\begin{gather*}
E_1\leq \xi_1^{-1}\tau_n\left| \sum_{i=1}^{q+1}\omega_i m_h(\partial_t (I-\boldsymbol{\Pi}^{\tau})\bc_h(t_{n,i}),\zz_{\tau,h}(t_{n,i}))\right|=\xi_1^{-1}\tau_n\left|\sum_{i=1}^{q+1}\omega_i m_h(I^{\tau}\partial_t (I-\boldsymbol{\Pi}^{\tau})\bc_h(t_{n,i}),\zz_{\tau,h}(t_{n,i}))\right|\\
=\xi_1^{-1} \left|\int_{I_n} m_h(I^{\tau}\partial_t (I-\boldsymbol{\Pi}^{\tau})\bc_h,\zz_{\tau,h})\,dt\right|,
\end{gather*}
Using the cotinuity of $m_h$ and the Cauchy-Schwarz inequality to the time integrals over $I_n$, shifting to the global space-time norm $L^2(I_n;L^2(\Omega))$
\begin{gather*}
E_1\lesssim \xi_1^{-1} \int_{I_n} \Vert I^{\tau}\partial_t (I-\boldsymbol{\Pi}^{\tau})\bc_h\Vert_{0,\Omega} \Vert \zz_{\tau,h}\Vert_{0,\Omega}\,dt \lesssim \xi_1^{-1}  \Vert I^{\tau}\partial_t (I-\boldsymbol{\Pi}^{\tau})\bc_h\Vert_{L^2(I_n;L^2(\Omega))} \Vert \zz_{\tau,h}\Vert_{L^2(I_n;L^2(\Omega))}.
\end{gather*}
By the temporal $L^2$-stability of the Lagrange interpolant $I^{\tau}$ and the standard temporal inverse inequality over $I_n$
\begin{gather}\label{boundE1}
E_1\lesssim \xi_1^{-1}  \Vert \partial_t (I-\boldsymbol{\Pi}^{\tau})\bc_h\Vert_{L^2(I_n;L^2(\Omega))} \Vert \zz_{\tau,h}\Vert_{L^2(I_n;L^2(\Omega))}\lesssim \xi_1^{-1}  \tau_n^{-1}\Vert  (I-\boldsymbol{\Pi}^{\tau})\bc_h\Vert_{L^2(I_n;L^2(\Omega))} \Vert \zz_{\tau,h}\Vert_{L^2(I_n;L^2(\Omega))}
\end{gather}
To bound the projection error, we insert the exact solution $\bc$ via the triangle inequality, yielding
\begin{equation}\label{triangularbc}
\Vert  (I-\boldsymbol{\Pi}^{\tau})\bc_h\Vert_{L^2(I_n;L^2(\Omega))}\leq \Vert  \boldsymbol{\Pi}^{\tau}(\bc-\bc_h)\Vert_{L^2(I_n;L^2(\Omega))}+\Vert  (I-\boldsymbol{\Pi}^{\tau})\bc \Vert_{L^2(I_n;L^2(\Omega))}+\Vert  \bc_h-\bc\Vert_{L^2(I_n;L^2(\Omega))} 
\end{equation}
Using the error bounds for $\bc_h$ and the approximation property provided in Lemma \ref{Pitau} we obtain
\begin{gather}\nonumber
\Vert  \boldsymbol{\Pi}^{\tau}(\bc-\bc_h)\Vert_{L^2(I_n;L^2(\Omega))}^2\lesssim \Vert  \bc-\bc_h\Vert_{L^2(I_n;L^2(\Omega))}^2\lesssim \sup_{t\in (0,T]}\Vert (\bc-\bc_h)(t)\Vert_{0,\Omega}^2\int_{I_n}\,dt\lesssim C(\bu,\bc)h^{2s}\tau_n\\
\label{aproxproperty} 
\text{and}\, \,\Vert  (I-\boldsymbol{\Pi}^{\tau})\bc \Vert_{L^2(I_n;L^2(\Omega))}\lesssim \tau_n^{q+1}\Vert \bc \Vert_{H^{q+1}(I_n;L^2(\Omega))} 
\end{gather}
Employing inequalities \eqref{boundE1} and \eqref{aproxproperty}, we have that
\begin{gather}\nonumber
E_1\lesssim \xi_1^{-1}  \tau_n^{-1}\left[C(\bu,\bc)^{1/2}h^{s}\tau_n^{1/2}+\tau_n^{q+1}\Vert \bc \Vert_{H^{q+1}(I_n;L^2(\Omega))}\right]\Vert \zz_{\tau,h}\Vert_{L^2(I_n;L^2(\Omega))}\\
\label{youngE1}
=\xi_1^{-1}  \left[C(\bu,\bc)^{1/2}h^{s}+\tau_n^{q+\frac{1}{2}}\Vert \bc \Vert_{H^{q+1}(I_n;L^2(\Omega))}\right]\tau_n^{-1/2}\Vert \zz_{\tau,h}\Vert_{L^2(I_n;L^2(\Omega))}
\end{gather}
Regarding the term $E_2$, invoking the continuity of the operator $\mathcal{D}_h$ together with the Cauchy–Schwarz inequality yields
\begin{align*}
E_2&\lesssim  \tau_n\left|\sum_{i=1}^{q+1}\omega_i\xi_i^{-1}\mathcal{D}_{h}((I-\boldsymbol{\Pi}^{\tau})\bc_h(t_{n,i}),\zz_{\tau,h}(t_{n,i}))\right|=  \tau_n\left|\sum_{i=1}^{q+1}\omega_i\xi_i^{-1}\mathcal{D}_{h}(I^{\tau}(I-\boldsymbol{\Pi}^{\tau})\bc_h(t_{n,i}),\zz_{\tau,h}(t_{n,i}))\right|\\
&\lesssim  \tau_n\sum_{i=1}^{q+1}\omega_i\xi_i^{-1}\Vert I^{\tau}(I-\boldsymbol{\Pi}^{\tau})\bc_h(t_{n,i})\Vert_h \Vert \zz_{\tau,h}(t_{n,i}))\Vert_h\lesssim  \tau_n\xi_1^{-1}\sum_{i=1}^{q+1}\omega_i\Vert I^{\tau}(I-\boldsymbol{\Pi}^{\tau})\bc_h(t_{n,i})\Vert_h \Vert \zz_{\tau,h}(t_{n,i}))\Vert_h \\
&\lesssim \xi_1^{-1}\left( \tau_n\sum_{i=1}^{q+1}\omega_i\Vert I^{\tau}(I-\boldsymbol{\Pi}^{\tau})\bc_h(t_{n,i})\Vert_h^2\right)^{1/2} \left( \tau_n\sum_{i=1}^{q+1}\omega_i\Vert \zz_{\tau,h}(t_{n,i}))\Vert_h^2\right)^{1/2} \\
&\lesssim \xi_1^{-1}\left( \tau_n\sum_{i=1}^{q+1}\omega_i\Vert I^{\tau}(I-\boldsymbol{\Pi}^{\tau})\bc_h(t_{n,i})\Vert_{1,\Omega}^2\right)^{1/2} \left( \tau_n\sum_{i=1}^{q+1}\omega_i\Vert \zz_{\tau,h}(t_{n,i}))\Vert_h^2\right)^{1/2}.
\end{align*}
Using the definition of $\Vert \cdot \Vert_{n,h}$, the exactness of the Gauss–Radau quadrature for polynomials of degree at most $2q$ and the temporal $L^2$-stability of the Lagrange interpolant $I^{\tau}$, we have that
\begin{align*}
E_2&\lesssim \xi_1^{-1}\Vert I^{\tau}(I-\boldsymbol{\Pi}^{\tau})\bc_h\Vert_{L^2(I_n; H^1(\Omega))} \left( \tau_n\sum_{i=1}^{q+1}\omega_i\Vert \zz_{\tau,h}(t_{n,i}))\Vert_h^2\right)^{1/2}\\
&\lesssim \xi_1^{-1}\Vert I^{\tau}(I-\boldsymbol{\Pi}^{\tau})\bc_h\Vert_{L^2(I_n; H^1(\Omega))} \Vert \zz_{\tau,h}\Vert_{n,h}
\lesssim \xi_1^{-1}\Vert (I-\boldsymbol{\Pi}^{\tau})\bc_h\Vert_{L^2(I_n; H^1(\Omega))} \Vert \zz_{\tau,h}\Vert_{n,h}
\end{align*}
Employing inequalities \eqref{triangularbc} and \eqref{aproxproperty}, we have that
\begin{gather}\label{youngE2}
E_2\lesssim \xi_1^{-1}  \left[C(\bu,\bc)^{1/2}h^{s}\tau_n^{1/2}+\tau_n^{q+1}\Vert \bc \Vert_{H^{q+1}(I_n;L^2(\Omega))}\right]\Vert \zz_{\tau,h}\Vert_{n,h}
\end{gather}
By virtue of the continuity of $\bc_h$ with respect to the temporal variable, we obtain $\bc_h^{+,n-1}=\bc_h^{-,n-1}$, hence
\[
m_h(\bc_{h}^{-,n-1},\mathcal{L}_{\tau}\zz_{\tau,h}^{+,n-1})-m_h(\bc_{h}^{+,n-1},\mathcal{L}_{\tau}\zz_{\tau,h}^{+,n-1})=0.
\]
Employing this last inequality, we can write $E_3$ as
\[
E_3=m_h(\bc_{\tau,h}^{-,n-1}-\bc_{h}^{-,n-1},\mathcal{L}_{\tau}\zz_{\tau,h}^{+,n-1})+m_h((\bc_{h}-\boldsymbol{\Pi}^{\tau}\bc_{h})^{+,n-1},\mathcal{L}_{\tau}\zz_{\tau,h}^{+,n-1}).
\]
Using the continuity of $m_h$, we have that
\[
E_3\lesssim \left[\Vert \bc_{\tau,h}^{-,n-1}-\bc_{h}^{-,n-1} \Vert_{0,\Omega}+\Vert (\bc_{h}-\boldsymbol{\Pi}^{\tau}\bc_{h})^{+,n-1} \Vert_{0,\Omega}\right]\Vert \mathcal{L}_{\tau}\zz_{\tau,h}^{+,n-1}\Vert_{0,\Omega}.
\]
Employing the inequalities in \eqref{minusineq} and the $L^2$ continuity of $I^{\tau}$, we have that
\begin{gather*}
E_3\lesssim \left[\Vert \bc_{\tau,h}^{-,n-1}-\bc_{h}^{-,n-1} \Vert_{0,\Omega}+\tau_n^{-1/2}\Vert I^{\tau}\bc_{h}-\boldsymbol{\Pi}^{\tau}\bc_{h} \Vert_{L^{2}(I_n;L^2(\Omega))}\right]\tau_n^{-1/2}\Vert \zz_{\tau,h}\Vert_{L^{2}(I_n;L^2(\Omega))}\\
= \left[\Vert \bc_{\tau,h}^{-,n-1}-\bc_{h}^{-,n-1} \Vert_{0,\Omega}+\tau_n^{-1/2}\Vert I^{\tau}(\bc_{h}-\boldsymbol{\Pi}^{\tau}\bc_{h}) \Vert_{L^{2}(I_n;L^2(\Omega))}\right]\tau_n^{-1/2}\Vert \zz_{\tau,h}\Vert_{L^{2}(I_n;L^2(\Omega))}\\
\lesssim\left[\Vert \bc_{\tau,h}^{-,n-1}-\bc_{h}^{-,n-1} \Vert_{0,\Omega}+\tau_n^{-1/2}\Vert \bc_{h}-\boldsymbol{\Pi}^{\tau}\bc_{h} \Vert_{L^{2}(I_n;L^2(\Omega))}\right]\tau_n^{-1/2}\Vert \zz_{\tau,h}\Vert_{L^{2}(I_n;L^2(\Omega))}
\end{gather*}
Using inequalities \eqref{triangularbc} and \eqref{aproxproperty}, we have that
\[
\Vert \bc_{h}-\boldsymbol{\Pi}^{\tau}\bc_{h} \Vert_{L^{2}(I_n;L^2(\Omega))}\lesssim C(\bu,\bc)^{1/2}h^{s}\tau_n^{1/2}+\tau_n^{q+1}\Vert \bc \Vert_{H^{q+1}(I_n;L^2(\Omega))}
\]
This last inequality implies that
\begin{equation}\label{youngE3}
E_3\lesssim \left[\Vert \bc_{\tau,h}^{-,n-1}-\bc_{h}^{-,n-1} \Vert_{0,\Omega}+ C(\bu,\bc)^{1/2}h^{s}+\tau_n^{q+\frac{1}{2}}\Vert \bc \Vert_{H^{q+1}(I_n;L^2(\Omega))}\right]\tau_n^{-1/2}\Vert \zz_{\tau,h}\Vert_{L^{2}(I_n;L^2(\Omega))}.
\end{equation}
Combining inequalities \eqref{youngE1}, \eqref{youngE2}, and \eqref{youngE3} yields
\begin{gather}\nonumber
E\lesssim \left[\Vert \bc_{\tau,h}^{-,n-1}-\bc_{h}^{-,n-1} \Vert_{0,\Omega}+ C(\bu,\bc)^{1/2}h^{s}+\tau_n^{q+\frac{1}{2}}\Vert \bc \Vert_{H^{q+1}(I_n;L^2(\Omega))}\right]\tau_n^{-1/2}\Vert \zz_{\tau,h}\Vert_{L^{2}(I_n;L^2(\Omega))}\\
\label{youngE}
\qquad+\left[C(\bu,\bc)^{1/2}h^{s}\tau_n^{1/2}+\tau_n^{q+1}\Vert \bc \Vert_{H^{q+1}(I_n;L^2(\Omega))}\right]\Vert \zz_{\tau,h}\Vert_{n,h}.
\end{gather}
Applying Young's inequality to the inequality \eqref{youngE} and employing inequality \eqref{eqerror1}, we can deduce that
\begin{gather*}
\Vert (\zz_{\tau,h})^{-,n}\Vert_{0,2,\Omega}^2+\frac{1}{2\tau_n}\Vert \zz_{\tau,h} \Vert_{L^2(I_n;L^2(\Omega))}^2+\frac{1}{2}\Vert \zz_{\tau,h}\Vert_{n,h}^2 \lesssim
\Vert \bc_{\tau,h}^{-,n-1}-\bc_{h}^{-,n-1} \Vert_{0,\Omega}^2+ C(\bu,\bc)h^{2s}\\
+\tau_n^{2q+1}\Vert \bc \Vert_{H^{q+1}(I_n;L^2(\Omega))}^2+C(\bu,\bc)h^{2s}\tau_n+\tau_n^{2q+2}\Vert \bc \Vert_{H^{q+1}(I_n;L^2(\Omega))}^2,
\end{gather*}
which completes the proof.
}
\end{proof}

\begin{thm}\label{errorfully}
Let $\bc$, $\tilde{\bc}$, $\bc_I$, $f$ and $\bu$ satisfy the assumptions of Theorem \ref{errorsemi} and Lemma \ref{errorfullysemi}. Then there exists a constant $C>0$ such that the error $\mathfrak{e}=\bc-\bc_{\tau,h}$ satisfies:
\begin{equation}\label{ferror}
    \Vert \mathfrak{e}^{-,n}\Vert_{0,\Omega}^2+\Vert \mathfrak{e}\Vert_{L^2(I_n;H^1(\Omega))}^2  \lesssim C(h^{2s}+\tau_n^{2q+1}),
\end{equation}
for $n=1,...,N$.
\end{thm}
\begin{proof}
\corr{
Employing the triangle inequality alongside the approximation properties of $\boldsymbol{\Pi}^{\tau}$, it follows that
\begin{gather*}
\Vert \bc-\boldsymbol{\Pi}^{\tau} \bc_h\Vert_{L^2(I_n;H^1(\Omega))}\leq \Vert \bc-\boldsymbol{\Pi}^{\tau} \bc\Vert_{L^2(I_n;H^1(\Omega))}+\Vert \boldsymbol{\Pi}^{\tau} (\bc-\bc_h)\Vert_{L^2(I_n;H^1(\Omega))}\\
\lesssim \tau_n^{q+1}\Vert \bc\Vert_{H^{q+1}(I_n;H^1(\Omega))}+\Vert \bc-\bc_h\Vert_{L^2(I_n;H^1(\Omega))}
\end{gather*}
Combining the triangle inequality with the preceding inequality, we obtain
\begin{gather*}
\Vert \bc- \bc_{\tau,h}\Vert_{L^2(I_n;H^1(\Omega))}\leq \Vert \bc-\boldsymbol{\Pi}^{\tau} \bc_h\Vert_{L^2(I_n;H^1(\Omega))}+\Vert \boldsymbol{\Pi}^{\tau} \bc_h-\bc_{\tau,h}\Vert_{L^2(I_n;H^1(\Omega))}\\
\lesssim \tau_n^{q+1}\Vert \bc\Vert_{H^{q+1}(I_n;H^1(\Omega))}+\Vert \bc-\bc_h\Vert_{L^2(I_n;H^1(\Omega))}+\Vert \boldsymbol{\Pi}^{\tau} \bc_h-\bc_{\tau,h}\Vert_{L^2(I_n;H^1(\Omega))}
\end{gather*}
On the other hand, using the triangle inequality, we obtain
\begin{gather*}
\Vert (\bc- \bc_{\tau,h})^{-,n}\Vert_{0,\Omega}\lesssim \Vert (\bc- \bc_{h})^{-,n}\Vert_{0,\Omega}+\Vert (\bc_h- \bc_{\tau,h})^{-,n}\Vert_{0,\Omega}\lesssim \sup_{t\in(0,T]}\Vert (\bc- \bc_{h})(t)\Vert_{0,\Omega}+\Vert (\bc_h- \bc_{\tau,h})^{-,n}\Vert_{0,\Omega}
\end{gather*}
Combining all the preceding inequalities with Theorems \ref{errorsemi} and \ref{errorfullysemi}, we obtain
\begin{gather*}
\Vert \mathfrak{e}^{-,n}\Vert_{0,\Omega}^2+\Vert \mathfrak{e}\Vert_{L^2(I_n;H^1(\Omega))}^2\lesssim \sup_{t\in(0,T]}\Vert (\bc- \bc_{h})(t)\Vert_{0,\Omega}^2+\Vert (\bc_h- \bc_{\tau,h})^{-,n}\Vert_{0,\Omega}^2+\tau_n^{q+1}\Vert \bc\Vert_{H^{q+1}(I_n;H^1(\Omega))}\\
+\Vert \bc-\bc_h\Vert_{L^2(I_n;H^1(\Omega))}+\Vert \boldsymbol{\Pi}^{\tau} \bc_h-\bc_{\tau,h}\Vert_{L^2(I_n;H^1(\Omega))}\\
\lesssim C(\bu,\bc)h^{2s}+M(\bu,\bc)(\tau_n^{2q+1}+h^{2s})+\tau_n^{2q+2}\Vert \bc\Vert_{H^{q+1}(I_n;H^1(\Omega))}^2+\Vert (\bc_h- \bc_{\tau,h})^{-,n-1}\Vert_{0,\Omega}^2.
\end{gather*}
Then, using the fact that $\Vert (\bc_h- \bc_{\tau,h})^{-,0}\Vert_{0,\Omega}\lesssim h^{s+1}\Vert c(0)\Vert_{0,\Omega}$, the desired result follows.
}
\end{proof}

%%%%%%%%%%%%%%%%%%%%%%%%%%%%%%%%%% english checked
\section{Numerical Experiment}\label{sec:numExe}

In this section we present numerical evidence that supports the theoretical results of the previous sections.
Specifically, in Section~\ref{sec:num:conv} we perform a convergence analysis of the proposed method, while Section~\ref{sec:num:appl} presents a more application-oriented example.
Finally, in Section~\ref{sec:num:appl}, we consider a more applicative example.
In all these experiments, we have used the \texttt{C++} library \texttt{vem++}~\cite{Dassi:2023:VAC}.

\subsection{Convergence Analysis}\label{sec:num:conv}

In this section, we consider the following model problem.
Let $\Omega=(0,\,1)^2$ be the domain.
We set $K=D=\mathbb{I}$ and $\mu=1$.
We choose the Darcy forcing term $f$ and the functions $\tilde{c}$ and $c_I$ 
so that the exact solution is
$$
c(x,\,y,\,t) = \sin(t)\,e^{(x-1)^2(y-1)^2}\,,
$$
subjected to the following flow field 
$$
\mathbf{u}(x,\,y) = \begin{bmatrix}
e^x\\
e^y
\end{bmatrix}\,.
$$
According to this vector field,
we have that the inflow and outflow boundaries are 
\begin{align*}
\Gamma_{\mathrm{I}}=\left\{ (x,\,y) \in \partial\Omega:\, x=0 \text{ or } y=0 \right\}\quad\text{and}\quad
\Gamma_{\mathrm{O}}=\left\{ (x,\,y) \in \partial\Omega:\, x=1 \text{ or } y=1 \right\}.
\end{align*}
 
Starting from this problem, we aim to numerically verify the behavior of the 
error estimator defined in Equation~\eqref{ferror}:
$$
\texttt{err}^2 :=
    \Vert e^{-,m}\Vert_{L^2(\Omega)}^2+\sum_{n=1}^m\int_{I_n}\Vert e\Vert_{H^1(\Omega)}^2 \, dt
$$
To achieve this goal,
we consider a sequence of four space–time discretizations with decreasing mesh size and time step.
In order to check the robustness of the proposed method with respect to the elements' shape,
we use the following types of space discretizations:
\begin{itemize}
    \item \texttt{quad}: structured uniform square elements, see Figure~\ref{fig:meshTypes}(a);
    \item \texttt{hexa}: distorted hexagonal elements, see Figure~\ref{fig:meshTypes}(b);
    \item \texttt{voro}: Voronoi cells optimized using the Lloyd algorithm~\cite{Du:1999:CVT}, see Figure~\ref{fig:meshTypes}(c);
    \item \texttt{rand}: Voronoi cells with seed points randomly distributed within the domain~$\Omega$, see Figure~\ref{fig:meshTypes}(d).
\end{itemize}

\begin{figure}[!htb]
    \centering
    \begin{tabular}{cccc}
    \texttt{quad2} &\texttt{hexa2} &\texttt{voro2} &\texttt{rand2}\\
    \includegraphics[width=0.22\textwidth]{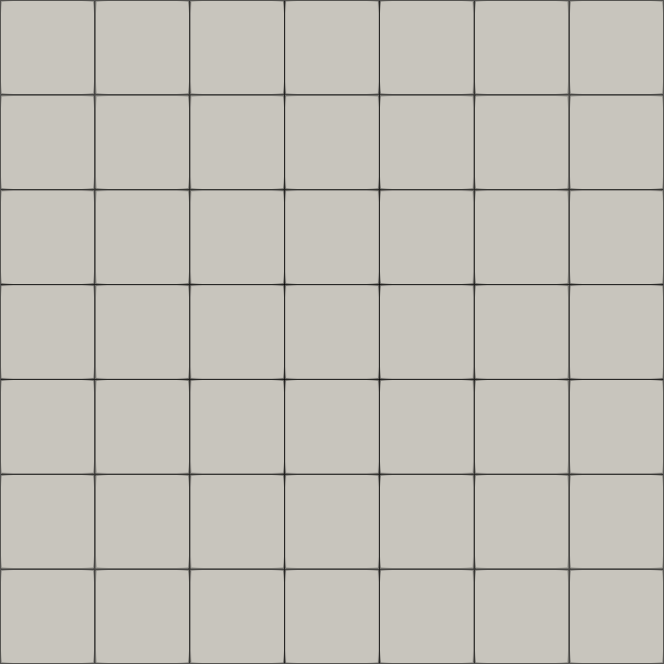} &
    \includegraphics[width=0.22\textwidth]{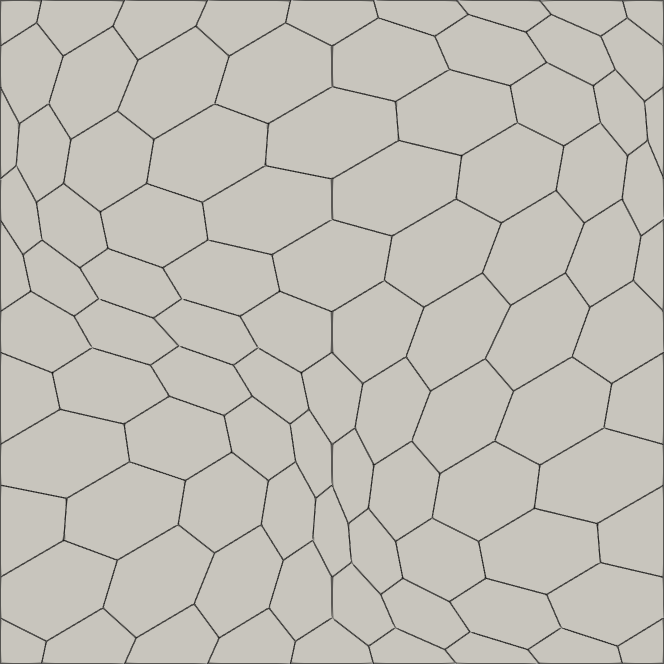} &
    \includegraphics[width=0.22\textwidth]{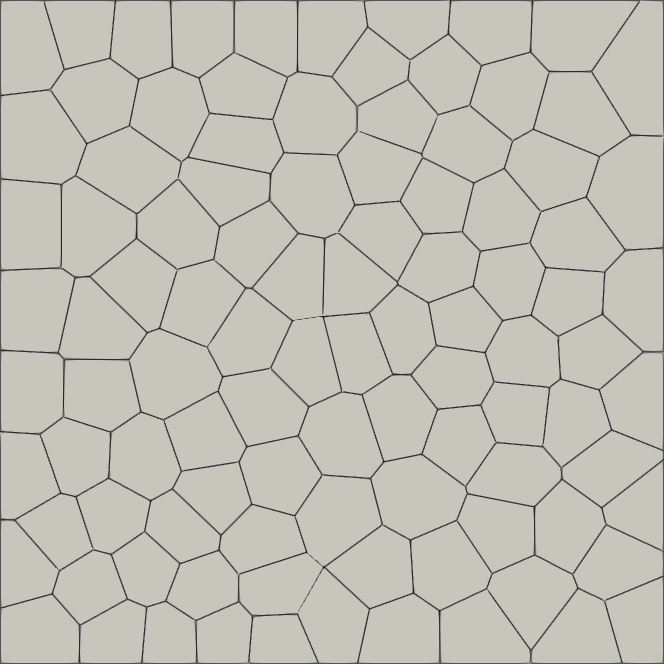} &
    \includegraphics[width=0.22\textwidth]{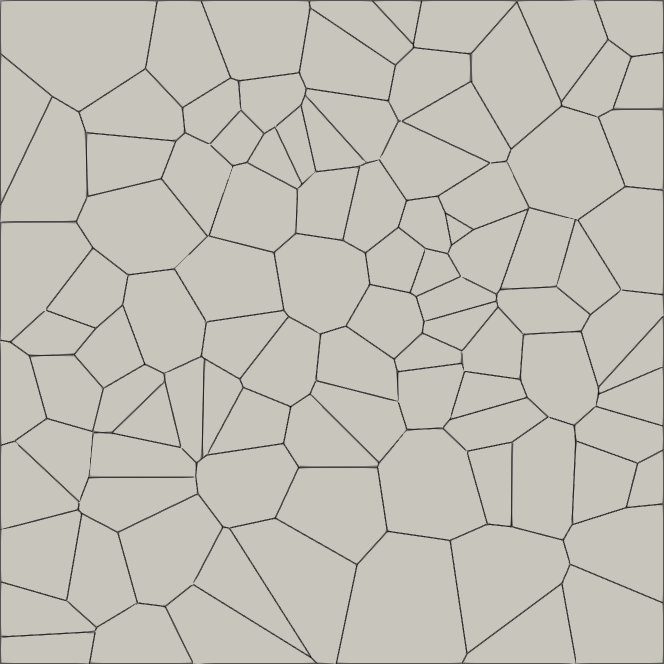} \\
    (a) &(b) &(c) &(d)\\
    \end{tabular}
    \caption{Different mesh types used for the convergence analysis.}
    \label{fig:meshTypes}
\end{figure}

These four types of meshes have an increasing distortion level.
The \texttt{quad} one is the most regular. 
The \texttt{hexa} mesh type has distorted elements but a mostly uniform topology, consisting mainly of hexagons.
The \texttt{voro} meshes consist of general polygons with both short and long edges,
but they have a regular shape due to the Lloyd optimization. 
Finally, the \texttt{rand} mesh is the most challenging one, 
as it contains highly distorted elements and very short edges.

Then, to show the convergence of the errors indicators, 
we construct four decompositions the unit square~$\Omega=(0,\,1)^2$ with decreasing mesh size~$h$ for each of them.
We label these discretizations with the numbers from 1 to 4,
1 will be the coarsest and 4 the finest.
For instance \texttt{voro1} refers to the coarser mesh of the \texttt{voro} type.

For the time discretization, we always consider the time interval~$[0,\,1]$
and use uniform time steps. 
More precisely, we define four temporal partitions, 
\texttt{inter1}, \texttt{inter2}, \texttt{inter3}, and \texttt{inter4},
corresponding to time steps $\Delta t = 1/3$, $1/6$, $1/12$, and $1/24$, respectively.
For all the simulations we keep the ratio between $\Delta t$ and $h$ constant.
We set the same approximation degree for both space and time, denoted by~$k$.
We compute the errors 
using the following space--time discretizations:
$$
(\texttt{mesh1},\texttt{inter1}),\qquad(\texttt{mesh2},\texttt{inter2}),\qquad (\texttt{mesh3},\texttt{inter3}),\qquad\text{and}\qquad(\texttt{mesh4},\texttt{inter4}),
$$
where \texttt{mesh*} refer to one type of mesh.

In Figure~\ref{fig:convH}, we report the convergence curves corresponding 
to each mesh type for polynomial degrees 1, 2, and 3.
For all space–time approximation degrees, the convergence rates are as expected.
Moreover, the curves remain close to each other across different mesh geometries, 
highlighting the robustness of the proposed method with respect to the element shape.
Although the theoretical estimate is not explicitly provided in this paper,
we also compute the $H^1$ seminorm of the error at the final time.
This error indicator exhibits the expected convergence rate and,
as further evidence of the robustness with respect to the element shape,
the convergence curves corresponding to the same approximation degree are very close to each other.

\begin{figure}[!htb]
     \centering
    \begin{tabular}{cc}
    \includegraphics[height=0.41\textwidth]{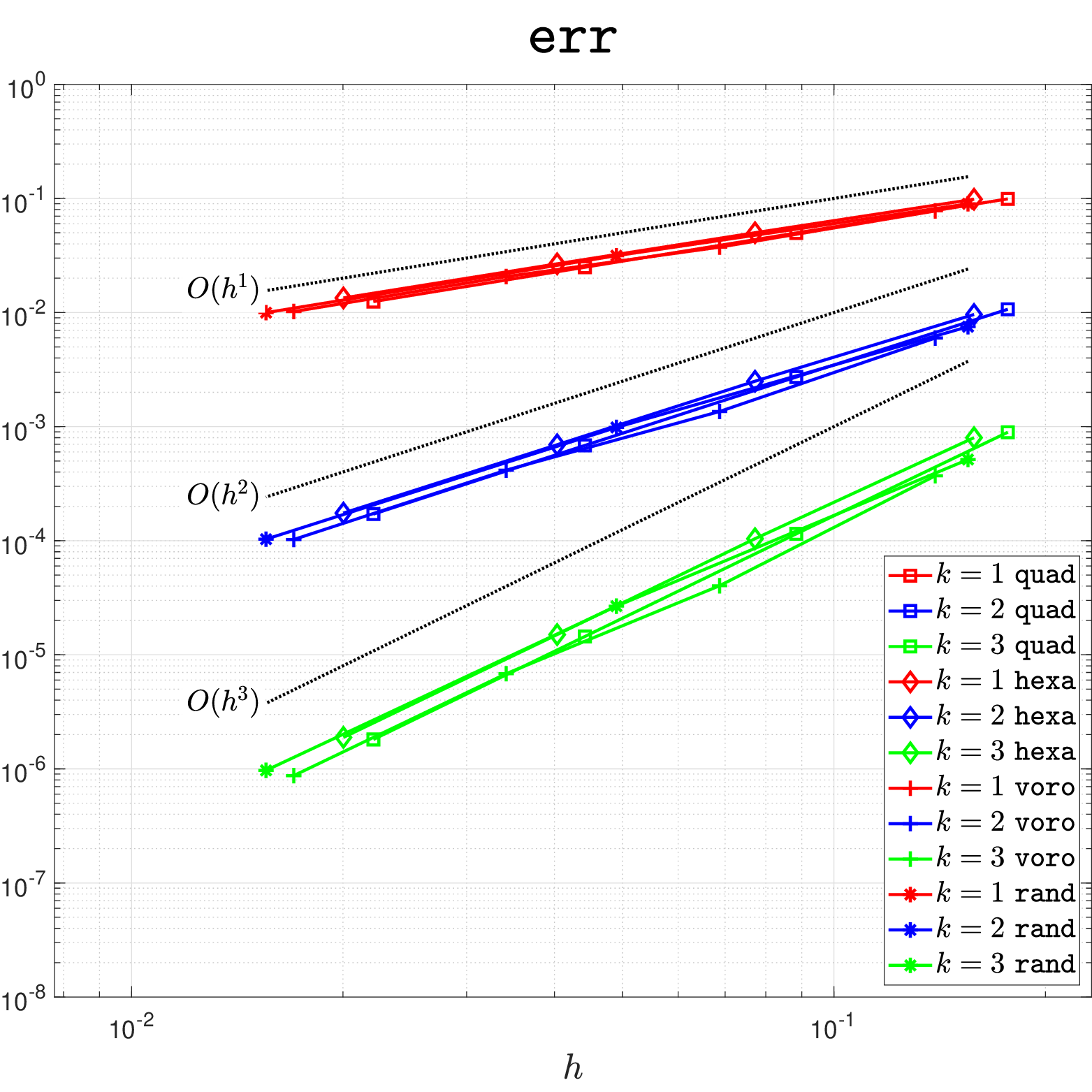} &
    \includegraphics[height=0.41\textwidth]{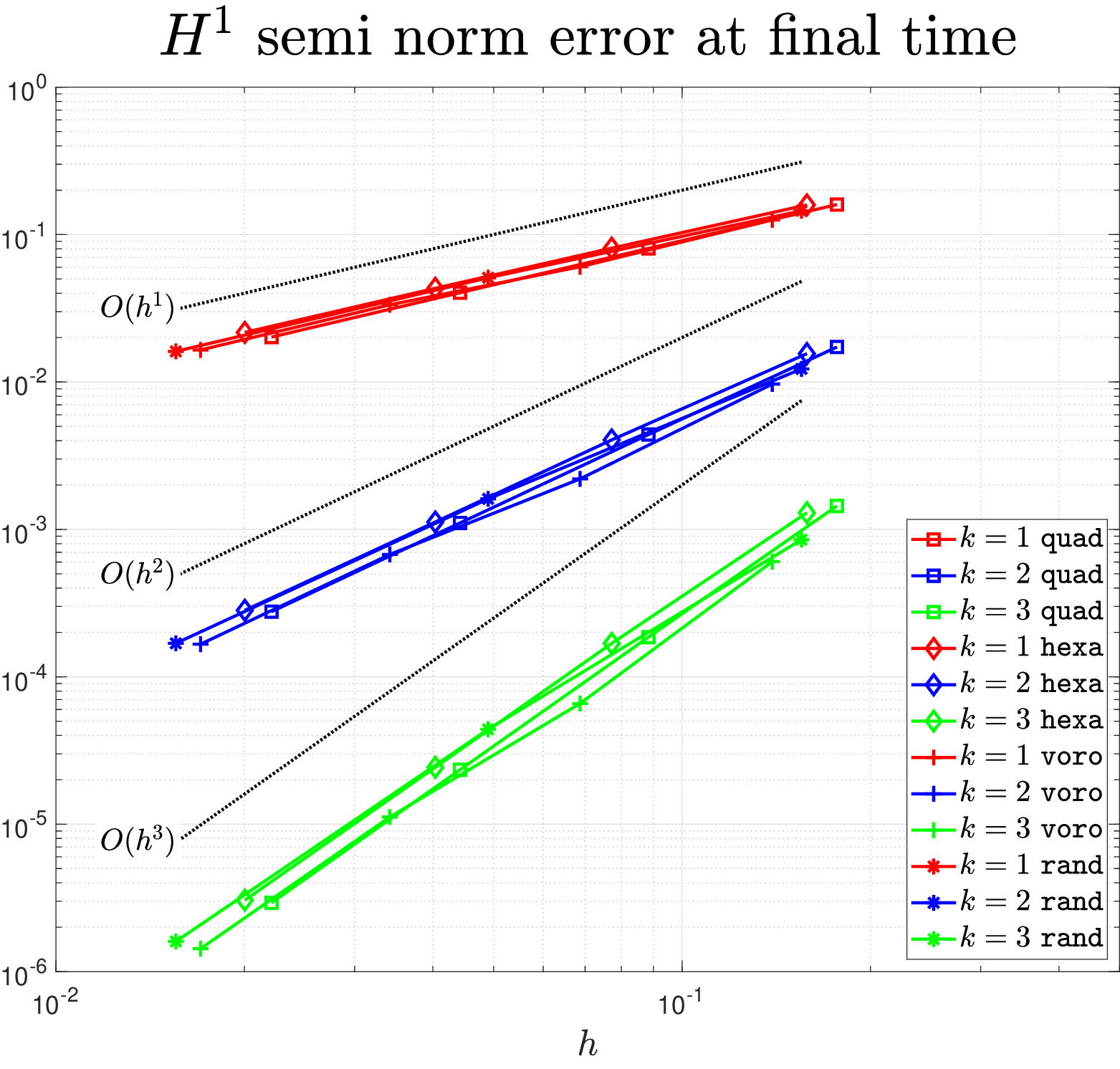}
    \end{tabular}
    \caption{Convergence analysis: convergence curves for different mesh types and polynomial degrees. 
    The error indicator \texttt{err} is shown on the left, and the $H^1$ seminorm of the error on the right.}
    \label{fig:convH}
\end{figure}

\subsubsection{\texorpdfstring{$k$}{k}-convergence}

In this numerical experiment, we analyze the convergence of the proposed scheme 
with respect to the space–time approximation degree.
To this end, we consider the coarsest space–time discretization, (\texttt{mesh1}, \texttt{inter1}), 
for each mesh type, 
and compute both the error indicator \texttt{err} and the $H^1$ seminorm error at the final time.

In Figure~\ref{fig:convK}, we report the convergence curves.
At each refinement in the polynomial degree $k$, 
we observe an improvement of approximately one order of magnitude in the error,
which is in agreement with the estimate given in Equation~\eqref{ferror}.

\begin{figure}[!htb]
     \centering
    \begin{tabular}{cc}
    \includegraphics[width=0.41\textwidth]{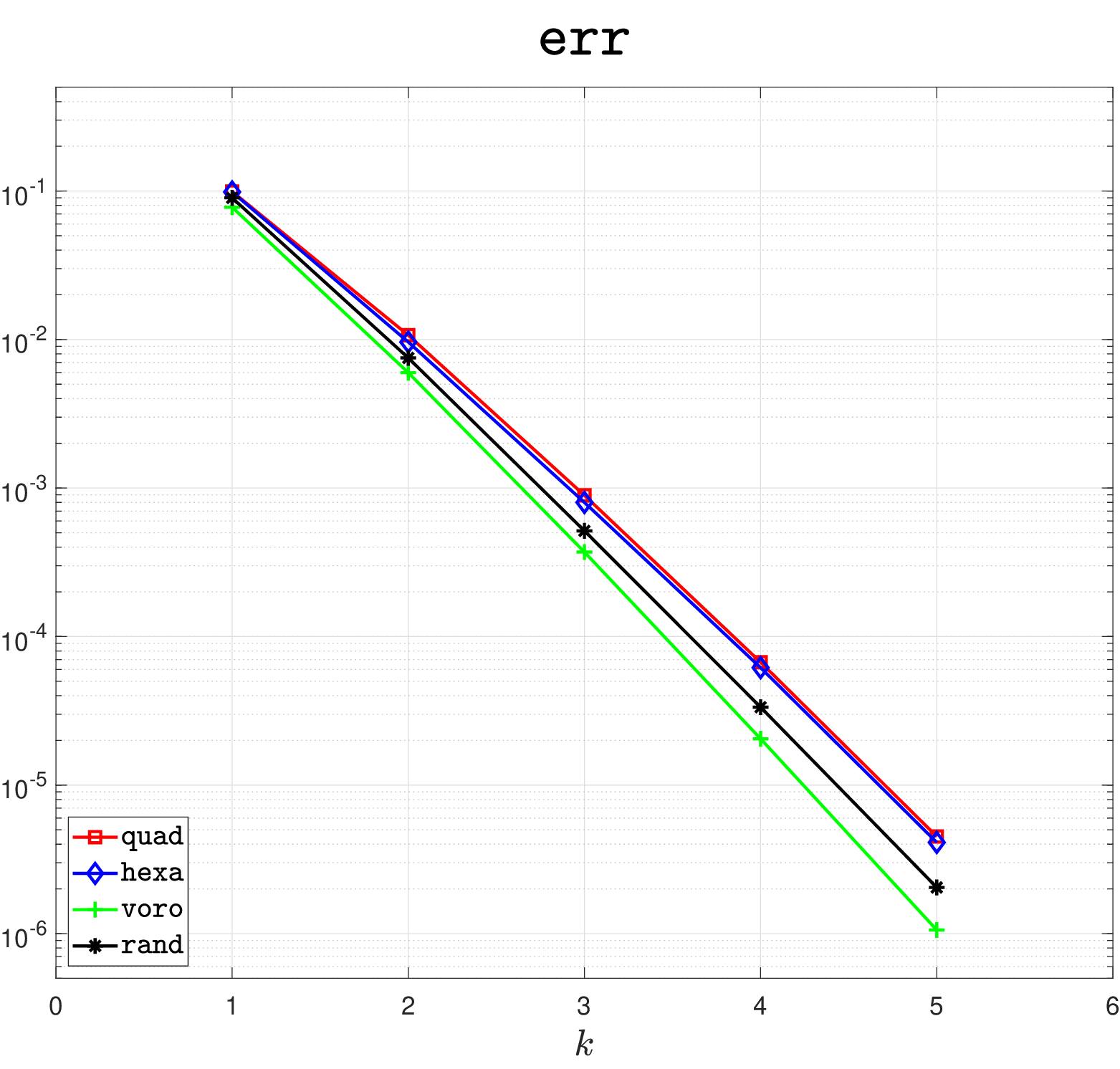} &
    \includegraphics[width=0.41\textwidth]{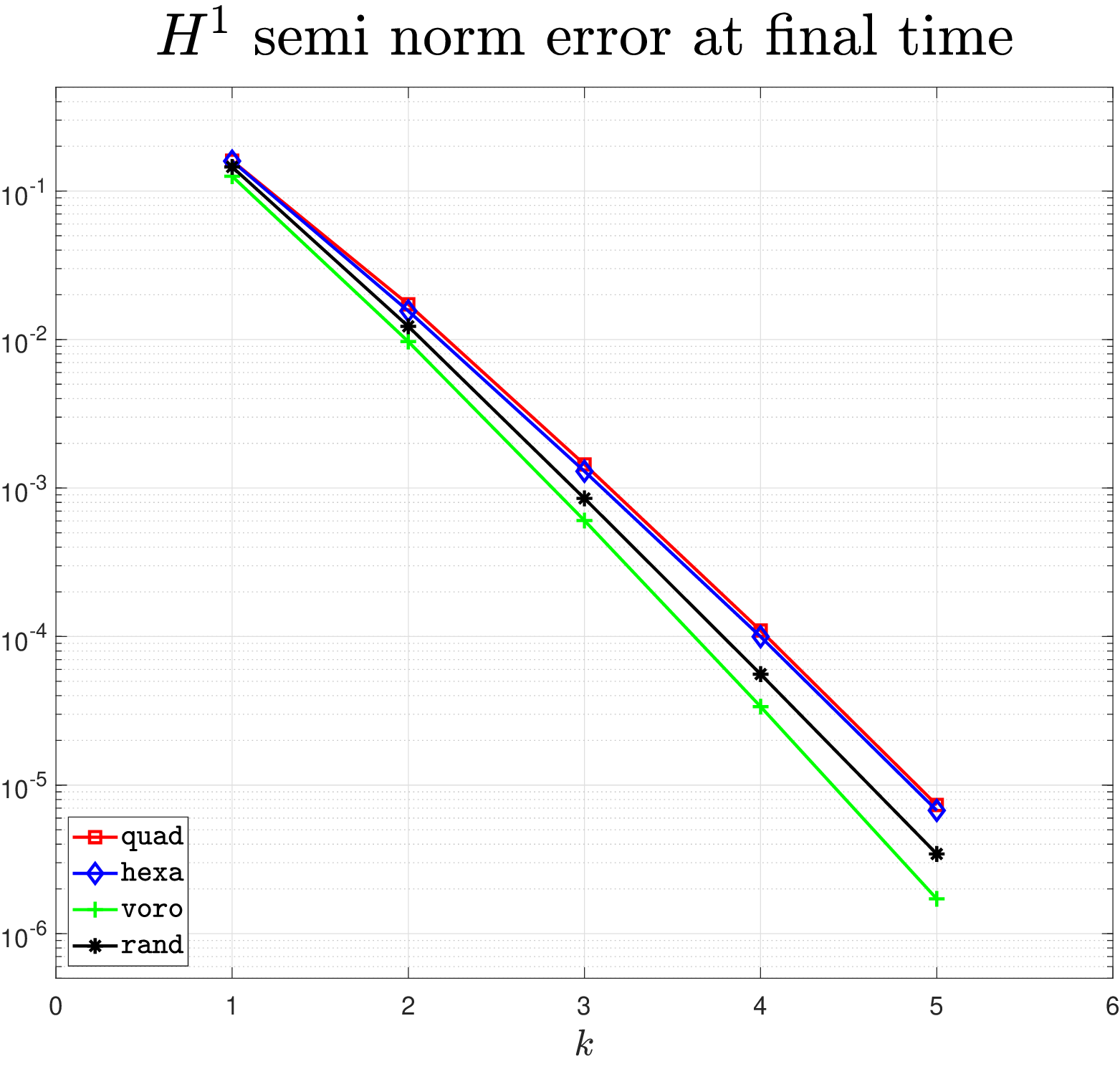}
    \end{tabular}
    \caption{$k-$convergence:
    The error indicator \texttt{err} (left) and the $H^1$ seminorm error~(right) varying mesh types.}
    \label{fig:convK}
\end{figure}

\subsubsection{Robustness with respect to the coefficient \texorpdfstring{\boldmath $D$}{D}}

In this section, we analyze the robustness of the proposed method with respect to the diffusion coefficient $D$.
We consider the same setting as before, 
but we now solve the model problem for different values of $D$:
$$
D = [10^{0},\,10^{-1},\,10^{-2},\,10^{-3},\,10^{-4},\,10^{-5},\,10^{-6},\,10^{-7}]\,,
$$
modifying the data $\tilde{c}$ and $c_I$ accordingly.
For each value of $D$, 
we solve the problem using all mesh types at the second refinement level in both space and time, i.e., 
(\texttt{mesh2}, \texttt{inter2}).

In Figure~\ref{fig:robD}, 
we observe that, while keeping the space–time discretization fixed and varying $D$, 
the error remains nearly constant.
This provides numerical evidence that the proposed method is robust with respect to the diffusion coefficient.

\begin{figure}[!htb]
    \centering
    \includegraphics[width=0.90\textwidth]{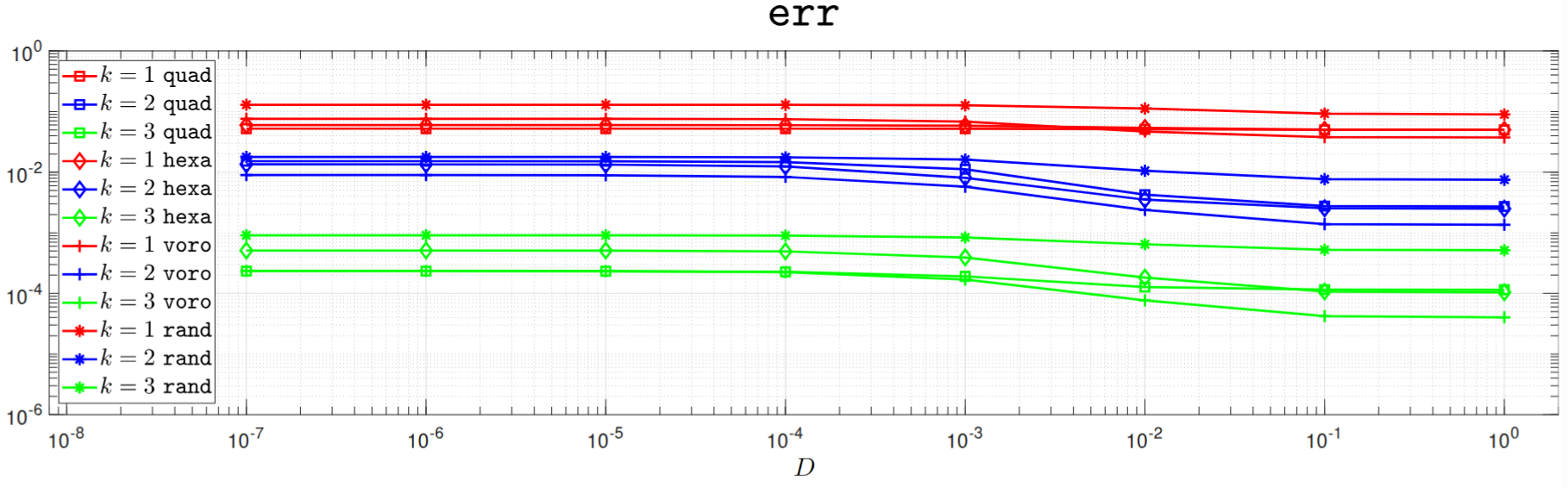}
    \caption{Robustness with respect to the coefficient $D$: Value of the error varying the parameter $D$.}
    \label{fig:robD}
\end{figure}

\subsection{Multi--species example}\label{sec:num:appl}

In this section, we apply the proposed method to a multi--species transport problem. 
The test case is inspired by the \emph{Porous Catalytic Reactor with Injection Needle} model, 
available in the open-source COMSOL Multiphysics~\cite{Comsol:2024:CMS} software and 
described in detail in~\cite{Comsol:2024:PCR}.

The reactor is a tubular structure, illustrated in Figure~\ref{fig:eseTunn}.
The lengths of its parts are proportional to a unit length denoted by~\say{u}.
The first species, \(c_1\), is injected from the \corr{left} side of the tunnel,
while the second species, \(c_2\), enters through a tubular inlet positioned on the top.
The reactor exhibits a homogeneous distribution of porosity,
we set \(\mu K^{-1} = \mathbb{I}\),
The diffusion coefficients are set to~\(D_1 = D_2 = 0.01\), 
and the final simulation time is~\(T = 20\). 
For the Darcy problem, we prescribe a constant unit flux 
at both the inflow and outflow sections of the tube, as illustrated in Figure~\ref{fig:eseTunn},
and impose a zero normal flux condition on the remaining boundaries.
For the transport equations, the inflow concentrations of \(c_1\) and \(c_2\) are defined as
\[
c_1^I(x,y) = y(y-4), \qquad c_2^I(x,y) = (x-2)(4-x)\,,
\]
respectively.

\begin{figure}[!htb]
\centering
\begin{overpic}[abs,unit=1mm,width=0.95\textwidth]{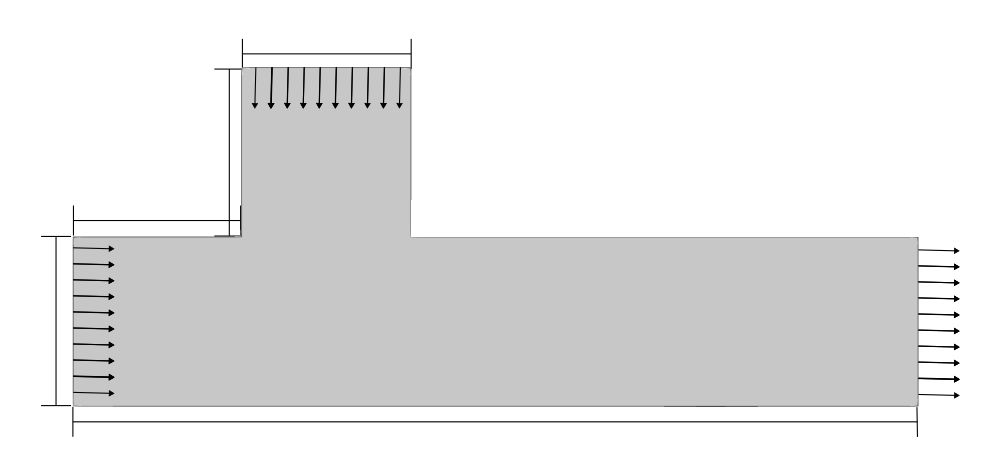}
\put(21,23){{\large \(c_1\)}}
\put(50,52){{\large \(c_2\)}}
\put(21,40){{\large 2u}}
\put(50,66){{\large 2u}}
\put(75,3){{\large 10u}}
\put(3,23){{\large 2u}}
\put(30,50){{\large 2u}}
\end{overpic}
\caption{Multi-species example: reactor tunnel}
\label{fig:eseTunn}
\end{figure}

The computational domain is discretized using a Delaunay triangulation, 
and the time interval is divided into 100 equally spaced sub-intervals.
Both the spatial and temporal polynomial degrees are set to one.

Figure~\ref{fig:flow} shows the flow field obtained from the solution of the Darcy problem,
which is consistent with the expected behavior based on the problem data.

\begin{figure}[!htb]
\centering
\includegraphics[width=0.85\linewidth]{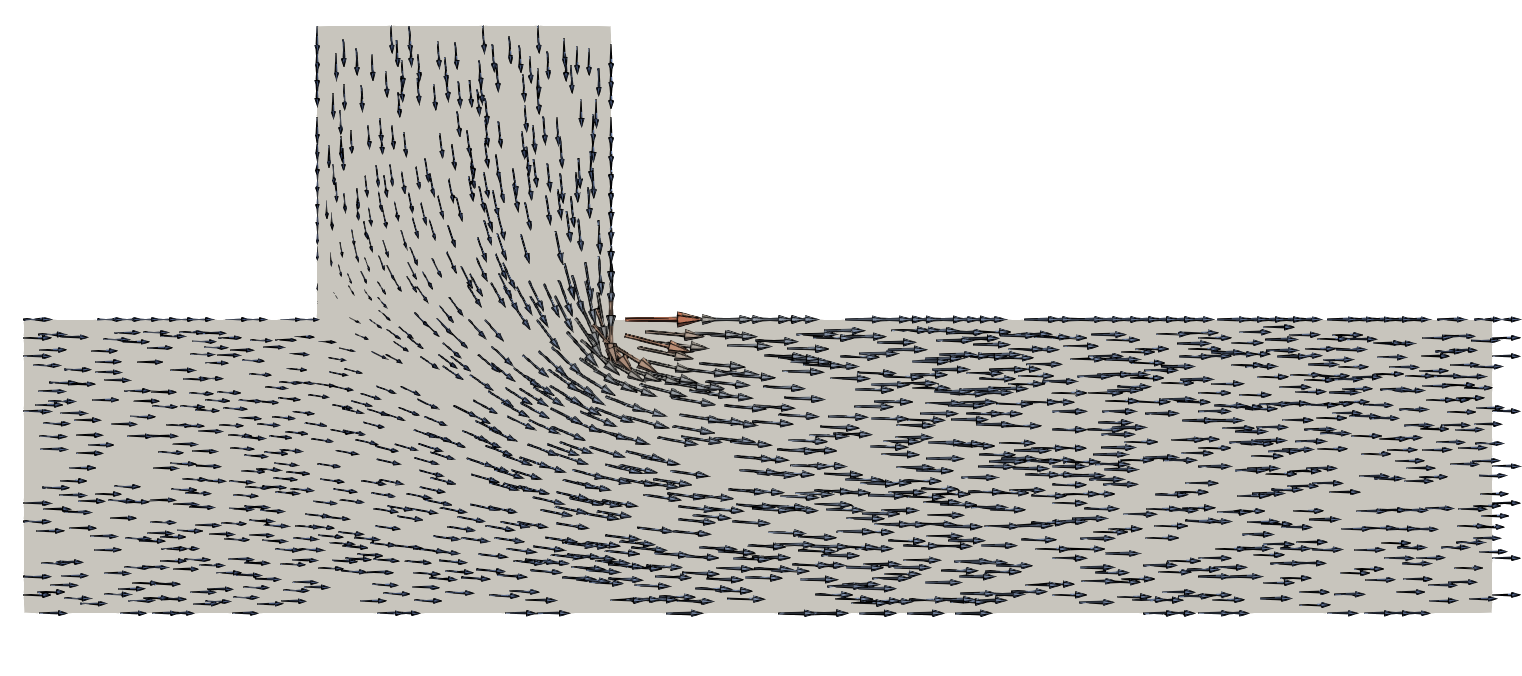}
\caption{Multi-species example: flow field computed from the resolution of the Darcy problem.}
\label{fig:flow}
\end{figure}

To analyze the impact of the parameters in the first-order reaction term,
we fix the degradation rates of the two species to \(\gamma_1 = \gamma_2 = 0.2\).
We then consider three different reaction scenarios:
\begin{itemize}
\item \texttt{noReaction}: in this case, the two species evolve independently,
and the degradation of one species does not generate the other.
This behavior is obtained by setting \(y_{1/2} = y_{2/1} = 0.0\);
\item \texttt{lowReaction}: in this case, a small amount of \(c_2\) 
is produced from the degradation of \(c_1\), and vice versa,
by setting \(y_{1/2} = y_{2/1} = 0.1\);
\item \texttt{highReaction}: here, we consider a non-symmetric case, 
with a high production of \(c_2\) from \(c_1\),
and a low production of \(c_1\) from the degradation of \(c_2\);
specifically, \(y_{1/2} = 1.0\) and \(y_{2/1} = 0.1\).
\end{itemize}

Figures~\ref{fig:speces1} and~\ref{fig:speces2} show selected time steps of the simulations
for the three considered scenarios.
The results are consistent with the chosen values of the parameters \(y_{*/\circ}\).
In both the \texttt{lowReaction} and \texttt{highReaction} cases,
a noticeable presence of \(c_2\) is observed along the flow path of \(c_1\), and vice versa.
In the \texttt{lowReaction} case, this interaction is roughly symmetric,
whereas in the \texttt{highReaction} case the degradation of \(c_1\) produces
a significantly higher concentration of \(c_2\).
As expected, in the \texttt{noReaction} case no such production is observed for either species.

\begin{figure}[!htb]
    \centering
    \begin{tabular}{ccc}
        & \(T=1.\) & \(T=2.\)\\ 
        \rotatebox{90}{\texttt{noReaction}} &
        \includegraphics[width=0.42\linewidth]{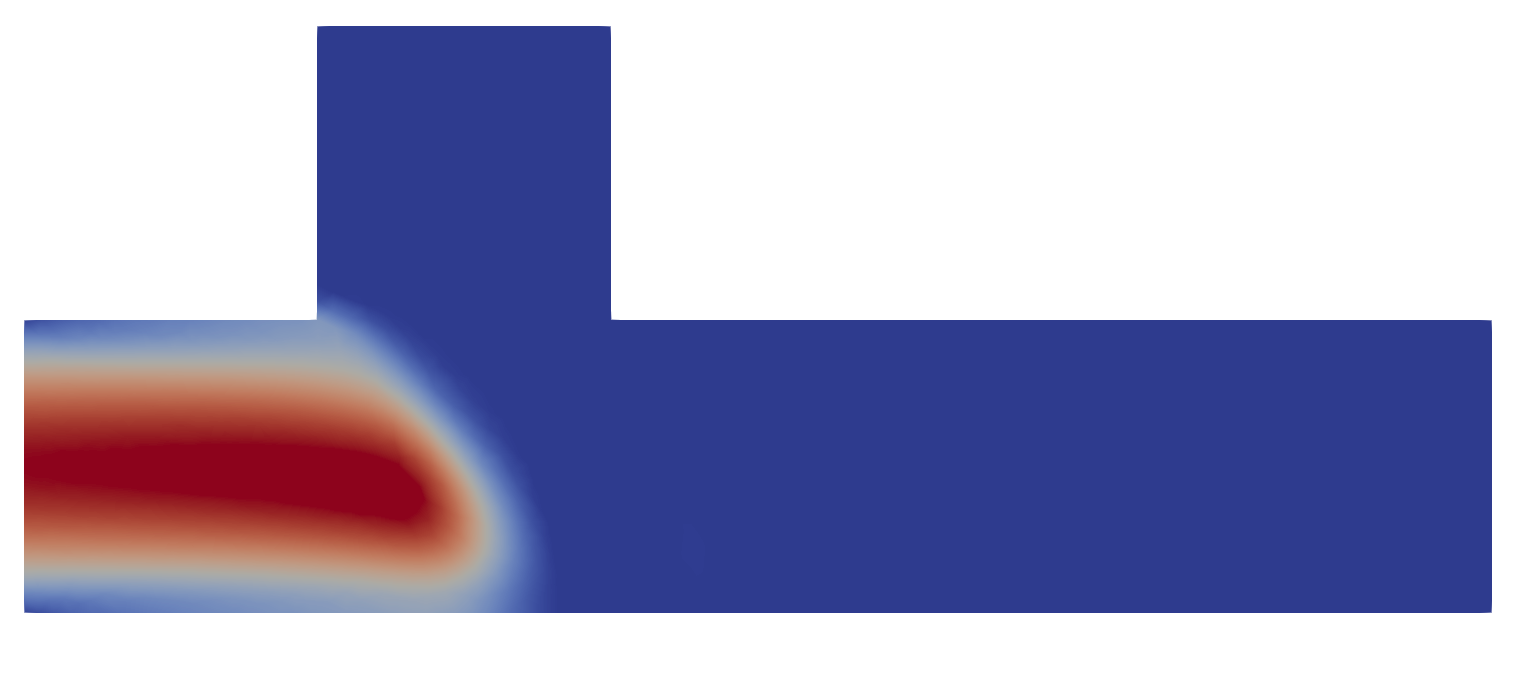} &
        \includegraphics[width=0.42\linewidth]{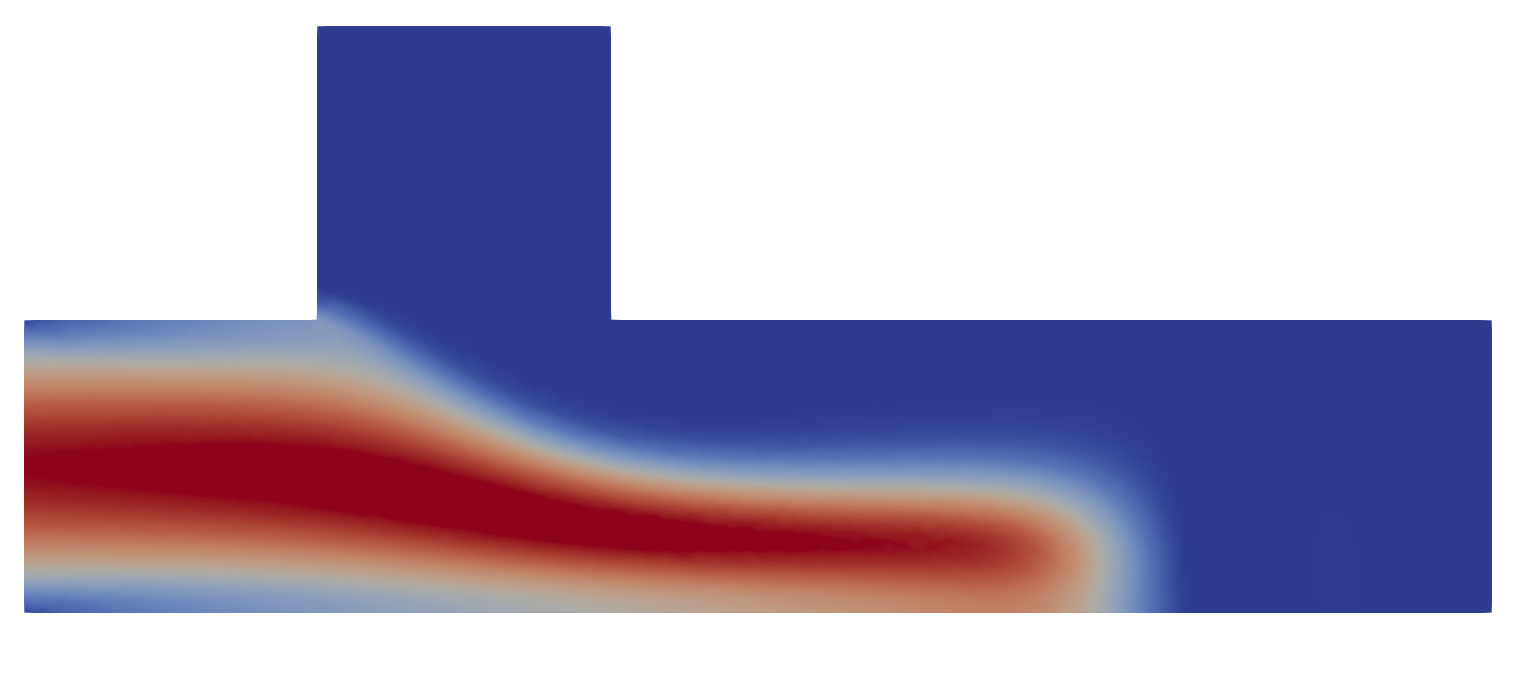}\\[1ex]
        
        \rotatebox{90}{\texttt{lowReaction}} &
        \includegraphics[width=0.42\linewidth]{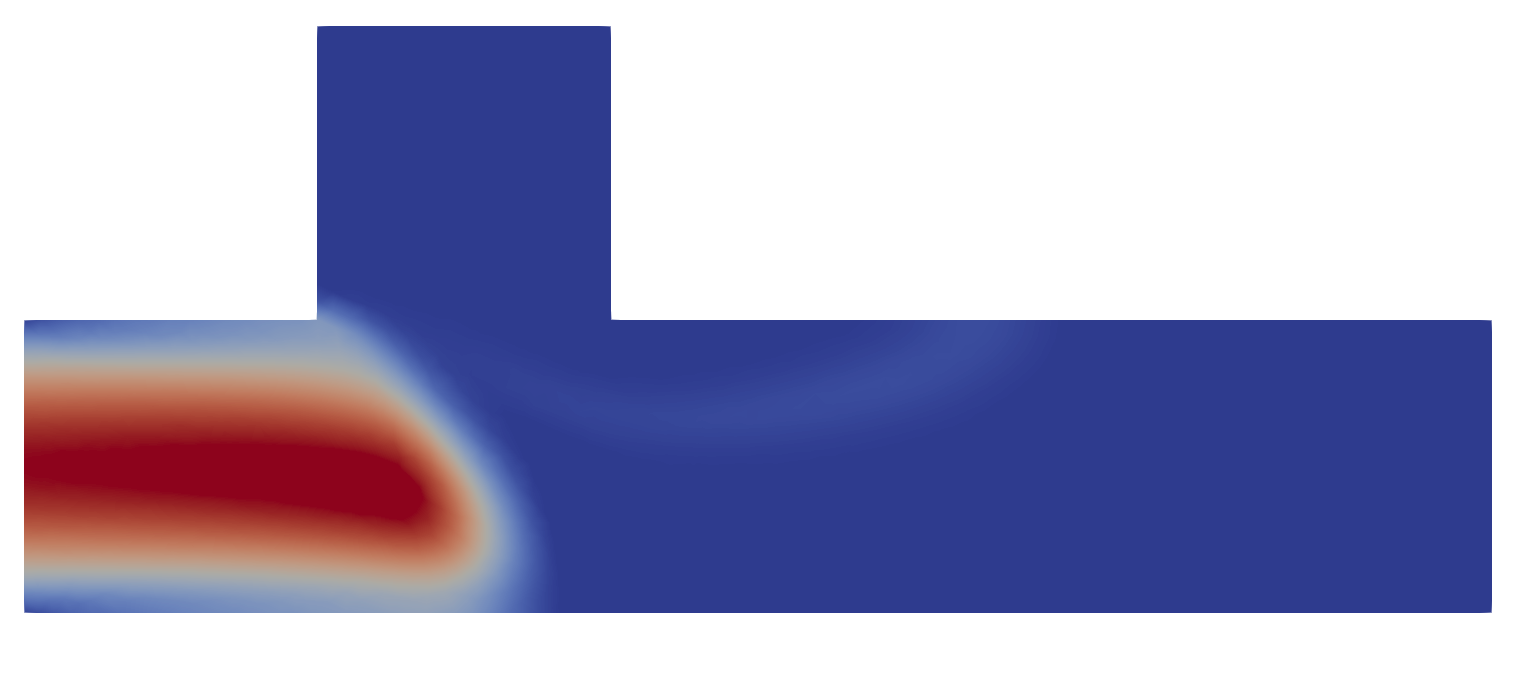} &
        \includegraphics[width=0.42\linewidth]{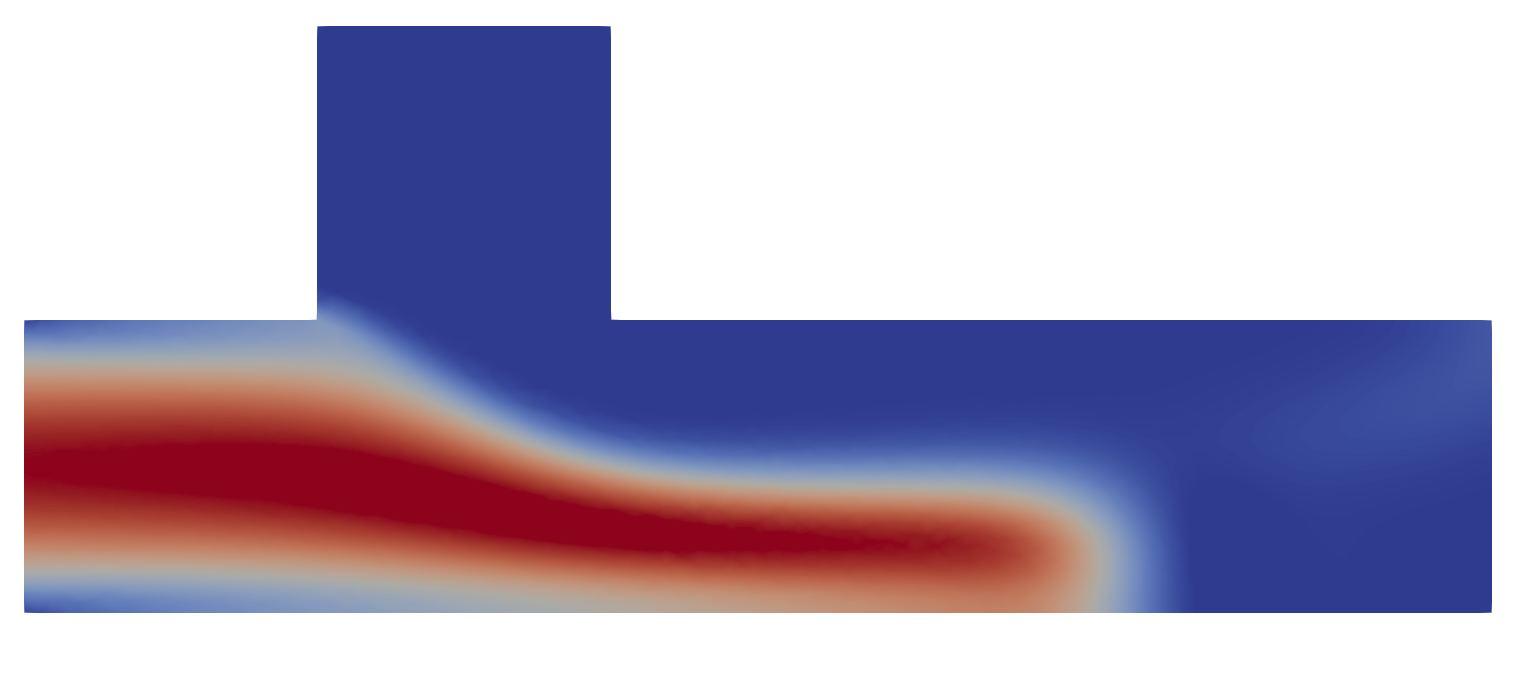}\\[1ex]
        
        \rotatebox{90}{\texttt{highReaction}} &
        \includegraphics[width=0.42\linewidth]{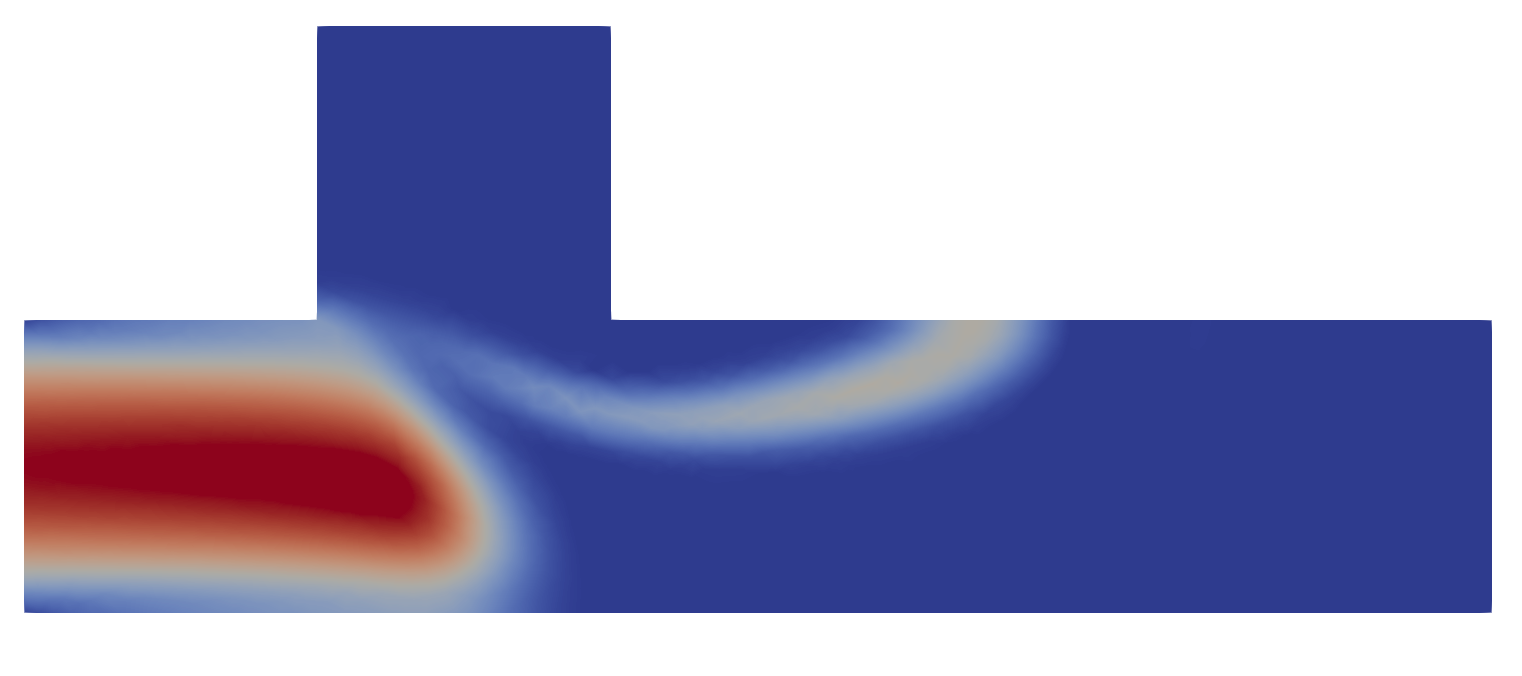} &
        \includegraphics[width=0.42\linewidth]{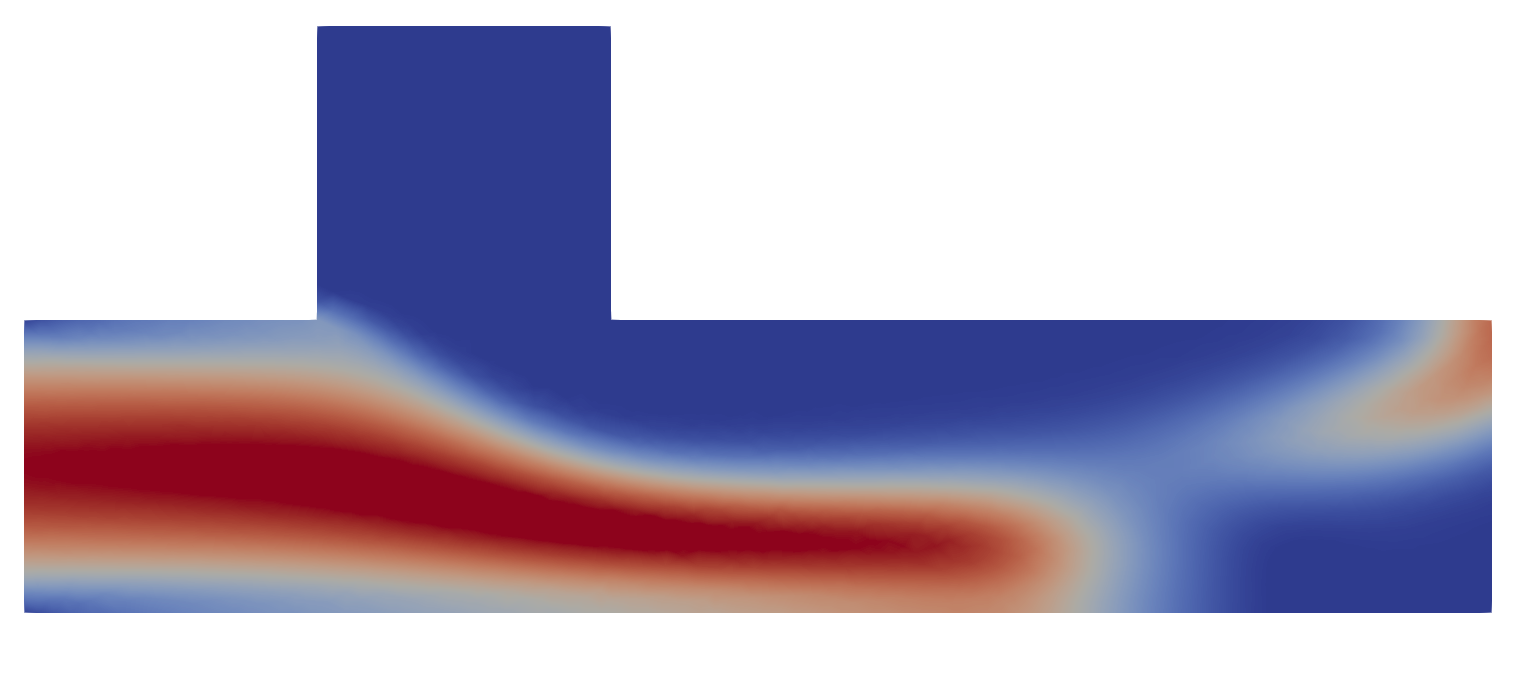}\\[1ex]
        
        & \multicolumn{2}{c}{\includegraphics[width=0.4\textwidth]{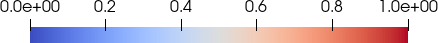}}\\
    \end{tabular}
    \caption{Multi--species example: evolution of~\(c_1\) considering different scenarios.}
    \label{fig:speces1}
\end{figure}
\begin{figure}[!htb]
    \centering
    \begin{tabular}{ccc}
        & \(T=1.\) & \(T=2.\)\\ 
        \rotatebox{90}{\texttt{noReaction}} &
        \includegraphics[width=0.42\linewidth]{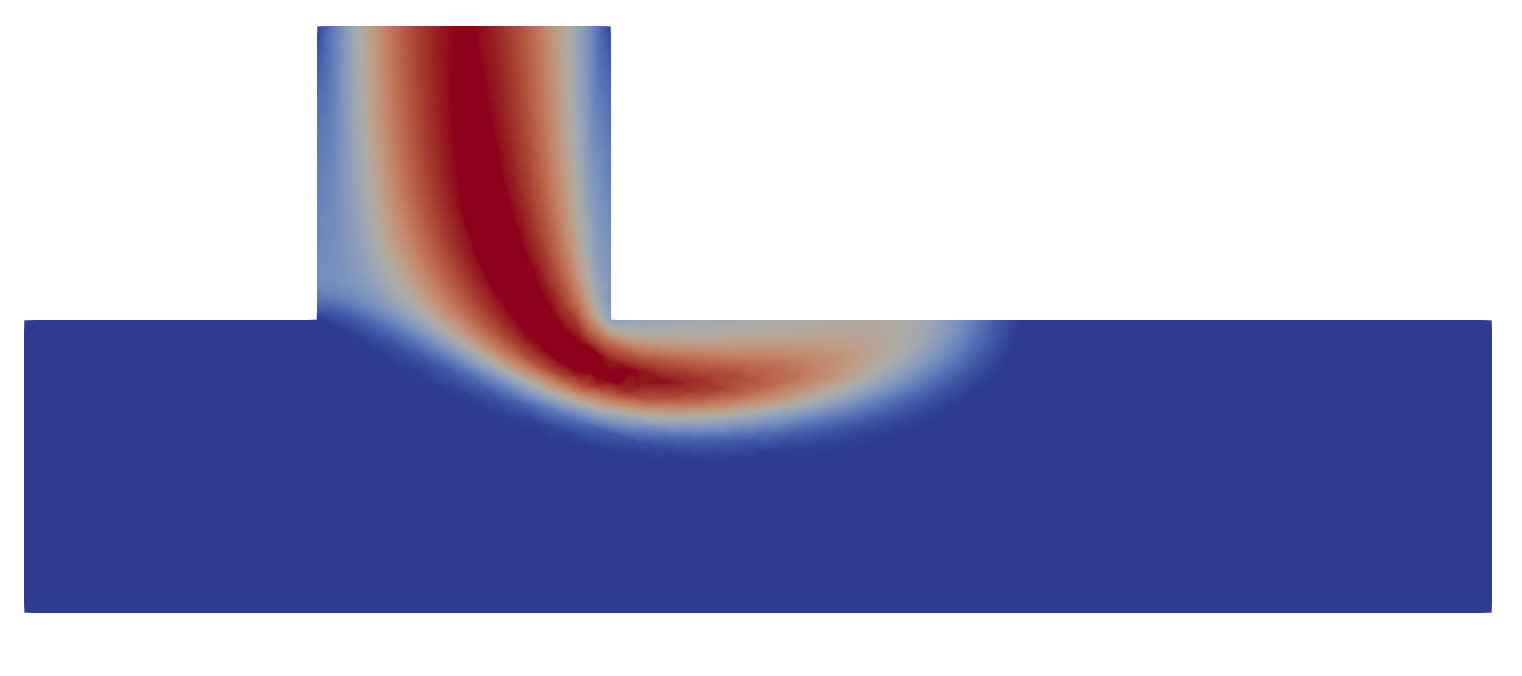} &
        \includegraphics[width=0.42\linewidth]{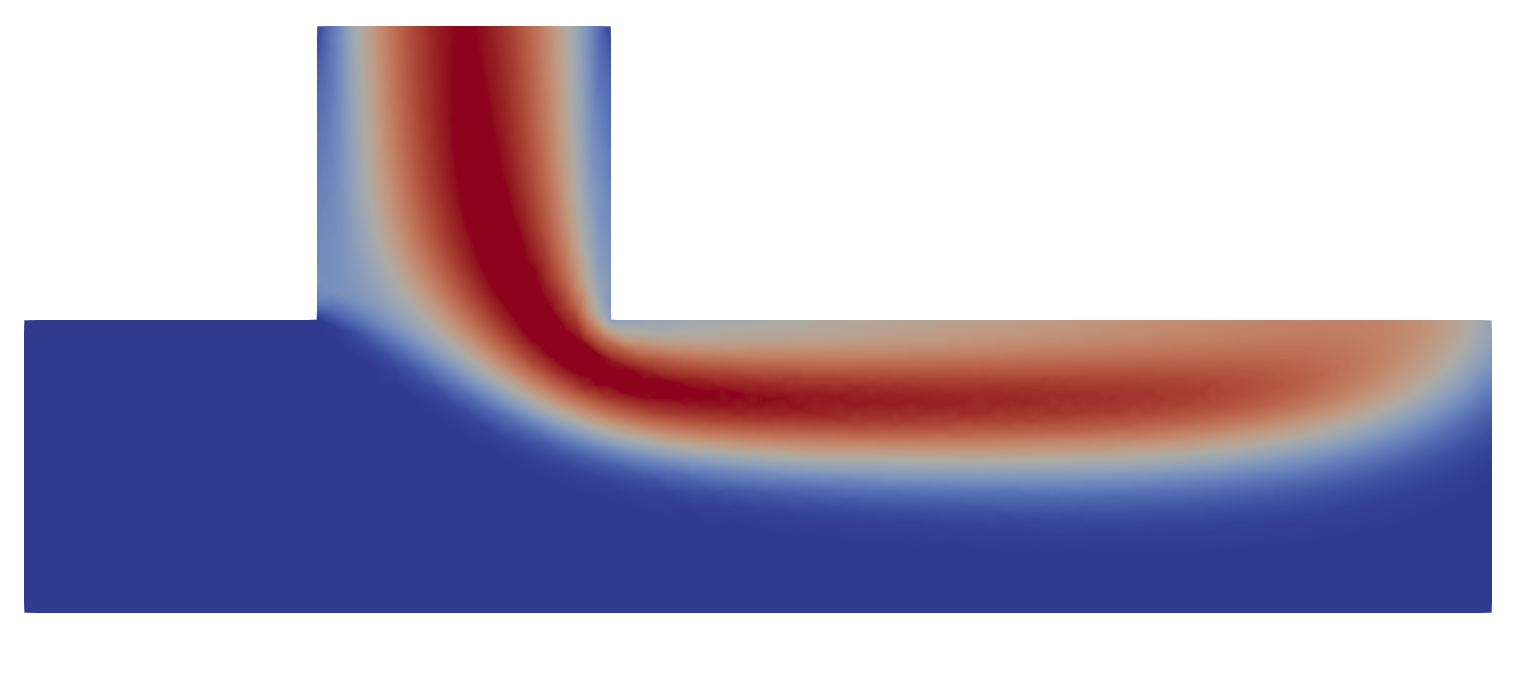}\\[1ex]
        
        \rotatebox{90}{\texttt{lowReaction}} &
        \includegraphics[width=0.42\linewidth]{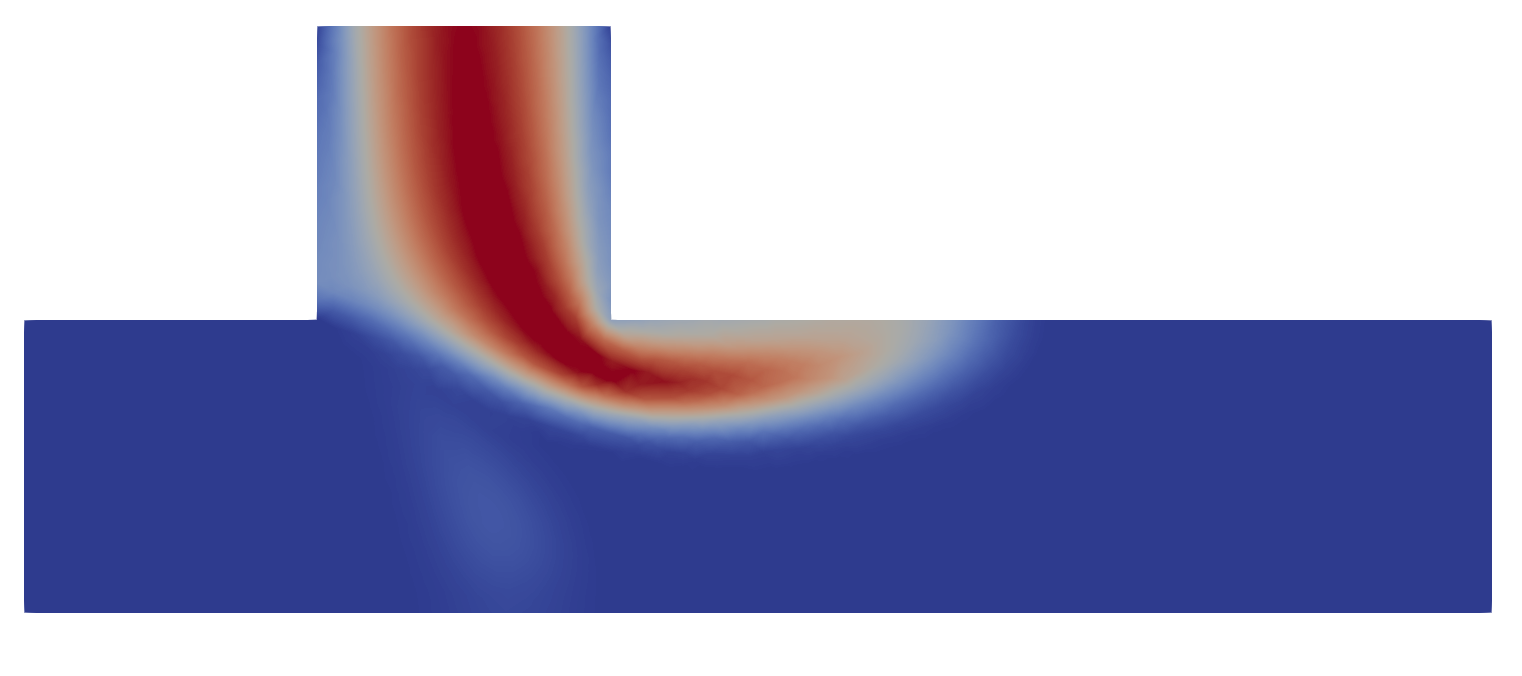} &
        \includegraphics[width=0.42\linewidth]{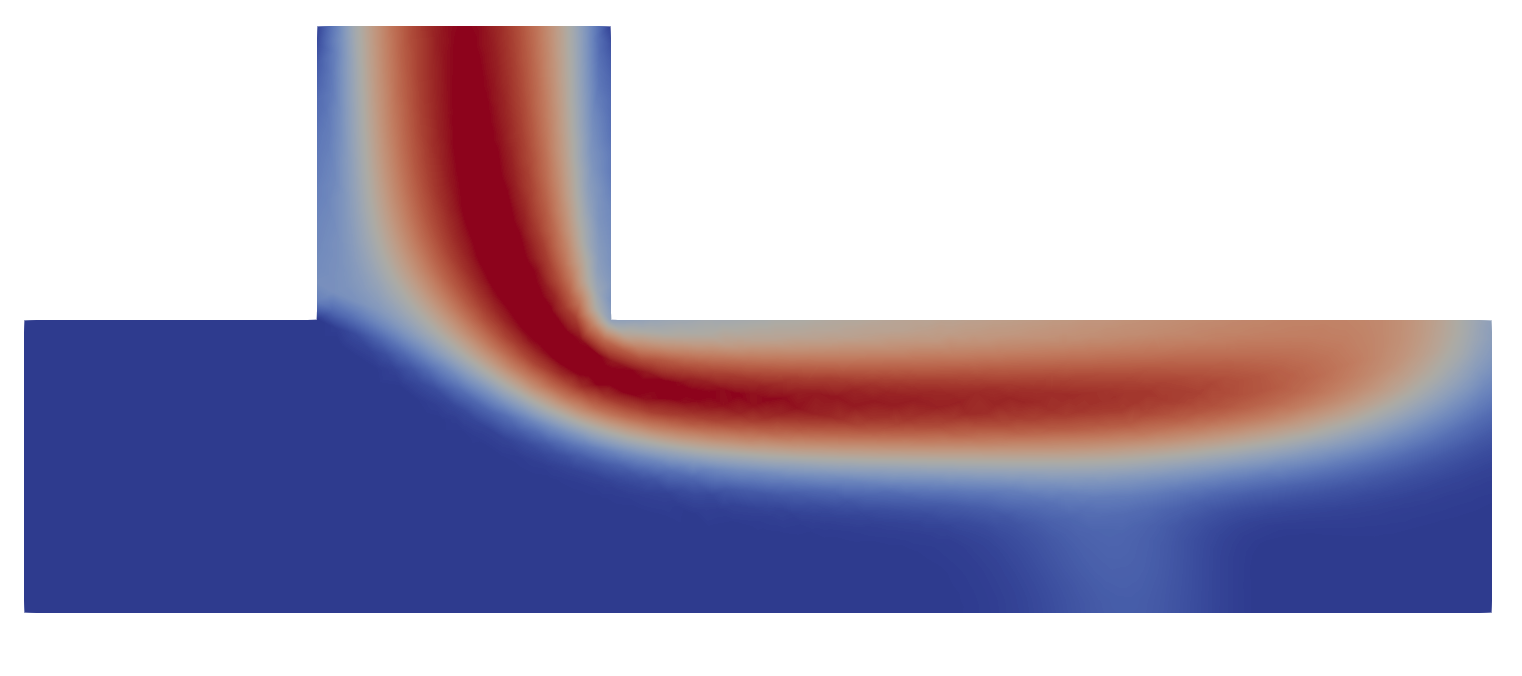}\\[1ex]
        
        \rotatebox{90}{\texttt{highReaction}} &
        \includegraphics[width=0.42\linewidth]{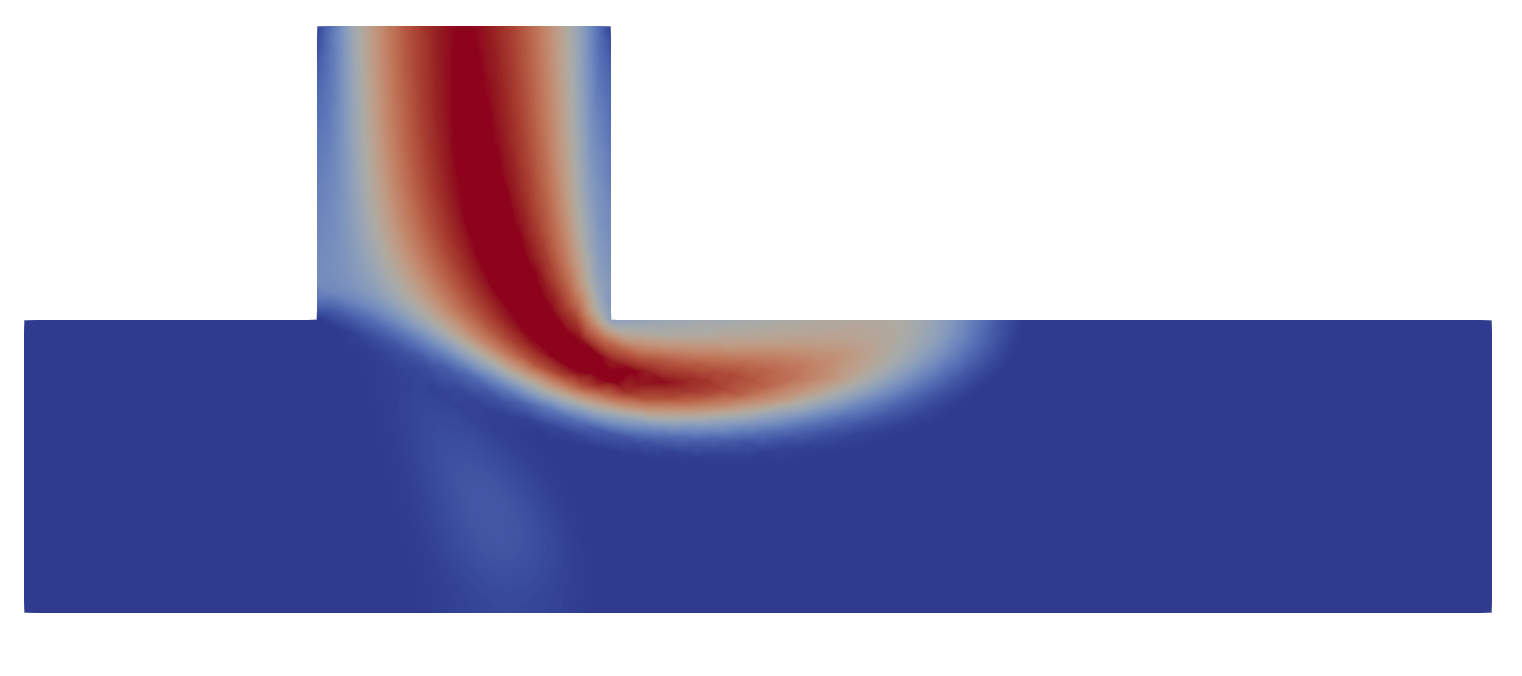} &
        \includegraphics[width=0.42\linewidth]{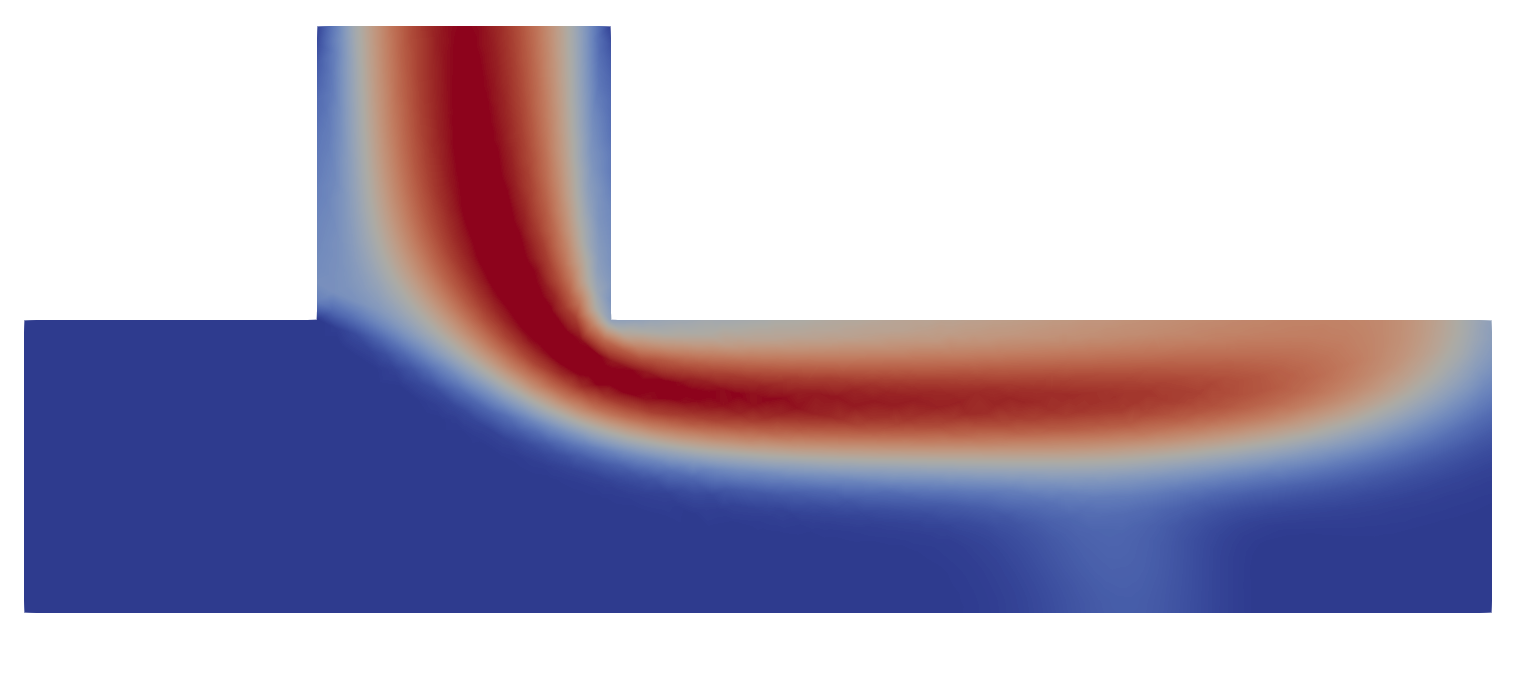}\\[1ex]
        
        & \multicolumn{2}{c}{\includegraphics[width=0.4\textwidth]{imm/legend.png}}\\
    \end{tabular}
    \caption{Multi--species example: evolution of~\(c_2\) considering different scenarios.}
    \label{fig:speces2}
\end{figure}

%%%%%%%%%%%%%%%%%%%%%%%%%%%%%%%%%%% english corrected!

\section*{Acknowledgments}% english checked!
All authors were partially supported by the European Union (ERC Synergy, NEMESIS, project number 101115663) and 
the Italian Ministry of University and Research (MUR) through 
the PRIN 2022 project \say{FREYA – Fault REactivation: a hYbrid numerical Approach}. 
Views and opinions \corr{expressed are, however, those} of the authors only and 
do not necessarily reflect those of the EU or the ERC Executive Agency.
The authors would also like to thank Prof.~L. Beir\~ao da Veiga and 
the anonymous referees for their valuable comments and suggestions on the manuscript. 
FD is also a member of the INdAM-GNCS group.

%% refernces 
\bibliographystyle{plain}
\bibliography{references}

\end{document}